\documentclass[intlim,righttag,10pt]{amsart}
\usepackage{amscd}
\usepackage{amssymb}
\usepackage[all]{xy}
\usepackage{enumitem}
\RequirePackage{mathrsfs}
\def\jcdot{{\scriptscriptstyle\bullet}}

\def\invlim{\mathop{\vtop{\ialign{##\crcr$\hfill{\lim}\hfil$\crcr
\noalign{\kern1pt\nointerlineskip}\leftarrowfill\crcr\noalign
{\kern -3pt}}}}\limits}
\def\dirlim{\mathop{\vtop{\ialign{##\crcr$\hfill{\lim}\hfil$\crcr
\noalign{\kern1pt\nointerlineskip}\rightarrowfill\crcr\noalign
{\kern -3pt}}}}\limits}
\def\lomapr#1{\smash{\mathop{\relbar\joinrel\longrightarrow}\limits^{#1}}}

\def\epsilon{\varepsilon}
\let\mathcal\mathscr

\usepackage[textsize=tiny]{todonotes}

\newtheorem{theorem}[equation]{Theorem}
 \newtheorem{lemma}[equation]{Lemma}
 \newtheorem{proposition}[equation]{Proposition}
 \newtheorem{corollary}[equation]{Corollary}

\theoremstyle{definition}
\newtheorem{definition}[equation]{Definition}
\theoremstyle{remark}
\newtheorem{remark}[equation]{Remark}

\newtheorem{example}[equation]{Example}
\newtheorem*{acknowledgments}{Acknowledgments} 

\usepackage[T1]{fontenc}

\let\emptyset\varnothing

\let\cal\mathcal

\def\phi{\varphi}

\def\Q{{\bf Q}} \def\Z{{\bf Z}}
 \def\qp{\Q_p}

\def\N{{\bf N}}
\def\O{{\cal O}}
\def\G{{\cal G}} 

\def\dual{{\boldsymbol *}}

\def\epsilon{\varepsilon}

 \def\A{{\bf A}} \def\B{{\bf B}}

\def\rg{{\rm R}\Gamma}

\newcommand{\R}{\mathrm {R} }

\newcommand{\LL}{\mathrm {L} }

\newcommand{\ovk}{\overline{K} }

  \newcommand{\colim}{\operatorname{colim} }

  \newcommand{\proet}{\operatorname{pro\acute{e}t}  }

 \newcommand{\eet}{\operatorname{\acute{e}t} }
  \newcommand{\proeet}{\operatorname{pro\acute{e}t} }

 \newcommand{\nr}{\operatorname{nr} }

 \newcommand{\Hom}{\operatorname{Hom} }

 \newcommand{\Gal}{\operatorname{Gal} }

\newcommand{\synt}{ \operatorname{syn} }

\newcommand{\st}{\operatorname{st} }

\newcommand{\hk}{\operatorname{HK} }
\newcommand{\dr}{\operatorname{dR} }
\newcommand{\pdr}{\operatorname{pdR} }

 \newcommand{\crr}{\operatorname{cr} }
 \newcommand{\gr}{\operatorname{gr} }

 \newcommand{\sff}{{\mathcal{F}}}
 
 \newcommand{\sy}{{\mathcal{Y}}}
 \newcommand{\sh}{{\mathcal{H}}}
 \newcommand{\sg}{{\mathcal{G}}}

 \newcommand{\sbb}{{\mathbb{B}}}

 \newcommand{\so}{{\mathcal O}}

 \newcommand{\sx}{{\mathcal{X}}}
 
\newcommand{\sd}{{\mathcal{D}}}
\newcommand{\sm}{{\mathcal{M}}}

 \newcommand{\wt}{\widetilde}
 \newcommand{\wh}{\widehat}

   \def\B{{\bf B}}
      \def\A{{\bf A}}
\def\jcdot{{\scriptscriptstyle\bullet}}

\def\invlim{\mathop{\vtop{\ialign{##\crcr$\hfill{\lim}\hfil$\crcr
\noalign{\kern1pt\nointerlineskip}\leftarrowfill\crcr\noalign
{\kern -3pt}}}}\limits}
\def\dirlim{\mathop{\vtop{\ialign{##\crcr$\hfill{\lim}\hfil$\crcr
\noalign{\kern1pt\nointerlineskip}\rightarrowfill\crcr\noalign
{\kern -3pt}}}}\limits}
\def\lomapr#1{\smash{\mathop{\relbar\joinrel\longrightarrow}\limits^{#1}}}

\def\epsilon{\varepsilon}
\let\mathcal\mathscr

\numberwithin{equation}{section}

\begin{document}
\title[On de Rham flip-flopping  in  dual towers]
{On de Rham flip-flopping in  dual towers}
\author{Gabriel Dospinescu}
\address{UMR 6620 du CNRS-Laboratoire de Math\'ematiques Blaise Pascal-Universit\'e de Clermont-Auvergne,
Campus des C\'ezeaux 3, place Vasarely, 63178 Aubi\`ere cedex, France and Institute of Mathematics "Simion
Stoilow" of the Romanian Academy, 21 Calea Grivitei Street, 010702 Bucharest, Romania.}
\email{gabriel.dospinescu@uca.fr}
\author{Wies{\l}awa Nizio{\l}}
\address{CNRS, IMJ-PRG, Sorbonne Universit\'e, 4 place Jussieu, 75005 Paris, France}
\email{wieslawa.niziol@imj-prg.fr}
 \thanks{G.D. and W.N.'s research was partially  supported by the  project  ANR-19-CE40-0015-02 COLOSS. W.N.'s research was also supported in part  by the  Simons Collaboration on Perfection in Algebra, Geometry, and Topology.  G.D.'s research was also supported by the project "Group schemes, root systems, and related representations" funded by the European Union - NextGenerationEU through Romania's National Recovery and Resilience Plan (PNRR) call no. PNRR-III- C9-2023-I8, Project CF159/31.07.2023, and coordinated by the Ministry of Research, 
Innovation and Digitalization (MCID) of Romania.}
\begin{abstract}
We prove a version of  de Rham and Hyodo-Kato flip-flopping  for dual towers of rigid analytic spaces including those coming from dual basic local Shimura varieties. The main tools are comparison theorems expressing the two cohomologies as pro-\'etale cohomology of corresponding relative period sheaves that, by definition, satisfy pro-\'etale descent. As an application, we show that compactly supported de Rham and Hyodo-Kato cohomologies of  finite level coverings of the Drinfeld space of  any dimension $d$ over $K$ are admissible as  representations of $\mathbb{GL}_{d+1}(K)$. 
\end{abstract}

\maketitle

\tableofcontents
\section{Introduction} It is a classical result (due to Faltings, Fargues and extended by Scholze, see \cite[Prop.5.4]{Ssurvey}) that, for $\ell\ne p$, the $\ell$-adic cohomology with compact support of the (uncompleted) Lubin-Tate tower is isomorphic to the one for the (uncompleted) Drinfeld tower. This can be deduced more or less formally from the fact that there is a perfectoid space such that the two towers are decompletions of this space. The same argument works for $p$-torsion \'etale cohomology with compact support, but fails very badly for $p$-adic \'etale or $p$-adic pro-\'etale cohomologies.

  One of the key results in \cite{CDN1} was a similar flip-flopping theorem \cite[Th.0.6]{CDN1} for compactly supported de Rham cohomology of the Drinfeld and Lubin-Tate towers in dimension $1$. This result was somewhat surprising since, contrary to $\ell$-adic or $p$-adic \'etale cohomology, de Rham cohomology does not really make sense for perfectoid spaces (meanwhile the situation changed, thanks to the introduction of de Rham stacks, see \cite{RCS}). This comparison theorem was used in loc. cit. to deduce admissibility of compactly supported de Rham cohomology on the Drinfeld side from the one on the Lubin-Tate side, which was known. 
This paper reports on a  generalization of this result to any dimension and then to other dual towers. 

\subsection{Motivation and strategy} \label{duck1}  Let $p$ be a prime number. 
Let $K$ be a finite extension of $\Q_p$ and let $C:=\wh{\overline{\Q}}_p$.  Let $d >1$ and let $D^{\times}$  be the invertible elements in the central division algebra $D$ of invariant $1/d$ over $K$.  Let  ${\cal M}_{\infty}^\varpi$ and ${\rm LT}_{\infty}^\varpi$ be the (uncompleted) Drinfeld and Lubin-Tate spaces at infinity\footnote{See Section \ref{recall1} for precise definitions.} of dimension $d-1$ over $K$. Let $H^{\bullet}_{\dr,c}({\cal M}_{\infty,C}^\varpi)$ be the compactly supported de Rham cohomology of ${\cal M}_{\infty, C}^\varpi$, i.e. the colimit of compactly supported cohomologies at finite levels in the Drinfeld tower. Define similarly the compactly supported Hyodo-Kato cohomology $H^{\bullet}_{\hk,c}({\cal M}_{\infty,C}^\varpi)$. Identical considerations apply to the Lubin-Tate tower.

We prove the following result that motivated our research:
\begin{theorem} {\rm(Drinfeld and Lubin-Tate  flip-flopping)}\label{main0} There are $\mathbb{GL}_d(K)\times D^{\times}$-equivariant
 isomorphisms in solid $C$-modules  and solid $(\phi,N,\sg_{{K}})$-modules\footnote{ We write $C^{\nr}$ for  the maximal unramified extension of $\Q_p$.} over ${C^{\nr}}$, respectively
 $$H^i_{\dr,c}({\cal M}_{\infty,C}^\varpi)\simeq H^i_{\dr,c}({\rm LT}_{\infty,C}^\varpi),\quad H^i_{\hk,c}({\cal M}_{\infty,C}^\varpi)\simeq H^i_{\hk,c}({\rm LT}_{\infty,C}^\varpi).
 $$ Moreover all the above representations are 
  smooth and  admissible already as $\mathbb{GL}_d(K)$-representations.  
\end{theorem}
   As in \cite{CDN1}, the main difficulty in proving Theorem \ref{main0} lies in the fact that the differential forms on the base of the towers do not come, by descent, from differential forms on the completed tower at infinity, and only the completed towers compare naturally as perfectoid spaces. Hence, what objects should represent the differentials at infinity?
 The inspiration came from the arguments used in \cite{CDN1}. Namely, what we proved in loc. cit. could be called 
 "Hodge flip-flopping":  for $d=2$ in the above theorem, we constructed  an isomorphism
 $$
 H^1_{\dr,c}({\cal M}_{\infty,C}^\varpi)\simeq H^1_{\dr,c}({\rm LT}_{\infty,C}^\varpi)
 $$
as the colimit, over $m,n\geq 0$,  of the compositions of isomorphisms:  
 \begin{align*}
  H^1_{\dr,c}({\cal M}_{n,C}^\varpi)^{G_m}\stackrel{\sim}{\leftarrow}H^1_c(\wh{{\cal M}}_{\infty,C}^\varpi,{\so})^{G_m\times\check{G}_n}\simeq 
  H^1_c(\wh{{\rm LT}}_{\infty,C}^\varpi,{\so})^{G_m\times\check{G}_n}\stackrel{\sim}{\leftarrow} H^1_{\dr,c}({\rm LT}_{m,C}^\varpi)^{\check{G}_n}.
 \end{align*}
 Here $G_m, \check{G}_n$ are level $m,n$ congruence subgroups of $G=\mathbb{GL}_d(K)$ and $\check{G}=D^{\times}$, respectively. The middle isomorphism follows from the duality of the Drinfeld and Lubin-Tate towers at infinity, but the other isomorphisms require non-trivial arguments using the $p$-adic period maps, which are not easily extendable to higher dimensions. Hence 
 the sheaf ${\so}$  is the avatar of the Hodge cohomology: it plays the role of  Hodge cohomology on perfectoid spaces\footnote{However, this turned out to be a wrong Hodge avatar
 in higher dimensions. We do prove a version of Hodge flip-flopping but this follows from a filtered de Rham flip-flopping (see Corollary \ref{IAS11}). }.
 
 In this paper, the avatar of the filtered de Rham cohomology will be the filtered period sheaf ${\mathbb B}_{\dr}$.
 More specifically, we derive the de Rham flip-flopping from a de Rham comparison theorem of Bosco \cite{GB1} that expresses de Rham cohomology twisted by the period ring $\B_{\dr}$ as the pro-\'etale cohomology of the relative period sheaf ${\mathbb B}_{\dr}$. The latter can be defined for adic spaces more general than rigid analytic spaces, for example for perfectoid spaces and, for  trivial reasons, descends along the two towers and hence satisfies flip-flopping in dual towers. This flip-flopping then can be transferred  to de Rham cohomology via the de Rham comparison theorem, but with a twist by $\B_{\dr}$, which can be killed by taking  Galois  fixed points on the level of cohomology. On the derived level the same can be achieved by replacing ${\mathbb B}_{\dr}$ with Fontaine's almost de Rham period  sheaf ${\mathbb B}_{\pdr}$. 
 Similarly, for the avatar of the Hyodo-Kato cohomology we take the period sheaf ${\mathbb B}$ (the relative version of the  Fargues-Fontaine curve  analog of the period ring $\B_{\crr}$). We derive the Hyodo-Kato flip-flopping from a ${\mathbf B}$-comparison theorem of Colmez-Gilles-Nizio{\l}  \cite{CGN2} that expresses Hyodo-Kato  cohomology twisted by the period ring $\B$ as the pro-\'etale cohomology of the relative period sheaf ${\mathbb B}$, which satisfies flip-flopping in dual towers for trivial reasons. 
 On the derived level the sheaf ${\mathbb B}$ is replaced by the 
   period sheaf ${\mathbb B}_{\mathrm{pFF}}$ introduced in \cite{CGN3}. 
     
     That takes care of the flip-flopping part of Theorem \ref{main0}. The admissibility part is now proved by looking at the Lubin-Tate side and using global methods (it would be very interesting to prove this more directly, in particular using purely local arguments).
  
   \subsubsection{Dual basic local Shimura varieties}   The Drinfeld and Lubin-Tate  towers are examples of dual towers of Rapoport-Zink spaces (see \cite{SW})  and one would naturally expect that  a flip-flopping as in Theorem \ref{main0} holds in this more general context. This is indeed the case (see  Theorem \ref{mainSh0} below) and the proof is basically the same as in the Drinfeld and Lubin-Tate case. 
 
   To state the theorem, let $(G,[b],\{\mu\})$ be a basic local Shimura datum over $\Q_p$.  Thus
   $G$ is a connected reductive group over $\Q_p$, $\mu:\mathbb{G}_{m, \overline{\Q}_p}\to G_{\overline{\Q}_p}$ is a (conjugacy class of a) minuscule cocharacter, and 
   $b$ is the unique basic element in the set 
   $B(G,\mu)$ of neutral acceptable elements.  Let $E$ be the reflex field, i.e. the 
   field of definition of (the conjugacy class of) $\mu$ and let  $\breve E$ be the completion of its maximal unramified extension.  Attached to this datum there is a canonical tower $({\rm Sh}(G, b, \mu)_K)_{K}$ (indexed by sufficiently small compact open subgroups $K$ of $G(\Q_p)$) of local Shimura varieties, which are smooth rigid analytic spaces over $\breve E$ (this is one of the main results of \cite{SW1}). As for the Drinfeld tower, we define the groups $H^{\bullet}_{\dr,c}(\mathrm{Sh}({G},{b},{\mu})_{\infty})$ and $H^{\bullet}_{\hk,c}(\mathrm{Sh}({G},{b},{\mu})_{\infty})$ by taking the colimit of the de Rham (respectively Hyodo-Kato) cohomology groups with compact support of the various ${\rm Sh}(G, b, \mu)_{K}$. There is also a natural dual local Shimura datum $(\check G,[\check b],\{\check \mu\})$, and so a dual tower of local Shimura varieties, for which the above considerations apply. 

  \begin{theorem}{\rm(Local Shimura varieties flip-flopping)} 
\label{mainSh0} There are $G(\qp)\times \check G(\qp)$-equivariant
 isomorphisms in solid $\breve{E}$-modules  and solid $(\phi,N,\sg_{\breve{E}})$-modules\footnote{We write $\breve{C}$ for the fraction field of the Witt vectors $W(\overline{{\mathbf F}}_p)$.} over $\breve{C}$, respectively
 $$H^i_{\dr,c}(\mathrm{Sh}(\check{G},\check{b},\check{\mu})_{\infty})\simeq H^i_{\dr,c}(\mathrm{Sh}({G},{b},{\mu})_{\infty}),\quad H^i_{\hk,c}(\mathrm{Sh}(\check{G},\check{b},\check{\mu})_{\infty,C})\simeq H^i_{\hk,c}(\mathrm{Sh}({G},{b},{\mu})_{\infty,C}).
 $$  

\end{theorem}

  \subsection{The basic flip-flopping results} The setup for our main  technical flip-flopping result is the following. Let $K$ be a complete discretely valued field of mixed characteristic $(0,p)$ with perfect  residue field $k$.
Consider a diagram
\begin{equation}\label{pierre1}
\xymatrix{
& T\ar[dl]^-{{G}}_{\pi}\ar@(u,r)[]^{G\times\check{G}} \ar[dr]_{\check{G}}^{\check{\pi}}\\
X=[T/G]\ar@(l,u)[]^{\check{G}} & &   \check{X}=[T/\check{G}]\ar@(u,r)[]^{{{G}}} 
 }
\end{equation}
where

$\bullet$  $X$ and $\check{X}$ are smooth rigid analytic spaces over $K$.

$\bullet$ $T$ is a diamond over ${\rm Spd}(K)$, endowed with a continuous action of $G\times \check{G}$, where 
$G$ and $\check{G}$ are profinite groups.

$\bullet$ The map\footnote{We abuse notation and write $X$ for the diamond attached to $X$.} $\pi: T\to X$
is a $\check{G}$-equivariant pro-\'etale 
$G$-torsor and similarly the map 
 $\check{\pi}: T\to \check{X}$ is a $G$-equivariant pro-\'etale 
$\check{G}$-torsor. 

   The following theorem is proved using the strategy sketched in Section \ref{duck1}:
 \begin{theorem}\label{main1} Assume that $X, \check{X}$ are partially proper. 
 \begin{enumerate}[leftmargin=*]
 \item There is a natural quasi-isomorphism\footnote{Here $K_{\Box}$ denotes the category of solid $K$-modules.}  in $\sd(K_{\Box})$
 $$
\rg(\check{G},F^r\R\Gamma_{\dr,*}(X))  \simeq \rg({G},F^r\R\Gamma_{\dr,*}(\check{X})),\quad r\in\Z, *= \emptyset, c.
$$
\item   If $G$, $\check{G}$ are compact $p$-adic Lie groups with Lie algebras $\mathfrak{g}, \check{\mathfrak{g}}$, then there is a natural $G\times \check{G}$-equivariant quasi-isomorphism in  $\sd(K_{\Box})$
$$\R\Gamma(\check{\mathfrak{g}},K)\otimes^{\LL_{\Box}}_K\R\Gamma_{\dr,c}(X_{\infty})  \simeq \R\Gamma(\mathfrak{g},K)\otimes^{\LL_{\Box}}_K\R\Gamma_{\dr,c}(\check{X}_{\infty}).$$ 
\item The analogous unfiltered statements hold for geometric Hyodo–Kato cohomology over \(C^{\mathrm{nr}}\), compatibly with Frobenius, monodromy, and the \(\sg_K\)-action.
\end{enumerate}
 \end{theorem}
Here the cohomology $\R\Gamma_{\dr,c}(X_{\infty})$ is defined as the colimit of compactly supported de Rham cohomologies from the finite levels of the $\pi$-tower in diagram \eqref{pierre1}; similarly, $\R\Gamma_{\dr,c}(\check{X}_{\infty})$
is defined using the $\check{\pi}$-tower. The second claim follows from the first one because the compactly supported cohomology is a smooth representation in the sense of \cite{RR22}, \cite{RR23}.

         In the case of compactly supported cohomology we have the following underived version of flip-flopping if we are willing to pass to infinity. 
\begin{corollary}\label{dept10} Assume that $X, \check{X}$ are partially proper and 
 that $G, \check{G}$ are $p$-adic Lie groups with
\begin{equation}\label{dimcond}
{\dim}_{K}H^i(\mathfrak{g},K)={\dim}_{K}H^i(\check{\mathfrak{g}},K),\quad i\geq 0.
\end{equation}
Then there exist   compact open subgroups $H\subset G$, $\check{H}\subset \check{G}$,  such that there is a $H\times\check{H}$-equivariant isomorphism\footnote{Here $\sd_{\phi,N,\sg_K}({C}^{\nr}_{\Box})$ denotes the derived $\infty$-category of $(\phi,N,\sg_K)$-modules over ${C}^{\nr}_{\Box}$.}  in $K_{\Box}$ and $\sd_{\phi,N,\sg_K}({C}^{\nr}_{\Box})$, respectively
\begin{align*}
H^i_{\dr,c}(X_{\infty}) & \simeq H^i_{\dr,c}(\check{X}_{\infty}),\quad i\geq 0,\\
H^i_{\hk,c}(X_{\infty,C}) & \simeq H^i_{\hk,c}(\check{X}_{\infty,C}),\quad i\geq 0.\notag
\end{align*}
\end{corollary}
We can choose $H,\check{H}$ to be any compact open subgroups acting trivially on all Lie algebra cohomology groups. We do not know a case when the condition \eqref{dimcond} is not satisfied (it holds for all dual towers considered in Section \ref{shimura1}, because the relevant groups are inner forms).   The de Rham case of Corollary \ref{dept10} is obtained from claim (2) of Theorem \ref{main1} by untwisting the Lie algebra cohomology  using cancellation in categories of finite-rank diagrams. For Hyodo–Kato cohomology  cancellation takes place in the category retaining Frobenius, monodromy, and the Galois action. 
  
   Theorem \ref{mainSh0}  follows directly from Corollary \ref{dept10} since the  condition \eqref{dimcond}  is satisfied in this case (the groups $G, \check{G}$ being inner forms) and the groups $H,\check{H}$ could be taken to be maximal compact open subgroups of $G, \check{G}$, respectively.  To obtain equivariance under the full groups, we perform cancellation in a diagram category containing the support, level, and translation maps.    
 \begin{remark}({\rm Related work.)}
 \begin{enumerate}[leftmargin=*]
\item  An alternative approach to   the de Rham flip-flopping  for dual towers of local Shimura varieties was found  by the first author and  Rodriguez Camargo  (see  \cite[Th. 1.7]{RD}). It is a byproduct of a flip-flopping result for locally analytic vectors in the pro-\'etale cohomology at infinite level of relative period sheaves. Roughly  speaking, working with locally analytic vectors instead of the whole period sheaves avoids the Lie algebra cohomology factors and the subsequent cancellation argument used here.
Hence their flip-flop isomorphism has stronger functoriality properties. 
\item For the Drinfeld and Lubin-Tate towers, yet another approach (though similar in spirit to that of \cite{RD}) to   the de Rham part of Theorem \ref{main0} was found by Su \cite{Su}. 

\item These results can also be obtained using the machinery of de Rham and Hyodo-Kato stacks developed in \cite{RCS}.
\end{enumerate}
 \end{remark}
\begin{acknowledgments} 
 We would like to thank  Piotr Achinger, Grigory Andreychev, Francesco Baldassarri,     Laurent Fargues, Juan Esteban Rodriguez Camargo, and Olivier Taibi  for helpful conversations related to the content of this paper. We have also profited from many conversations with Guido Bosco. Special thanks go to Pierre Colmez for endless conversations and immense optimism: without him this paper would have never been finished. 
 
   We owe a great deal to the referee whose very careful reading of the manuscript improved both the exposition and the results. We have used GPT-6 Astra from OpenAI in final proofreading during which  we were directed towards  the reference \cite{BCB}, which helped us remove the admissibility condition from the flip-flopping of the monodromy in an earlier version of this paper. 
 
 Parts of this paper were written during the second author's stay at the Banach Center in Warsaw in 2019 and the authors' stay at the Hausdorff Research Institute for Mathematics in Bonn in 2023. We would like to thank these institutes for support and hospitality. The first author  would like to thank the Institute of Mathematics "Simion Stoilow" of the Romanian Academy for the wonderful working conditions.

 \end{acknowledgments}
{\bf Notation and conventions}   Let $\so_K$ be a complete discrete valuation ring with fraction field
$K$  of characteristic 0 and perfect
residue field $k$ of characteristic $p$. Let $\ovk$ be an algebraic closure of $K$ and let $\so_{\ovk}$ denote the integral closure of $\so_K$ in $\ovk$.   Set $\sg_K=\Gal(\overline {K}/K)$. Let $C=\wh{\ovk}$ be the $p$-adic completion of $\ovk$.  Let
$W(k)$ be the ring of Witt vectors of $k$ with 
 fraction field $F$ (i.e., $W(k)=\so_F$). Let $K^{\nr}$ denote the maximal unramified extension of $K$ in $C$; let $\breve{K}$ denote its $p$-adic completion.  Let $C^{\nr}$ denote the maximal unramified extension of $F$ in $C$.  Let $\breve{C}$ denote the fraction field of $W(\overline {k})$ and let 
 $\phi$ be the absolute
Frobenius on $W(\overline {k})$.

  We will denote by $\B_{\crr}, {\B}_{\st},\B_{\dr}$ the (absolute)  crystalline, semistable, and  de Rham period rings of Fontaine, respectively. We will also use the (absolute) period rings $\B^+, \B:=\B^+[1/t]$, which are analogs of $\B^+_{\crr}, \B_{\crr}$ tailored to  the Fargues-Fontaine curve. 
We changed the standard notation here a bit replacing $\B$ with $\B^+$ etc  to make it uniform with the $\B_{\dr}$-notation.

  All rigid analytic spaces and dagger spaces considered will be over $K$ or $C$.  We assume that they are separated, taut, and countable at infinity. 
Pro-\'etale cohomology of rigid analytic spaces is the quasi-pro-\'etale cohomology of the associated diamond.  Group cohomology is always (inner) condensed.

\section{Preliminaries}
We gather here the facts we need concerning period rings and their Galois cohomology\footnote{Always taken in the condensed mathematics setting.}.

\subsection{Condensed and solid abelian groups} Let ${\rm Cond}$ (resp. ${\rm CondAb}$) be the category of condensed sets (resp. abelian groups), i.e., accessible sheaves of sets (resp. abelian groups) on the site\footnote{This is also the pro-\'etale site of the geometric point ${\rm Spa}(C, \mathcal{O}_C)$.} ${\rm PS}$ of profinite sets with finite jointly surjective families of maps as coverings. The category ${\rm CondAb}$ is closed symmetric monoidal and the forgetful functor ${\rm CondAb}\to {\rm Cond}$ has a left adjoint denoted by
 $X\mapsto \Z[X]$, the free condensed abelian group on $X$.

There is a functor $\underline{(-)}: {\rm Top}_1\to {\rm Cond}$ from the category 
of ${\rm T}1$ topological spaces (i.e., all points are closed) to ${\rm Cond}$, given by  
$\underline{T}(S):=\mathcal{C}^0(S, T)$, for $S\in {\rm PS}$. 
This functor is fully faithful on compactly generated weakly Hausdorff spaces, which cover all topological spaces we deal with in this paper, so we will sometimes be sloppy and simply write 
$T$ instead of $\underline{T}$, especially when $T$ is a discrete topological space (such as $\Z$ or $\mathcal{C}(S, \Z)$, for $S\in {\rm PS}$) or a profinite set. The functor $T\mapsto \underline{T}$ restricts to a functor
 ${\rm TopAb}\to {\rm CondAb}$, where 
 ${\rm TopAb}$ is the category of 
${\rm T}1$ topological abelian groups. 

    Let ${\rm Solid}$ be the full subcategory of 
  ${\rm CondAb}$ generated under (small) colimits by the objects $\Z^I:=\prod_{I} \Z$, for arbitrary sets $I$. A theorem of Clausen and Scholze   ensures that 
  ${\rm Solid}$ is stable under (small) limits, colimits, extensions (in particular it is abelian), the $\Z^I$ (as $I$ varies) form a class 
  of compact projective generators, and there is a closed 
   symmetric monoidal structure $\otimes^{\Box}$ on ${\rm Solid}$ for which $\Z^I\otimes^{\Box} \Z^J=\Z^{I\times J}$, for all sets $I,J$, and such that $\otimes^{\Box}$ commutes with colimits separately in each variable. Moreover, the inclusion 
${\rm Solid}\subset {\rm CondAb}$ has a symmetric monoidal left adjoint $M\mapsto M^{\Box}$ sending 
$\Z[S]$ (with $S\in {\rm PS}$) to  $\Z[S]^{\Box}:=\underline{{\rm Hom}}(\mathcal{C}^0(S, \Z), \Z)$, and 
     $\Z_{\Box}[S]:=\Z[S]^{\Box}$ is compact projective in ${\rm Solid}$. For any condensed ring $A$ we thus have 
$A[S]^{\Box}\simeq \Z_{\Box}[S]\otimes^{\Box} A^{\Box}$.
     
     The category ${\rm Solid}$ contains all discrete abelian groups (seen as objects of ${\rm CondAb}$), thus it also contains $\Z_p$ and $\Q_p$. It easily follows (see \cite[Prop. A.31]{GB1}) that ${\rm Solid}$ contains all $\underline{V}$, with $V$ a complete locally convex $\Q_p$-vector space (for instance a $\Q_p$-Banach space or a Fr\'echet space). 

\subsection{Solid and Fr\'echet $K$-vector spaces}

If $A$ is a solid $\Q_p$-algebra\footnote{That is,  a commutative unital $\Q_p$-algebra object in ${\rm Solid}$.}, we let 
$$A_{\Box}:={\rm Mod}_A({\rm Solid})$$
be the category of 
$A$-modules in ${\rm Solid}$. Setting $A_{\Box}[S]:=A[S]^{\Box}=\Z_{\Box}[S]\otimes^{\Box} A\in A_{\Box}$  (where $A[S]:=\Z[S]\otimes_{\Z} A\in {\rm CondAb}$), the category $A_{\Box}$ is generated under colimits by the $A_{\Box}[S]$, which are compact projective, and 
$A_{\Box}$ has a natural closed symmetric monoidal structure $\otimes^{\Box}_A$ for which $A_{\Box}[S]\otimes^{\Box}_A A_{\Box}[S']\simeq A_{\Box}[S\times S']$, for $S,S'\in {\rm PS}$. If $A=\underline{B}$, with $B$ a complete locally convex 
$\Q_p$-algebra, we simply write $B_{\Box}:=A_{\Box}$ and $B_{\Box}[S]:=A_{\Box}[S]$.

Let ${\rm Fr}_K$ be the category of Fr\'echet spaces over $K$ (i.e., countable inverse limits of $K$-Banach spaces), morphisms being $K$-linear continuous maps. By the open mapping theorem any surjective (as map of abstract $K$-vector spaces) map in ${\rm Fr}_K$ is a quotient map and any bijective morphism in ${\rm Fr}_K$ is an isomorphism of topological $K$-vector spaces. Hence a complex
$A\to B\to C$ in ${\rm Fr}_K$ is exact as a complex of topological $K$-vector spaces if and only if it is exact as a complex of abstract $K$-vector spaces.
By the above discussion, the functor $V\mapsto \underline{V}$ induces a fully faithful embedding ${\rm Fr}_K\subset K_{\Box}$. Recall the following facts:

\begin{proposition}

a) The functor ${\rm Fr}_K\to K_{\Box}$ 
is exact and symmetric monoidal. \footnote{We use the completed projective tensor product on ${\rm Fr}_K$.} 

b) For all $V\in {\rm Fr}_K$, the object $\underline{V}$ of $K_{\Box}$ is a filtered colimit of (images of) $K$-Banach spaces, and is flat for $\otimes^{\Box}_K$, i.e., $\underline{V}\otimes_K^{\Box}(-)$ is exact on $K_{\Box}$.
\end{proposition}

\begin{remark}
 As $S$ varies in ${\rm PS}$, the $V_S:=\underline{\rm Hom}_{{\rm CondAb}}(\Z[S], \Q_p)=\underline{\rm Hom}_{{\rm Solid}}(\Z_{\Box}[S], \Q_p)$ describe the full subcategory of ${\rm Solid}$ spanned by the $\Q_p$-Banach spaces, and they satisfy $V_S\otimes^{\Box}_{\Q_p} V_{S'}\simeq V_{S\times S'}$ (see \cite[Prop. A.49, Lemma A.30]{GB1}).
  \end{remark}

 Let $A$ be a solid $\Q_p$-algebra. Proposition $A.29$ in \cite{GB1} shows that there is a natural analytic ring structure 
 $\mathcal{A}=(A, \Z)_{\Box}$ with underlying condensed ring $A$ and functor of measures $S\mapsto \mathcal{A}[S]:=A_{\Box}[S]$ (for $S$ in the category ${\rm ED}$ of extremally disconnected profinite sets, i.e., projective objects of ${\rm PS}$). Moreover the category $A_{\Box}$ is the same as the full subcategory ${\rm Mod}_{\mathcal{A}}$ of ${\rm Mod}_A({\rm CondAb})$ consisting of $\mathcal{A}$-complete modules, and the tensor product $\otimes^{\Box}_A$ is nothing but the corresponding closed symmetric monoidal structure $\otimes_{\mathcal{A}}$ on ${\rm Mod}_{\mathcal{A}}$.

     \subsection{Condensed group cohomology} Let $G$ be a profinite group and let $M\in \sd(\Z[G])$. 
     We define 
     $$\R\Gamma(G, M):=\R\underline{{\rm Hom}}_{\Z[G]}(\Z, M)\in \mathcal{D}({\rm CondAb}).$$
      If moreover $M\in \sd(\mathrm{Solid})$, by \cite[Prop. B.2]{GB1}, there is a natural quasi-isomorphism 
     $$\R\Gamma(G, M)\simeq [M\to \underline{{\rm Hom}}(\Z[G], M)\to \underline{{\rm Hom}}(\Z[G^2], M)\to\ldots],$$
     the differentials being the usual ones in the bar resolution of $\Z$ by $\Z[G]$-modules. If $M$ is a $\underline{G}$-module in $K_{\Box}$, we can replace $\Z$ by $K_{\Box}$ in the above formulae. 
     
     Let ${\rm Nuc}_K\subset K_{\Box}$ be the full subcategory of nuclear solid $K$-vector spaces. This subcategory is generated under colimits by 
     (images of) $K$-Banach spaces and it is stable under countable products;  it contains all images of $K$-Fr\'echet spaces (see \cite[Prop. 
     A.57, Prop.  A.64]{GB1}).

   \begin{proposition}\label{magic}
     Let $G$ be a profinite group and let $W\in {\rm Nuc}_K$ be endowed with the trivial $G$-action. For any $V\in \mathcal{D}^+(K_{\Box}[G])$ there is a natural quasi-isomorphism
          $$\R\Gamma(G, V)\otimes^{\LL_\Box}_K W\simeq \R\Gamma(G, V\otimes^{\LL_\Box}_K W).$$
   \end{proposition}
   
   \begin{proof} Note first that there is a natural map  
   $$\R\Gamma(G, V)\otimes^{\LL_\Box}_K W\to \R\Gamma(G, V)\otimes^{\LL_\Box}_K \R\Gamma(G, W)\to \R\Gamma(G, V\otimes^{\LL_\Box}_K W),$$
   the first map being induced by the unit map $W\to \R\Gamma(G, W)$ and the second map being the derived cup-product pairing. 
    To see that the above map is a quasi-isomorphism we may assume (by d\'evissage) that $V$ is concentrated in degree $0$. Since everything commutes with colimits in $W$, we may also assume that $W$ is (the image of) a $K$-Banach space, in particular it is flat. But then 
   $\R\Gamma(G, V)\otimes^{\LL_\Box}_K W$ is quasi-isomorphic to the complex $\underline{\rm Hom}_K(K_{\Box}[G^{\bullet}], V)\otimes^{\Box}_K W$, while $ \R\Gamma(G, V\otimes^{\LL_\Box}_K W)$ is quasi-isomorphic to the complex $\underline{\rm Hom}_K(K_{\Box}[G^{\bullet}], V\otimes^{\Box}_K W)$. Now, \cite[Prop. 5.35]{And} provides a natural isomorphism $\underline{\rm Hom}_K(K_{\Box}[S], V)\otimes^{\Box}_K W\simeq \underline{\rm Hom}_K(K_{\Box}[S], V\otimes^{\Box}_K W)$, for any profinite set $S$, yielding the result. 
   
           \end{proof}

\begin{proposition}\label{contcond}
Let $V\in {\rm Fr}_K$ be a $K$-Fr\'echet space endowed with a continuous $K$-linear action of a profinite group 
$G$ and let $n\geq 0$. If $H^n_{\rm cont}(G,V)$ is a $K$-Fr\'echet space, then there is a natural isomorphism in 
$K_{\Box}$
$$H^n(G,\underline{V})\simeq \underline{H^n_{\rm cont}(G,V)}.$$
\end{proposition}

\begin{proof}
  If $W\in {\rm Fr}_K$ and $H$ is a profinite group, then $\mathcal{C}(H,  W)$ is naturally in ${\rm Fr}_K$ 
  and it is not difficult to check that $\underline{{\rm Hom}}(H, \underline{W})=\underline{\mathcal{C}(H,  W)}$. 
  Let $D^n=\mathcal{C}(G^n, V)\in {\rm Fr}_K$, then $\R\Gamma(G,\underline{V})$ is quasi-isomorphic to the complex 
  $\underline{D^{\bullet}}$, and $\R\Gamma(G,V)$ is quasi-isomorphic to the complex 
  $D^{\bullet}$ in the derived category of locally convex vector spaces over $K$. Thus it suffices to see that if $D^{\bullet}$ is a complex of $K$-Fr\'echet spaces and if 
  $H^n(D^{\bullet})$ is itself Fr\'echet, then $H^n(\underline{D^{\bullet}})=\underline{H^n(D^{\bullet})}$. This follows easily from the exactness of the functor $\underline{(-)}$ on the category ${\rm Fr}_K$. Indeed, let 
\[
Z^n=\ker(d^n),\qquad B^n=\operatorname{im}(d^{n-1}).
\]
If \(H^n_{\mathrm{cont}}=Z^n/B^n\), with its usual quotient topology, is Fr\'echet, then \(B^n\) is closed in \(Z^n\), thus \(B^n\) is Fr\'echet and the open mapping theorem makes \(D^{n-1}\to B^n\) strict, giving the required comparison after condensation.
\end{proof}

\subsection{Period rings and their Galois cohomology} We pass now to Galois cohomology of some period rings. 
\subsubsection{Galois cohomology of $C$}
  
 For  $j\in\Z$, let $C(j)$ be the $\sg_K$-module twist of $C$ by $\chi_{\rm cycl}^j$, where $\chi_{\rm cycl}:\sg_K\to\Z_p^\dual$
is the cyclotomic character.
\begin{proposition} {\rm (Tate)}\label{gopal1}
We have isomorphisms in $K_{\Box}$
$$
H^i(\sg_K,C(j))\simeq \begin{cases} K&{\text{if  $j=0, i=0,$}}\\ 
K\cdot \log\chi_{\rm cycl}&{\text{if $j=0,i=1,$}}\\
0&{\text{otherwise}}.
\end{cases}
$$
\end{proposition}

\begin{proof}  This  is a classical computation of Tate ~\cite{Ta67} in the context of continuous cohomology. To see that the same result holds in the condensed setting it suffices to 
use Proposition \ref{contcond}.
\end{proof}

  We will now modify the period ring $C$ to kill the Galois cohomology classes in degree $1$ in the above formulas. We define  $\log t$ as a transcendental variable over $C$, and  we equip  $C[\log t]$ with the action of 
 $\sg_K$ given by  $ g (\log t)=\log t+\log\chi_{\rm cycl}( g )$, $ g \in \sg_K$.
\begin{lemma}\label{Galois11}

 We have isomorphisms in $K_{\Box}$
$$
H^i(\sg_K,C[\log t](j))\simeq \begin{cases} K & \mbox {if } i=0,j=0,\\
 0 & \mbox{otherwise}.
\end{cases}
$$
\end{lemma}
\begin{proof} 
Let $A=C[\log t]$ and let $A_k=C[\log t]^{\leq k}:=\sum_{i=0}^k C (\log t)^i$, so that 
  $H^i(\sg_K, A)=\varinjlim_{k} H^i(\sg_K, A_k)$.
   
   We start by proving that the natural map $K\to H^0(\sg_K, A_k)$ is an isomorphism for all $k$. It suffices to check that, for $k>0$, the natural inclusion 
$H^0(\sg_K, A_{k-1})\subset H^0(\sg_K, A_k)$ is an equality. Let $P=\sum_{j=0}^k v_j (\log t)^j\in H^0(\sg_K, A_k)$ and let 
$ g \in \sg_K$. Identifying the coefficients of $(\log t)^k$ in $ g (P)=P$ yields $v_k\in H^0(\sg_K, C)\simeq K$. Identifying the coefficients of $(\log t)^{k-1}$ yields $-kv_k \log \chi_{\rm cycl}( g )=( g -1)v_{k-1}$. The isomorphism $H^1(\sg_K,  C)=K\log \chi_{\rm cycl}$ forces 
$v_k=0$ and so $P\in H^0(\sg_K, A_{k-1})$, as desired.

   Next, we observe that $A_k$ is stable under the action of $\sg_K$ and we have an exact sequence of $\sg_K$-modules, for all $j\in \mathbf{Z}$:
   $$0\to A_k(j)\to A_{k+1}(j)\to  C(j)\to 0.$$
   Consider the long exact sequence in cohomology associated to this exact sequence. If $j\ne 0$, Proposition \ref{gopal1} immediately yields (by induction on $k$) the vanishing of all $H^i(\sg_K, A_k(j))$, so we are done (by passing to the colimit over $k$). Suppose now that 
$j=0$.  Since, by Proposition \ref{gopal1}  and induction on \(k\), 
  we have $H^i(\sg_K,A_k)=0$, for $i\ge 2,$   the long exact sequence reduces to the following exact sequence 
$$0\to H^0(\sg_K, C)\simeq K\to H^1(\sg_K, A_k)\to H^1(\sg_K, A_{k+1})\to H^1(\sg_K, C)\simeq K\log \chi_{\rm cycl}\to 0.$$
 An induction on $k$ shows that $\dim H^1(\sg_K, A_k)=1$ for all $k$, thus the maps on the left and on the right of the exact sequence must be isomorphisms and the middle arrow is $0$, in particular $H^1(\sg_K, A)=\varinjlim_{k} H^1(\sg_K, A_k)=0$. 
\end{proof}

\subsubsection{Almost de Rham period ring}
Let $\B^+_{\pdr}$ be the algebra $\B^+_{\dr}[\log t]$ defined in \cite[Sec. 4.3]{FonA} and let $\B_{\pdr}:=\B^+_{\pdr}[1/t]$. The Galois group $\sg_K$ acts on $\B^+_{\pdr}$ via ring homomorphisms extending its usual action on $\B^+_{\dr}$ and such that $ g (\log t)=\log t+\log(\chi_{\rm cycl}( g ))$, $ g \in\sg_K$. This action naturally extends to $\B_{\pdr}$. The filtration on $\B_{\pdr}$ is given by powers of $t$. 

%Moreover there is a unique $\B_{\dr}$-derivation $\nu_{\B_{\pdr}}$ such that $\nu_{\B_{\pdr}}(\log(t))=-1$, and it obviously preserves $\B^+_{\pdr}$ and commutes with $\sg_K$.
\begin{proposition} \label{Galois1}
 We have natural quasi-isomorphisms in $\mathcal{D}(K_{\Box})$ ($r\in \Z$)
$$\R\Gamma(\sg_K, t^{r}\B^+_{\pdr})\simeq \begin{cases} K[0] & \mbox {if } r\leq 0,\\
 0 & \mbox{if } r>0,
 \end{cases}
$$
and 
$$
\R\Gamma(\sg_K, \B^+_{\pdr})\simeq \R\Gamma(\sg_K,\B_{\pdr})\simeq K[0].$$

% we have $H^i(\sg_K, t^{r}\B^+_{\pdr})=0$ for $i\geq 1$ and 
%$$H^0(\sg_K,t^{r}\B^+_{\pdr})\simeq \begin{cases} K & \mbox {if } r\leq 0,\\
 %0 & \mbox{if } r>0.
 %\end{cases}
% $$
 
\end{proposition}
\begin{proof} Let $A_n=C[\log t]^{\leq n}=\sum_{i=0}^n C (\log t)^i$. The proof of Lemma \ref{Galois11} shows that 
$H^i(\sg_K, A_n(j))=0$ for all $i,n\geq 0$ and $j\ne 0$. Moreover, the inductive system $(H^i(\sg_K, A_n))_{n\geq 0}$ is 
zero for $i\geq 1$ (more precisely $H^i(\sg_K, A_n)=0$ for $i\geq 2$, and the map $H^1(\sg_K, A_n)\to H^1(\sg_K, A_{n+1})$ is $0$) and it is constant of value $K$, for $i=0$.

   Let 
   $$M_r:=t^{r}\B^+_{\pdr}=\colim_{n} M_r^{\leq n},\,\, M_r^{\leq n}:= t^{r}\B_{\dr}^+[\log t]^{\leq n}=\sum_{i=0}^n t^r\B_{\dr}^+ (\log t)^i.$$
   We claim that, for any $i\geq 0, r\in \Z$, we have an isomorphism of inductive systems 
   $$(H^i(\sg_K, M_r^{\leq n}))_{n}\simeq \begin{cases} (H^i(\sg_K, A_n(r)))_n & \mbox {if } r\geq 0,\\
 (H^i(\sg_K, A_n))_n & \mbox{if } r<0.
\end{cases}$$

   The claim implies the result for $H^*(\sg_K, M_r)=\colim_n H^*(\sg_K, M_r^{\leq n})$, using the description of the inductive system $(H^i(\sg_K, A_n(r)))_n $ in the first paragraph. With this in hand, the second part of the proposition is immediate, since $H^i(\sg_K, \B_{\pdr})=\colim_{r\geq 0} H^i(\sg_K, t^{-r}\B_{\pdr}^+)$.

   Fix $n\geq 0$ and $r\in \Z$, we  define 
   $$X_m:=t^{r} (\B_{\dr}^+/t^m)[\log t]^{\leq n},$$
   so that $M_r^{\leq n}=\varprojlim_{m} X_m$, $X_1\simeq A_n(r)$ and we have an exact sequence 
   $$0\to {\rm R^1lim}_m\, H^{i-1}(\sg_K, X_m)\to H^i(\sg_K, M_r^{\leq n})\to \lim_{m}\, H^i(\sg_K, X_m)\to 0.$$
   The natural projection $\B_{\dr}^+/t^{m+1}\to \B_{\dr}^+/t^m$ induces an exact sequence of
$\sg_K$-modules
$$0\to A_n(m+r)\to X_{m+1}\to X_m\to 0.$$

 Suppose that $r\geq 0$. Then $H^i(\sg_K, A_n(m+r))=0$, for all $m\geq 1$, thus the projective system 
 $(H^i(\sg_K, X_m))_m$ is constant with value $H^i(\sg_K, A_n(r))$, for all $i\geq 0$, and the result follows.
 
   Suppose that $r=-j$ with $j>0$. Then the above exact sequence and the fact that $H^i(\sg_K, X_1)=H^i(\sg_K, A_n(r))=0$,
   for all $i$, imply that $H^i(\sg_K, X_k)=0$, for $k\leq j$ and all $i$, that the projective system 
   $(H^i(\sg_K, X_m))_m$ is constant starting with $m=j+1$, with value $H^i(\sg_K, M_r^{\leq n})$, and we have 
   $H^i(\sg_K, A_n)\simeq H^i(\sg_K, X_{j+1})\simeq H^i(\sg_K, M_r^{\leq n})$, for all $i$. By construction, these isomorphisms are compatible with the variation of $n$, yielding the result. 
 \end{proof}

\begin{remark} \label{Galois-zurich} 
\begin{enumerate}[leftmargin=*]
\item Lemma \ref{Galois11} and Proposition \ref{Galois1} are folklore. Somewhat different proofs to the ones presented above can be found in \cite[Prop. 1.2, Th. 1.4]{CGN3}. 
\item Proposition \ref{magic} and the above results yield the following more general version of Proposition \ref{Galois1}: 
if $W\in {\rm Nuc}_K$ is equipped  with the trivial $\sg_K$-action, then there are 
 natural quasi-isomorphisms in $\mathcal{D}(K_{\Box})$
 $$
\R\Gamma(\sg_K,W\otimes^{\Box}_KC[\log t](j))\simeq 
\begin{cases} W[0] & \mbox {if } j=0,\\
 0 & \mbox{otherwise.}
\end{cases}
$$
This implies that  we have natural quasi-isomorphisms in $\mathcal{D}(K_{\Box})$ ($r\in \Z$)
$$\R\Gamma(\sg_K, W\otimes^{\Box}_Kt^{r}\B^+_{\pdr})\simeq \begin{cases} W[0] & \mbox {if } r\leq 0,\\
 0 & \mbox{if } r>0,
 \end{cases}
$$
and 
$$
\R\Gamma(\sg_K, W\otimes^{\Box}_K\B^+_{\pdr})\simeq \R\Gamma(\sg_K,W\otimes^{\Box}_K\B_{\pdr})\simeq W[0].$$
\end{enumerate}
\end{remark}
\subsubsection{Period ring  $\B_{\rm pFF}$} \label{chatus1} Our main reference here is \cite{CGN3}. 

   Let  $Y_{\rm FF}$ be the  Fargues-Fontaine curve obtained  by removing from ${\rm Spa}(\A_{\rm inf})$
the divisors  $p=0$ and    $[p^\flat]=0$.
Let  $I$ be an interval of  $\R_+^\dual$. Let  $\B_I=\O(Y_{\rm FF,I})$, where
$Y_{\rm FF,I}$ is the locus of  $v$ such that  $v([p^{\flat}])/v(p)\in I$.
If  $I=[r,s]$, with  $r,s\in\Q$, then $\B_I=\A_{[r,s]}[\frac{1}{p}]$, where  $\A_{[r,s]}$ is the $p$-adic completion
of  $\A_{\rm inf}[\frac{[p^{\flat}]^{1/r}}{p},\frac{p}{[p^{\flat}]^{1/s}}]$.
Set
$$\B^+:=\B^{]0,\infty[}=\O(Y_{\rm FF})=\varprojlim\B_{[r,s]},\quad \B:=\B^+[1/t].$$  
(Usually, the ring ${\B}^+$ is denoted by ${\B}$.)
The natural Frobenius automorphism $\varphi$ on ${\A}_{\rm inf}$ induces isomorphisms $\varphi: {\B}_I\simeq {\B}_{I/p}$ and, in the limit, an automorphism 
     $\varphi$ of ${\B}^+$ and ${\B}$. 
   We denote  by $x_{\rm dR}\in Y_{\rm FF}$ the point defined by  $[p^{\flat}]=p$; the completion of the local ring at this point  is  $\B_{\dr}^+$, which yields a $\G_K$-equivariant  injective morphism
 $\B^+\hookrightarrow\B_{\dr}^+$.

   Define  $U=\log[p^{\flat}]$ as a transcendental element over  $\B^+$, with 
$$\varphi(U)=p\,U
\quad{\rm and }\quad
g(U)=U+ \log [g(p^{\flat})/p^{\flat}].$$
We set:
$$\B_{\rm FF}^+:=\B^+[U],\quad\B_{\rm FF}:=\B_{\rm FF}^+[\tfrac{1}{t}].$$
The ring
    ${\B}^+_{\rm FF}$ is endowed with a surjective monodromy operator $N=-\frac{d}{dU}$, commuting with the action of $\sg_K$, such that 
    $N\varphi=p\varphi N$ and ${\B}^+= {\B}_{\rm FF}^{+,N=0}$.   We extend the canonical  injection
$\B^+\hookrightarrow\B_{\dr}^+$ to a ring morphism 
$\B_{\rm FF}^+\otimes_FK\to\B_{\dr}^+$ by sending $U$ to $\log\frac{[p^{\flat}]}{p}$.
This morphism is  $\sg_K$-equivariant and  injective.

     To define $\B_{\rm pFF}^+$, 
choose an increasing exhaustion \((I_n)_{n\geq0}\) of
\((0,\infty)\) by bounded open intervals with rational endpoints.
For each finite extension \(L/K\) in \(\overline K\), let
\(\B^{\mathrm{uni}}_{I_n,L}\) denote the universal unipotent
\(\sg_L\)-algebra over \(\B_{I_n}\) constructed in \cite{CGN3},
by successively adjoining primitives of Galois 1-cocycles; denote by $\B^{\mathrm{uni}}_{I_n,m,L}$ its level $m$ part (built from words of length $\leq m$).  These are $\sg_K$-modules over $\B_{I_n}$. 
Fix compatible restriction and finite-extension morphisms as in
loc.\ cit., and set
\begin{align*}
 & \B_{\mathrm{pFF},m,L}^{+}
:=
\lim_n \B^{\mathrm{uni}}_{I_n,m,L},\quad 
\B^+_{\mathrm{pFF},L}
:=
\colim_m \B^+_{\rm pFF,m,L},\\
 & \B_{\mathrm{pFF}}^{+}
:=
\operatorname*{colim}_{K\subset L\subset\overline K,\,[L:K]<\infty}
\B_{\mathrm{pFF},L}^{+}.
\end{align*}
We set 
$ \B_{\mathrm{pFF},m,L}:=  \B_{\mathrm{pFF},m,L}^{+}[1/t], \B_{\mathrm{pFF},L}:=\B_{\mathrm{pFF},L}^{+}[1/t], \B_{\mathrm{pFF}}:=\B_{\mathrm{pFF}}^{+}[1/t].$
  The operators $\phi, N$ extend from $\B^+_{\rm FF}$ to $\B^+_{\rm pFF}$ and then to $\B_{\rm pFF}$ (with the commutativity condition $N\varphi=p\varphi N$ on $\phi,N$ dropped). 
   We choose a $\sg_K$-equivariant and injective morphism 
$\B_{\rm pFF}\otimes_FK\to\B_{\rm pdR}.$
Moreover, we have

\begin{theorem}\label{hawaii1}
For $\Lambda=\B_{\rm pFF}^+,\B_{\rm pFF}$, we have
$$\R\Gamma(\G_K,\Lambda)\simeq F[0].$$
\end{theorem}
\begin{remark}
 Proposition \ref{magic} and Theorem \ref{hawaii1}  yield the following more general version of Theorem  \ref{hawaii1}: 
if $W\in {\rm Nuc}_F$ is equipped  with the trivial $\sg_K$-action, then there are
 natural quasi-isomorphisms in $\mathcal{D}(F_{\Box})$
 $$
\R\Gamma(\sg_K, W\otimes^{\Box}_F\B^+_{\rm pFF})\simeq \R\Gamma(\sg_K,W\otimes^{\Box}_F\B_{\rm pFF})\simeq W[0].$$
\end{remark}

\subsubsection{Derived $\B_e$-formula} Set  $\B_e:=\B^{\phi=1}$.  We quote the following result\footnote{The second claim in the proposition is obtained from the first one by passing to a colimit.}:
\begin{proposition}{\rm (\cite[Prop. 10.3.20]{FF})} \label{FFB}
\begin{enumerate}[leftmargin=*]
\item Let $M$ be a finite rank $(\phi,N)$-module over $F$.  We have a  canonical isomorphism\footnote{Perhaps the quickest way to see this is to equip $M$ with an admissible filtration (this is always possible) and to use the associated Galois representation.}
 of $(\phi, N, \sg_K)$-modules over $\B_{\rm FF}$:
   \begin{equation}\label{dec1} (M\otimes_{F}\B_{\rm FF})^{N=0,\phi=1} \otimes_{\B_e}\B_{\rm FF}\stackrel{\sim}{\to} M\otimes_F\B_{\rm FF}.
   \end{equation}
   It follows that
 the canonical morphism
   $$ M\to   ((M\otimes_{F}\B_{\rm FF})^{N=0,\phi=1} \otimes_{\B_e}\B_{\rm FF})^{\sg_K}$$
   is an isomorphism of $(\phi, N)$-modules over $F$. 
   \item Let $M$ be a finite rank $(\phi,N,\sg_K)$-module over $C^{\nr}$. The canonical morphism
   $$ M\to   ((M\otimes_{{C}^{\nr}}\B_{\rm FF})^{N=0,\phi=1} \otimes_{\B_e}\B_{\rm FF})^{\sg_K-{\rm sm}}$$
   is an isomorphism of $(\phi, N,\sg_K)$-modules over $C^{\rm nr}$.  Here $(-)^{\sg_K-{\rm sm}}:=\colim_{H\subset \sg_K}(-)^H$, where the index set runs over open subgroups $H$ of $\sg_K$. 
\end{enumerate}
\end{proposition}
 
  The following Corollary of Theorem \ref{hawaii1} is a derived version of Proposition \ref{FFB}.
  \begin{corollary}\label{derived-FFB}
  \begin{enumerate}[leftmargin=*]
\item Let $M$ be a bounded  complex of finite rank $(\phi,N)$-modules over $F$. The canonical morphism\footnote{The prime in $\sd^{\prime}_{\phi,N,\sg_K}(-)$ refers to dropping the commutativity condition $N\varphi=p\varphi N$ on $\phi,N$.} in $\sd^{\prime}_{\phi,N}(F)$ 
   $$ M\to   \R\Gamma(\sg_K,[M\otimes^{\LL_{\Box}}_{F}\B_{\rm FF}]^{N=0,\phi=1} \otimes^{\LL_\Box}_{\B_e}\B_{\rm pFF})$$
   is a  quasi-isomorphism. 
   \item Let $M$ be a bounded complex of finite rank $(\phi,N,\sg_K)$-modules over $C^{\nr}$, with a smooth semilinear \(\sg_K\)-action. The canonical morphism  in $\sd^{\prime}_{\phi,N,\sg_K}(C^{\nr})$
   $$ M\to   ([M\otimes^{\LL_{\Box}}_{{C}^{\nr}}\B_{\rm FF}]^{N=0,\phi=1} \otimes^{\LL_\Box}_{\B_e}\B_{\rm pFF})^{\R\sg_K-{\rm sm}}$$
   is a quasi-isomorphism.  Here we set $(-)^{\R\sg_K-{\rm sm}}:=\colim_{H\subset \sg_K}(-)^{\R H}$. 
\end{enumerate}
  \end{corollary}
  \begin{proof} For the first claim,  we may assume that $M=M[0]$ for a $(\phi,N)$-module over $F$.  We need to show that the canonical morphism
  $$
  M\to   \R\Gamma(\sg_K,[M\otimes^{\LL_{\Box}}_{F}\B_{\rm FF}]^{N=0,\phi=1} \otimes^{\LL_\Box}_{\B_e}\B_{\rm pFF})
  $$
  is a quasi-isomorphism. Since we have Proposition \ref{FFB}, it suffices to show that
  \begin{align}\label{dark1}
   H^i(\sg_K,[M\otimes^{\LL_{\Box}}_{F}\B_{\rm FF}]^{N=0,\phi=1} \otimes^{\LL_\Box}_{\B_e}\B_{\rm pFF})=0, \quad i\geq 1. 
  \end{align}
  But, by \eqref{dec1},  we have
$$  [M\otimes^{\LL_{\Box}}_{F}\B_{\rm FF}]^{N=0,\phi=1} \otimes^{\LL_\Box}_{\B_e}\B_{\rm pFF}  \simeq (M\otimes^{{\Box}}_{F}\B_{\rm FF})^{N=0,\phi=1} \otimes^{\Box}_{\B_e}\B_{\rm pFF} \simeq M \otimes^{\Box}_{F}\B_{\rm pFF}$$
  and the Galois action on $M \otimes^{\Box}_{F}\B_{\rm pFF}$ is via the action on $\B_{\rm pFF}$. 
  Hence, for $i\geq 1$, by Theorem \ref{hawaii1}
  $$
  H^i(\sg_K,[M\otimes^{\LL_{\Box}}_{F}\B_{\rm FF}]^{N=0,\phi=1} \otimes^{\LL_\Box}_{\B_e}\B_{\rm pFF})\simeq H^i(\sg_K, M \otimes^{\Box}_{F}\B_{\rm pFF})=0,
  $$
  as wanted.
  
  The second claim follows from the first one by a colimit argument:  Choose a finite Galois extension \(L/K\) such that \(\sg_L\) fixes a basis of each term of \(M\). Setting \(L_0=W(k_L)[1/p]\), we have
$M_L:=M^{\sg_L}, M\simeq M_L\otimes_{L_0}C^{\mathrm{nr}},$ 
where \(M_L\) is a bounded complex of finite-rank \((\varphi,N)\)-modules over \(L_0\). For every finite extension \(L'/L\), apply claim (1) over \(L'\) to $M_L\otimes_{L_0}L'_0.$ 
Taking the filtered colimit over \(L'\) gives the  reconstruction of \(M\) we wanted.
   \end{proof}

\section{Comparison theorems} We recall here the de Rham and Hyodo-Kato comparison theorems from \cite{GB1}, \cite{GB2}, \cite{CN4}, \cite{CN5} and adapt  them to compactly supported cohomology. For $L=K,C$ let 
${\rm Sm}_L$ be the category of smooth rigid analytic varieties over $L$.
\subsection{Period sheaves} Let $X$ be an analytic adic space over ${\rm Spa}(\Q_p, \Z_p)$. On the pro-\'etale site\footnote{Recall that, by definition, this is the quasi-pro-\'etale site of the associated diamond $X^{\diamond}$.} 
   $X_{\proet}$ of $X$ one has the following sheaves of rings:
   
   $\bullet$ the integral structure sheaf $\widehat{\mathcal{O}}^{+}_X$ and its tilted version 
   $\widehat{\mathcal{O}}^{\flat,+}_X=\varprojlim_{s\mapsto s^p} \widehat{\mathcal{O}}^+_X/p$. For any perfectoid space 
   $U\in X_{\proet}$ we have $\widehat{\mathcal{O}}^{+}_X(U)=\mathcal{O}^+_{U^{\sharp}}(U^{\sharp}) $ and $\widehat{\mathcal{O}}^{\flat,+}_X(U)=\mathcal{O}^+_{U}(U)$, where 
   $U^{\sharp}$ is the untilt of $U$ induced by the map $U\to X^{\diamond}\to {\rm Spd}(\Z_p)$.
   
   $\bullet$ the sheaf (of $p$-typical Witt vectors) $\mathbb{A}_{\rm inf}=W(\widehat{\mathcal{O}}^{\flat,+}_X)$ endowed with Fontaine's map $\theta: \mathbb{A}_{\rm inf}\to \widehat{\mathcal{O}}^+_X$, which extends to a map 
   $\theta: \mathbb{B}_{\rm inf}:=\mathbb{A}_{\rm inf}[1/p]\to \widehat{\mathcal{O}}_X
   :=\widehat{\mathcal{O}}^+_X[1/p]$. 
   
   $\bullet$ the relative de Rham period sheaves $\mathbb{B}_{\dr}^+=\varprojlim_{n} \mathbb{B}_{\rm inf}/(\ker \theta)^n$, endowed with the filtration ${\rm F}^n \mathbb{B}_{\dr}^+=(\ker \theta)^n
    \mathbb{B}_{\dr}^+$. Locally on $X_{\proet}$ there is a generator 
    $t$ of ${\rm F}^1 \mathbb{B}_{\dr}^+$, it is a non-zero divisor and is unique up to a unit, which allows us to define the sheaf $\mathbb{B}_{\dr}=\mathbb{B}_{\dr}^+[1/t]$, with the filtration 
    ${\rm F}^n \mathbb{B}_{\dr}=t^n \mathbb{B}_{\dr}^+$.
    
    $\bullet$ If $I=[s,r]\subset (0,\infty)$ is a compact interval with rational endpoints, one defines the sheaf 
    $\mathbb{A}_{\rm inf, I}=\mathbb{A}_{\rm inf}[\frac{p}{[\alpha]}, \frac{[\beta]}{p}]$, where 
    $\alpha, \beta\in \mathcal{O}_{C^{\flat}}$ satisfy $v(\alpha)=1/r$ and $v(\beta)=1/s$ (here 
    $v$ is the standard valuation on $C^{\flat}$). The sheaf $\mathbb{B}_I$ is obtained from 
    $\mathbb{A}_{\rm inf, I}$ by $p$-adic completion, followed by inverting $p$. Finally, the period sheaves $\mathbb{B}^+$ and 
    $\mathbb{B}$ are defined by $\mathbb{B}^+=\varprojlim_I \mathbb{B}_I$ and $\mathbb{B}=\mathbb{B}^+[1/t]$.
    The natural Frobenius automorphism $\varphi$ on $\mathbb{A}_{\rm inf}$ induces isomorphisms $\varphi: \mathbb{B}_I\simeq \mathbb{B}_{I/p}$ and, in the limit, an automorphism 
     $\varphi$ of $\mathbb{B}^+$ and $\mathbb{B}$.

     $\bullet$  The sheaves 
 $$\sbb^+_{\pdr}:=\sbb^+_{\dr}[\log t],\quad \sbb_{\pdr}:=\sbb^+_{\pdr}[1/t]
  $$
  will play a crucial role in this paper. They are endowed with natural filtrations induced by that on 
  $\sbb^+_{\dr}$.
  We endow them with a Galois action by setting $g(\log t)=\log t+\log(\chi_{\rm cycl}(g))$, $g\in \sg_K$.

    $\bullet$  We finally define $\mathbb{B}^+_{\rm FF}:=\mathbb{B}^+[U]$, $U=\log[p^{\flat}]$ endowed with a Frobenius 
    $\varphi$ and an action of $\sg_K$ extending the ones on $\mathbb{B}^+$ and such that 
    $\varphi(U)=pU$ and $g(U)=U+\log [g(p^{\flat})/p^{\flat}]$ for $g\in \sg_K$. The sheaf 
    $\mathbb{B}^+_{\rm FF}$ is endowed with a surjective monodromy operator $N=-\frac{d}{dU}$, commuting with the action of $\sg_K$, such that 
    $N\varphi=p\varphi N$ and $\mathbb{B}^+= \mathbb{B}_{\rm FF}^{+,N=0}$.  We set $\mathbb{B}_{\rm FF}:=\mathbb{B}^+_{\rm FF}[1/t]$. 
    
  $\bullet$    The sheaves 
 $\sbb^+_{\rm pFF}, \sbb_{\rm pFF}:=\sbb^+_{\rm pFF}[1/t]
  $
  will play a crucial role in this paper.   They are defined  starting with the finite-level
sheaves
\[
\mathbb{B}^{\mathrm{uni}}_{I_n,m,L}:= \{S\mapsto {\B}^{\mathrm{uni}}_{I_n,m,L}(S)\},
\]
where ${\B}^{\mathrm{uni}}_{I_n,m,L}(S)$  is the $S$-version of the absolute period  ring from Section \ref{chatus1}, 
 and then applying 
the same formulas as in Section \ref{chatus1} in the category of $v$-sheaves.
  We endow them with a Frobenius,  monodromy,  and Galois descent.

     For $X={\rm Spa}(C, \mathcal{O}_C)$, the above sheaves define condensed rings 
     $\mathbf{A}_{\rm inf}, \mathbf{B}, \mathbf{B}_{\dr}, \mathbf{B}_{\pdr}, \mathbf{B}_{\rm FF}, \mathbf{B}_{\rm pFF}$, etc, which are condensed versions of Fontaine's classical period rings.

\subsection{Condensed enhancement of sheaf cohomology} Let $X$ be a rigid analytic variety over $K$ and let 
$\mathcal{B}$ be the basis of affinoid open subspaces of $X$. If $F$ is a 
${\rm TopAb}$-valued sheaf on $\mathcal{B}$ (for instance $F$ could be a coherent sheaf on $X$, in which case $F(U)$ is naturally a $K$-Banach space for all $U\in \mathcal{B}$), there is a unique ${\rm CondAb}$-valued sheaf $F_{\rm cond}$ such that $F_{\rm cond}(U)=\underline{F(U)}$ for all affinoid open subspaces $U$ of $X$. If $F(U)\in {\rm Fr}_K$ for all $U\in \mathcal{B}$, then $F_{\rm cond}$ is a $K_{\Box}$-valued sheaf on $X$. We thus obtain a complex 
$$\R\Gamma_{\rm cond}(X,F):=\R\Gamma(X, F_{\rm cond})\in \mathcal{D}({\rm CondAb})$$ with cohomology groups 
$H^i_{\rm cond}(X,F)\in {\rm CondAb}$, which are enhancements of the usual $\R\Gamma(X, F)\in \mathcal{D}({\rm Ab})$ and 
$H^i(X,F)\in {\rm Ab}$ in the sense that $\R\Gamma_{\rm cond}(X,F)(*)=\R\Gamma(X, F)$ and similarly for cohomology groups. Moreover, if $F$ is ${\rm Fr}_K$-valued on $\mathcal{B}$ then $\R\Gamma_{\rm cond}(X,F)\in \mathcal{D}(K_{\Box})$.
A similar discussion applies if we replace $F$ with a bounded below complex of such sheaves. 

  If $X$ is an analytic adic space over $C$ (i.e. over ${\rm Spa}(C, \mathcal{O}_C)$), we can replace in the above discussion the analytic site of $X$ with its pro-\'etale site $X_{\proet}$ and $\mathcal{B}$ with the basis of $X_{\proet}$ consisting of affinoid perfectoid $U\in X_{\proet}$. Thus any ${\rm TopAb}$-valued sheaf $F$ on $\mathcal{B}$ yields an enhanced cohomology complex 
   $\R\Gamma_{\rm cond}(X_{\proet}, F):=\R\Gamma(X_{\proet}, F_{\rm cond})\in \mathcal{D}({\rm CondAb})$, and if
   $F$ is ${\rm Fr}_K$-valued then $\R\Gamma_{\rm cond}(X_{\proet}, F)\in \sd(K_{\Box})$. A similar discussion applies when $F$ is a bounded below complex of such sheaves. 

\begin{remark}\label{drcond} For any $X\in {\rm Sm}_K$, we obtain a bounded complex $$\R\Gamma_{\dr, \rm cond}(X):=\R\Gamma_{\rm cond}(X, \Omega^{\bullet}_X)\in \mathcal{D}(K_{\Box}),$$ whose cohomology groups are denoted $H^i_{\dr, \rm cond}(X)\in K_{\Box}$. It has a natural enhancement as object of the filtered $\infty$-category 
  $\sff\sd(K_{\Box})$ of $\sd(K_{\Box})$, where $F^r \R\Gamma_{\dr}(X)=\R\Gamma_{\rm cond}(X, \Omega^{\geq r}_X)$.
 If $X$ is either Stein or proper, then $H^i_{\dr, \rm cond}(X)=\underline{H^i_{\dr}(X)}$, with 
$H^i_{\dr}(X)\in {\rm Fr}_K$ (even finite dimensional if $X$ is proper).
 If $X$ is an affinoid of positive dimension, then 
$H^i_{\dr, \rm cond}(X)\in K_{\Box}$ is \emph{not} necessarily in (the image of) ${\rm Fr}_K$. Since we will always work in the condensed world, we will simply write $\rg_{\dr}(X)\in\mathcal{D}(K_{\Box})$ instead of $\R\Gamma_{\dr, \rm cond}(X)$ and $H^i_{\dr}(X)\in K_{\Box}$ instead of $H^i_{\dr, \rm cond}(X)$. 
\end{remark}

    If $X$ is an analytic adic space over $C$ and if 
 $\sff$ is one of the sheaves $\mathbb{B}_{\dr}^+, \mathbb{B}_{\dr}, \mathbb{B}_{\rm pdR}, \mathbb{B}_I, \mathbb{B}^+$, then $\sff$ defines a ${\rm TopAb}$-valued sheaf on the basis $\mathcal{B}$ of affinoid perfectoids $U\in X_{\proet}$. More precisely, if $\sff\in\{\mathbb{B}_{\dr}^+, \mathbb{B}_I, \mathbb{B}^+\}$ then $\sff(U)$ is naturally in ${\rm Fr}_K$ for $U\in \mathcal{B}$, while if $\sff\in \{\mathbb{B}_{\dr}, \mathbb{B}_{\rm pdR}\}$ then $\sff(U)$ is a countable colimit of Fr\'echet spaces, with closed immersions as transition maps. It follows that the associated sheaf $\sff_{\rm cond}$ is $K_{\Box}$-valued
 and $\R\Gamma_{\rm cond}(X_{\proet}, \sff)\in 
    \mathcal{D}(K_{\Box})$. Moreover, it is nuclear. In what follows, unless this causes confusion, we will simply write $\R\Gamma_{\proeet}(X, \sff)$ for $\R\Gamma_{\rm cond}(X_{\proet}, \sff)$.
      
          The following result will be very useful later (it is also instrumental in the proof of Lemma \ref{HS1}):
     \begin{lemma}{\rm (\cite[proof of Prop. 4.12]{GB1})} \label{HS}
     Let $\sff$ be one of the sheaves $\Q_p, \mathbb{B}_{\dr}^+, \mathbb{B}_{\dr}, \mathbb{B}_{\pdr}, \mathbb{B}_I, \mathbb{B}^+$. For any profinite set 
     $S$ we have a natural isomorphism in $\mathcal{D}({\rm Solid})$
     $$\R\Gamma_{\proeet}(X\times S, \sff)\simeq \R\underline{{\rm Hom}}(\Z[S], \R\Gamma_{\proeet}(X, \sff)).$$
     \end{lemma}
     
      % From now on, if $F$ is one of the above sheaves we will simply write $\rg_{\proet}(X,F)$ instead of 
      % $\R\Gamma_{\rm cond}(X_{\proet}, F)$.

     \subsection{Hochschild-Serre spectral sequences}We will now construct Hochschild-Serre spectral sequences for a variety of period sheaves. For that, 
we will use the following definition of torsors:
\begin{definition}Let $G$ be a profinite group. 
A map $f : Y\to X $ of analytic adic spaces over $\Z_p$ is  called a {\em  pro-\'etale $G$-torsor} if the associated map of diamonds $f^{\diamond} : Y^{\diamond} \to X^{\diamond}  $ is endowed with
an action of $G$ on $Y^{\diamond} $ over $X^{\diamond} $ and, quasi-pro-\'etale locally on $X^{\diamond} $, there is a $G$-equivariant
isomorphism of diamonds $Y^{\diamond}  \simeq X^{\diamond} \times G^{\diamond} $.  In such  case, we will call $G$ the Galois group of $f$.
\end{definition}
\begin{lemma}{\rm (Hochschild-Serre spectral sequences)} \label{HS1}Let $G$ be a profinite group and let  $f : Y\to X $ be  a {\em  pro-\'etale $G$-torsor}. Let $\sff$ be as in Lemma \ref{HS}. Then:
\begin{enumerate}
\item There is  a natural Hochschild-Serre quasi-isomorphism in $\sd({B_{\Box}})$, for  $B=\sff(C)$.
$$
\R\Gamma(G,\R\Gamma_{\proeet}(Y,\sff))\simeq  \R\Gamma_{\proeet}(X,\sff).
$$
\item  There is  a natural Hochschild-Serre spectral sequence
$$
E^{i,j}_2=H^i(G, H^j_{\proeet}(Y,\sff))\Rightarrow  H^{i+j}_{\proeet}(X,\sff).
$$
\end{enumerate}
\end{lemma}
\begin{proof} The case of $\Q_p$ was treated in \cite[Lemma 4.4]{CGN1}. The cases of ${\mathbb B}_{\dr}, {\mathbb B}^+, {\mathbb B}_I$ were treated in  \cite[Prop. 4.12]{GB1}. The case of ${\mathbb B}_{\pdr}$ 
follows by the same type of argument since the analog of \cite[Prop. 4.7]{GB1} holds for it by a colimit  argument starting from ${\mathbb B}^+_{\dr}$.
\end{proof}

\subsection{De Rham comparison theorem} Let  
  $X\in {\rm Sm}_K$. See Remark \ref{drcond} for our conventions concerning 
  the object $\R\Gamma_{\dr}(X)$ of the filtered $\infty$-category $\sff\sd(K_{\Box})$ of 
  $\sd(K_{\Box})$.
    
  Let $*\in \{\dr, \pdr\}$. The $\B_{*}$-cohomology of $X$ (\cite{BMS}, \cite{Guo} for $*=\dr$) is the object of the filtered $\infty$-category $\sff\sd({\B_{*,\Box}})$ of $\sd({\B_{*,\Box}})$ defined by 
$$
  \R\Gamma_{\dr}(X_C/\B_{*}):=\R\Gamma_{\dr}(X)\otimes_K^{{\rm L}_{\Box}}\B_{*}.
  $$
 It is endowed with the tensor product filtration ($r\in\Z$) 
$$
F^r \R\Gamma_{\dr}(X_C/\B_{*}):=\colim_{j+k\geq r} (F^j \R\Gamma_{\dr}(X)\otimes_K^{{\rm L}_{\Box}}F^k\B_{*}).
$$

\begin{example}
 If $X$ is Stein or affinoid, thanks to the vanishing of higher coherent cohomology on $X$, we have natural quasi-isomorphisms in $\sd({\B_{*,\Box}})$
\begin{align*}
 \R\Gamma_{\dr}(X_C/\B_{*}) & \simeq (\so(X)\otimes^{\Box}_K\B_{*}\to  \Omega^1(X)\otimes^{\Box}_K\B_{*}\to\ldots) \\
F^r \R\Gamma_{\dr}(X_C/\B_{*}) & \simeq  (\so(X)\otimes^{{\Box}}_Kt^r\B^+_{*}\to  \Omega^1(X)\otimes^{{\Box}}_Kt^{r-1}\B^+_{*}\to\ldots)\notag
\end{align*}
Here we were able to replace the derived tensor product by the tensor product itself because the spaces $\Omega^i(X)$ are Fr\'echet, hence flat ($\B_{*}$ is also flat, as a filtered colimit of Fr\'echet spaces). 
\end{example}

 Similarly, the object
 $\rg_{\proet}(X_C,\sbb_{*})\in \sd({\B_{*,\Box}})$ can be promoted 
  to an object of $\sff\sd({\B^+_{*,\Box}})$ by setting 
  $$F^r \rg_{\proet}(X_C,\sbb_{*}):=\rg_{\proet}(X_C, F^r \sbb_{*})\in \sd({\B^+_{*,\Box}}).$$

 We quote the following de Rham comparison theorem:
\begin{theorem}{\rm (Bosco \cite[Th. 1.8, Rem. 6.14]{GB1})} \label{Guid1}
Let $X\in {\rm Sm}_K$. 
\begin{enumerate}
\item We have a natural  quasi-isomorphism in $\sd({\B^+_{\dr,\Box}})$
\begin{align}\label{ober1p}
F^r\rg_{\proet}(X_C,\sbb_{\dr}) & \simeq F^r\rg_{\dr}(X_C/\B_{\dr}),\quad r\in\Z.
\end{align}
\item If  $X$ is quasi-compact we also have  a natural  quasi-isomorphism in $\sd({\B_{\dr,\Box}})$
\begin{align}\label{ober1pp}
\rg_{\proet}(X_C,\sbb_{\dr}) & \simeq \rg_{\dr}(X_C/\B_{\dr}).
\end{align}
\end{enumerate}
\end{theorem}

\begin{remark}
We warn the reader that the first claim in the above theorem yields a natural  quasi-isomorphism in $\sd({\B^+_{\dr,\Box}})$
\begin{align}\label{oberp10}
\rg_{\proet}(X_C,\sbb^+_{\dr})  \simeq  F^0(\rg_{\dr}(X)\otimes^{\LL_{\Box}}_K\B_{\dr}).
\end{align}
That is, the filtration on the right is twisted! For an untwisted version (see \cite[Th. 5.1]{GB2}) we need to change the group $\rg_{\proet}(X_C,\sbb^+_{\dr}) $ by using the d\'ecalage operator: setting\footnote{Here $\nu: X_{C,\proet}\to X_{C,\eet}$ is the canonical projection of sites.}
$$
\rg^{\eta}_{\proet}(X_C,\sbb^+_{\dr}):=\rg_{\eet}(X_C,\LL\eta_t\R\nu_*\sbb^+_{\dr})\in \sd({\B_{\dr,\Box}^+}),
$$ 
there is a 
natural quasi-isomorphism in $\sd({\B^+_{\dr,\Box}})$ 
$$\rg^{\eta}_{\proet}(X_C,\sbb^+_{\dr})\simeq \rg_{\dr}(X)\otimes^{\LL_{\Box}}_{K}\B^+_{\dr}.$$
\end{remark}

   We will need the following series of facts later on.

\begin{lemma}\label{GBcoh}
 Let $G$ be a profinite group acting (continuously) on $X\in {\rm Sm}_K$. There are natural quasi-isomorphisms in $\sd({\B_{\dr,\Box}})$ and $\sd(\B^+_{\dr, \Box})$, respectively:
  \begin{align*}\R\Gamma(G,  \R\Gamma_{\dr}(X_C/\B_{\dr})) & \simeq \R\Gamma(G, \R\Gamma_{\dr}(X))\otimes^{\LL_\Box}_K \B_{\dr},\\
 \R\Gamma(G,  F^r\R\Gamma_{\dr}(X_C/\B_{\dr})) & \simeq \colim_{j+k\geq r} \R\Gamma(G, F^j \R\Gamma_{\dr}(X))\otimes^{\LL_\Box}_K F^k\B_{\dr}.
 \end{align*}
\end{lemma}

\begin{proof}
 The first quasi-isomorphism follows directly from Proposition \ref{magic}, since $\B_{\dr}\in {\rm Nuc}_K$, as a filtered colimit of $K$-Fr\'echet spaces. For the second one, we use that $\R\Gamma(G,-)$ commutes with finite colimits and that 
 $F^r\B_{\dr}\in {\rm Fr}_K\subset {\rm Nuc}_K$, so by Proposition \ref{magic} we can write 
 \begin{align*}
 \R\Gamma(G,  F^r\R\Gamma_{\dr}(X_C/\B_{\dr})) & \simeq \R\Gamma(G, \colim_{j+k\geq r} F^j \R\Gamma_{\dr}(X)\otimes^{\LL_\Box}_K F^k\B_{\dr})\\
  & \simeq 
 \colim_{j+k\geq r} \R\Gamma(G, F^j \R\Gamma_{\dr}(X)\otimes^{\LL_\Box}_K F^k\B_{\dr})\\
 & \simeq
  \colim_{j+k\geq r} \R\Gamma(G, F^j \R\Gamma_{\dr}(X))\otimes^{\LL_\Box}_K F^k\B_{\dr},
  \end{align*}
  as wanted.
\end{proof}

\begin{corollary} \label{pGuid} If $X\in {\rm Sm}_K$ is quasi-compact, there are natural quasi-isomorphisms in $\sd({\B^+_{\pdr,\Box}})$ and $\sd({\B_{\pdr,\Box}})$, respectively:
\begin{align}\label{ober3}
F^r\rg_{\proet}(X_C,\sbb_{\pdr}) & \simeq F^r\rg_{\dr}(X_C/\B_{\pdr}),\quad r\in\Z.
\end{align}
\begin{align}\label{ober2}
\rg_{\proet}(X_C,\sbb_{\pdr}) & \simeq \rg_{\dr}(X_C/\B_{\pdr}).
\end{align}

\end{corollary}
\begin{proof} Since $X$ is quasi-compact, $\rg_{\proet}(X_C,-)$ commutes with filtered colimits, we have a natural quasi-isomorphism
 $$\rg_{\proet}(X_C,\sbb_{\pdr})\simeq  \rg_{\proet}(X_C,\sbb_{\dr})[\log t].$$
 Since $\rg_{\dr}(X)\otimes^{\LL_{\Box}}_{K} (-)$ commutes with colimits, we also have a natural quasi-isomorphism
 $$\rg_{\dr}(X_C/\B_{\pdr}) \simeq \rg_{\dr}(X)\otimes^{\LL_{\Box}}_{K}\B_{\pdr}\simeq (\rg_{\dr}(X)\otimes^{\LL_{\Box}}_{K}\B_{\dr})[\log t].$$
Hence the result is a direct consequence of Theorem \ref{Guid1}.
\end{proof}

\begin{lemma}\label{killing1} For any $X\in {\rm Sm}_K$ there are natural quasi-isomorphisms in $\sd(K_{\Box})$ ($r\in\Z$)
\begin{align*}
 \R\Gamma_{\dr}(X) &\simeq \R\Gamma(\sg_K, \R\Gamma_{\dr}(X_C/\B_{\pdr}) ), \\
F^r\R\Gamma_{\dr}(X) &\simeq \R\Gamma(\sg_K, F^r\R\Gamma_{\dr}(X_C/\B_{\pdr}) ).
\end{align*}
\end{lemma}
\begin{proof} 
  For simplicity write $A=\B_{\pdr}$ and $M=\R\Gamma_{\dr}(X)$. Note that the terms of $M$ and 
  $F^k M$ belong to ${\rm Nuc}_K$, as filtered colimits of $K$-Fr\'echet spaces. Thus by Proposition 
  \ref{magic} (and the fact that $\R\Gamma(\sg_K,-)$ commutes with finite colimits) we obtain natural quasi-isomorphisms
  \begin{align*}\R\Gamma(\sg_K, \R\Gamma_{\dr}(X_C/\B_{\pdr}) ) &\simeq \R\Gamma(\sg_K, M\otimes^{\LL_{\Box}}_K A)\simeq 
  \R\Gamma(\sg_K, A)\otimes^{\LL_{\Box}}_K M,\\
  \R\Gamma(\sg_K, F^r\R\Gamma_{\dr}(X_C/\B_{\pdr}) ) & \simeq \R\Gamma(\sg_K, \colim_{j+k\geq r} F^j M\otimes^{\LL_{\Box}}_K F^k A)\\
   & \simeq \colim_{j+k\geq r} \R\Gamma(\sg_K, F^j M\otimes^{\LL_{\Box}}_K F^k A)\simeq \colim_{j+k\geq r} \R\Gamma(\sg_K, F^kA)\otimes^{\LL_{\Box}}_K F^jM.
   \end{align*}
  The result follows then by the computation of $ \R\Gamma(\sg_K, A)$ and $\R\Gamma(\sg_K, F^kA)$ in Proposition \ref{Galois1}.
       \end{proof}

\begin{corollary}\label{quotG}
 If $X\in {\rm Sm}_K$ is quasi-compact and endowed with a continuous action of a profinite group 
 $G$, there are natural quasi-isomorphisms in $\sd(K_{\Box})$ ($r\in\Z$)
 \begin{align*}
 \R\Gamma(G, F^r\R\Gamma_{\dr}(X))\simeq \R\Gamma(G\times\sg_K, F^r\rg_{\proet}(X_C,\sbb_{\pdr}) ), \\
\R\Gamma(G,\R\Gamma_{\dr}(X))\simeq \R\Gamma(G\times\sg_K, \rg_{\proet}(X_C,\sbb_{\pdr}) )\notag
\end{align*}
\end{corollary}

\begin{proof} It suffices to write $\R\Gamma(G\times \sg_K, -)=\R\Gamma(G, \R\Gamma(\sg_K, -))$ and 
use Corollary \ref{pGuid}, Lemma \ref{killing1} and Lemma \ref{GBcoh}.
\end{proof}

\subsection{De Rham cohomology with compact support} \label{compact-bonn}
Let $L=K, C$. For any functor $F$ defined on ${\rm Sm}_L$, with values in some $\infty$-category with small colimits, and any $X\in {\rm Sm}_L$ we define 
  $$F(\partial X)=\colim_{Z\in\Phi(X)} F(X\setminus Z),$$
  where $\Phi(X)$ is the set of quasi-compact open subsets of $X$.  
   In particular, for $X\in {\rm Sm}_K$, we obtain objects $\R\Gamma_{\dr}(\partial X)\in \sff\sd(K_{\Box})$, $\R\Gamma_{\dr}(\partial X_C/\B_{\dr})\in \sff\sd({\B_{\dr,\Box}})$, as well as 
  $\rg_{\proet}(\partial X_C, \mathcal{F})\in \sd({\rm CondAb})$ for any pro-\'etale abelian sheaf $\mathcal{F}$ on $X_C$. 
  
     Using these constructions we can define various compactly supported cohomology groups
  (the brackets $[\cdots]$ denote the mapping fiber)
  \begin{align*}\rg_{\dr, c}(X) & :=[\rg_{\dr}(X)\to \rg_{\dr}(\partial X)]\in \sff\sd(K_{\Box}), \\
  F^r\rg_{\dr, c}(X) & :=[F^r\rg_{\dr}(X)\to F^r\rg_{\dr}(\partial X)], \quad r\in\Z.
  \end{align*}
  The same construction applied to the functor $X\mapsto \rg_{\dr}(X_C/\B_{\dr})$ yields 
  $\rg_{\dr,c}(X_C/\B_{\dr})\in \sff\sd({\B_{\dr,\Box}})$ and, for any abelian pro-\'etale sheaf 
  $\mathcal{F}$ on $X_{C, \proet}$, an object $\R\Gamma_{\proet,c}(X_C, \mathcal{F})\in 
  \sd({\rm CondAb})$. In particular, we get  $\rg_{\proet,c}(X_C, \sbb_{\dr})\in \sff\sd({\B_{\dr,\Box}})$ by setting 
  $$F^r \rg_{\proet,c}(X_C, \sbb_{\dr}):=\rg_{\proet,c}(X_C, F^r\sbb_{\dr})\simeq 
  [F^r \rg_{\proet}(X_C, \sbb_{\dr})\to F^r\rg_{\proet}(\partial X_C, \sbb_{\dr})].$$
  
  \begin{remark} Let  $X$ be a smooth partially proper  rigid analytic variety over $K$ and  let $\{X_n\}_{n\in \N}$ be an increasing covering of $X$ by quasi-compact opens.  We can write:
   \begin{align*}
  \rg_{\dr}(\partial X_C/\B_{\dr}) \simeq \colim_{n}\rg_{\dr}((X\setminus X_{n})_C/\B_{\dr})
\end{align*}
and similarly with other cohomology theories.
\end{remark}

\begin{example}  If $X$ is Stein of pure dimension $d$, since $H^i_c(X,\Omega^j_X)=0$ for $i\neq d$ (by Serre duality and the fact that coherent sheaves on $X$ are acyclic thanks to Kiehl's theorem), we compute in $\sd({\B_{\dr,\Box}})$:
 \begin{align*}
\R\Gamma_{\dr,c}(X_C/\B_{\dr}) & \simeq (H^d_c(X,\so_X)\otimes^{\LL_{\Box}}_K\B_{\dr}\to H^d_c(X,\Omega^1_X)\otimes^{\LL_{\Box}}_K\B_{\dr}\to\ldots)[-d]\\
& \simeq (H^d_c(X,\so_X)\otimes^{{\Box}}_K\B_{\dr}\to H^d_c(X,\Omega^1_X)\otimes^{{\Box}}_K\B_{\dr}\to\ldots)[-d].
\end{align*}
Here we were able to remove the derived tensor product because $\B_{\dr}$ is flat (as filtered colimit of $K$-Fr\'echet spaces). For similar reasons we have  in $\sd({\B^+_{\dr,\Box}})$, for $r\in\Z$: 
\begin{align}\label{hk1}
F^r\R\Gamma_{\dr,c}(X_C/\B_{\dr}) & \simeq (H^d_c(X,\so_X)\otimes^{\LL_{\Box}}_KF^r\B_{\dr}\to H^d_c(X,\Omega^1_X)\otimes^{\LL_{\Box}}_KF^{r-1}\B_{\dr}\to\ldots)[-d]\\
& \simeq (H^d_c(X,\so_X)\otimes^{{\Box}}_Kt^r\B^+_{\dr}\to H^d_c(X,\Omega^1_X)\otimes^{{\Box}}_Kt^{r-1}\B^+_{\dr}\to\ldots)[-d],\notag
\end{align}
and a natural quasi-isomorphism in $\sd(K_{\Box})$
\begin{align}\label{drwinter}
F^r\R\Gamma_{\dr,c}(X) \simeq (H^d_c(X,\Omega^{r}_X)\to H^d_c(X,\Omega^{r+1}_X)\to\ldots)[-d-r].
\end{align}

\end{example}

\begin{proposition}\label{commutefilcolim}
 If $X\in {\rm Sm}_K$ is partially proper, then 
 $\R\Gamma_{\proet, c}(X_C, -)$ commutes with filtered colimits of pro-\'etale abelian sheaves. 
\end{proposition}

\begin{proof}  Let $(\mathcal{F}_i)_{i\in I}$ be a filtered direct system of pro-\'etale sheaves and let 
$\mathcal{F}=\colim_{i\in I} \mathcal{F}_i$. Set 
$$P(X)=[\colim_{i} \R\Gamma_{\proet}(X_C, \mathcal{F}_i)\to \R\Gamma_{\proet}(X_C, \mathcal{F})].$$
It is not difficult to see that the quasi-isomorphism we want to establish is equivalent to proving that the natural map $P(X)\to P(\partial X):=\colim_{Z\in\Phi(X_C)} P(X_C\setminus Z)$ is a quasi-isomorphism. 
This follows using the same arguments as in \cite[Lemma 2.31]{AGN}.
\end{proof}

\begin{proposition}
 If $X\in {\rm Sm}_K$ is partially proper, then there are natural quasi-isomorphisms  in $\sd({\B_{\dr,\Box}})$ and $\sd({\B^+_{\dr,\Box}})$, respectively:
 \begin{align*}
  \rg_{\proet,c}(X_C, \sbb_{\dr}) & \simeq \rg_{\dr,c}(X)\otimes^{\LL_{\Box}}_{K}\B_{\dr},\\
    F^r\rg_{\proet,c}(X_C, \sbb_{\dr}) & \simeq F^r(\rg_{\dr,c}(X)\otimes^{\LL_{\Box}}_{K}\B_{\dr}),\quad r\in\Z.
  \end{align*}
\end{proposition}
\begin{proof}  To simplify the notation,  set
 $$M=\R\Gamma_{\dr}(X), \,\, \partial M=\R\Gamma_{\dr}(\partial X), \,\, M_c=\R\Gamma_{\dr,c}(X)=[M\to \partial M].$$
Since  mapping fibers commute with finite colimits,  using Theorem \ref{Guid1} we obtain 
natural quasi-isomorphisms 
 \begin{align*}
 F^r\rg_{\proet,c}(X_C, \sbb_{\dr}) & \simeq [F^r \rg_{\proet}(X_C, \sbb_{\dr})\to F^r \rg_{\proet}(\partial X_C, \sbb_{\dr})]\\
  & \simeq
  [F^r (M\otimes^{\LL_{\Box}}_{K}\B_{\dr})\to F^r(\partial M\otimes^{\LL_{\Box}}_{K}\B_{\dr})]\simeq F^r (M_c\otimes^{\LL_{\Box}}_{K}\B_{\dr}).
  \end{align*}
  This shows the second quasi-isomorphism of the proposition. 
  
  Concerning the first quasi-isomorphism, for $T\in \sd({\B^+_{\dr,\Box}})$, define 
$$T[1/t]:=T\otimes^{\LL_{\Box}}_{\B^+_{\dr}}\B_{\dr}\in \sd({\B_{\dr,\Box}}).$$
Proposition \ref{commutefilcolim} yields a natural quasi-isomorphism
  $$\rg_{\proet,c}(X_C,\sbb^+_{\dr})[1/t] \stackrel{\sim}{\to} \rg_{\proet,c}(X_C,\sbb_{\dr}).$$
   By what we have already established we have  natural quasi-isomorphisms 
  $$\rg_{\proet,c}(X_C,\sbb^+_{\dr})\simeq F^0\rg_{\proet,c}(X_C, \sbb_{\dr})\simeq F^0(M_c\otimes^{\LL_{\Box}}_{K}\B_{\dr}).$$
  Thus  it remains to see that we have a canonical  quasi-isomorphism
  $$F^0(M_c\otimes^{\LL_{\Box}}_{K}\B_{\dr})[1/t]\stackrel{\sim}{\to}M_c\otimes^{\LL_{\Box}}_{K}\B_{\dr}.$$
  But this is a formal consequence of the bounded filtration and the fact that every \(t^a\B_{\mathrm{dR}}^+\) becomes \(\B_{\mathrm{dR}}\) after inverting \(t\).
  \end{proof}
\subsection{Hyodo-Kato comparison theorem} 
Let $\sd_{\phi,N}({\B_{\Box}})$  be the $\infty$-derived category of solid $(\phi,N)$-modules over $\B$; similarly for various variants.

  We quote the following Hyodo-Kato comparison theorem: 
\begin{proposition}\label{Guid2}
For any $X\in {\rm Sm}_K$ 
there is a natural $\sg_K$-equivariant quasi-isomorphism in $\sd_{\phi}({\B_{\Box}})$
\begin{align}\label{dark3}
\rg_{\proet}(X_C,\sbb^+) \otimes^{\LL_{\Box}}_{\B^+}\B& \simeq [\rg_{\rm HK}(X_C)\otimes^{\LL_{\Box}}_{{C}^{\nr}}\B_{\rm FF}]^{N=0}.
\end{align}
\end{proposition}
\begin{proof} Let $d=\dim X$ and $r\geq 2d$. By \cite[Cor. 4.9]{CGN2}, we have a  natural $\sg_K$-equivariant comparison quasi-isomorphism in $\sd_{\phi}({\B^+_{\Box}})$  (see \cite[Cor. 4.9]{CGN2} for an explanation of the notation)
\begin{align}\label{CGN2}
\rg_{\proet}(X_C,\sbb^+) & \simeq \big[[\rg_{\rm HK}(X_C)\{r\}\otimes^{\LL_{\Box}}_{{C}^{\nr}}\B^+_{\rm FF}]^{N=0}\stackrel{\iota_{\hk}^{\B}}{\longrightarrow} \R\Gamma^{\B^+}_{\dr}(X_C,r)\big](-r).
\end{align}
The de Rham-type  complex $ \R\Gamma^{\B^+}_{\dr}(X_C,r)$  is annihilated by a fixed power of $t$.  The wanted quasi-isomorphism \eqref{dark3} follows. Note that \cite[Cor. 4.9]{CGN2} is stated with an assumption of partial properness on the space, but as we recall in the proof of proposition \ref{action1} below, the argument given above is still valid.
\end{proof}
\begin{remark} 
\begin{enumerate}[leftmargin=*]
\item Bosco in  \cite[Th. 4.1, Rem. 6.13]{GB2} proved a version of the comparison quasi-isomorphism \eqref{CGN2}, where the torsion
on the right-hand side is incorporated to the left-hand side via the $ \LL\eta_t$ operator. This also yields the quasi-isomorphism \eqref{dark3}.
\item Let $I\subset (0,\infty)$ be a compact interval with rational endpoints such that $t$ has one zero on ${\rm Spa}(\B_I)$. Then, still for $r\geq 2d$ (with $d=\dim X$),  we have the following analog of \eqref{CGN2}
\begin{equation}\label{CGNdr}
\rg_{\proet}(X_C,\sbb_I)  \simeq \big[[\rg_{\rm HK}(X_C)\{r\}\otimes^{\LL_{\Box}}_{{C}^{\nr}}\B_{I,\log}]^{N=0}\stackrel{\iota_{\hk}}{\longrightarrow} \R\Gamma_{\dr}(X_C,r)\big](-r).
\end{equation}
 Tensoring both sides of \eqref{CGNdr} with $\B_I/t^n\simeq \B^+_{\dr}/t^n$, $n\in\N$,   passing to  limit over $n$, and dropping the twist by $r$ since we do not care about the Frobenius anymore,  we get a quasi-isomorphism in $\sd(\B^+_{\dr,\Box})$
\begin{equation}\label{CGN4}
 \rg_{\proeet}(X_C,{\mathbb B}^+_{\dr})  \simeq F^0(\rg_{\dr}(X)\otimes^{\LL_{\Box}}_{K}\B_{\dr}),
\end{equation}
which recovers \eqref{ober1p}.  This proof  of \eqref{CGN4} is however more involved than the one of Bosco  in \cite{GB1} that follows directly from the  ${\mathbb B}_{\dr}$-Poincar\'e Lemma.
\end{enumerate}
\end{remark}
\begin{corollary}\label{killing1HK} Let  $X\in {\rm Sm}_K$. Assume that it has   finite rank de Rham cohomology. Then there are  natural $\sg_K$-equivariant quasi-isomorphisms in $\sd^{\prime}_{\phi,N}({C}^{\nr}_{\Box})$
\begin{align*}
 \R\Gamma_{\hk}(X_C)  & \stackrel{\sim}{\to} ([\R\Gamma_{\hk}(X_C)\otimes^{\LL_{\Box}}_{{C}^{\nr}}\B_{\rm FF}]^{N=0,\phi=1}\otimes^{\LL_{\Box}}_{\B_e}\B_{\rm pFF})^{\R\sg_K-{\rm sm}},\\
 \R\Gamma_{\hk}(X_C)  & \stackrel{\sim}{\to} ([\R\Gamma_{\proeet}(X_C,{\mathbb B^+})[1/t]]^{\phi=1}\otimes^{\LL_{\Box}}_{\B_e}\B_{\rm pFF})^{\R \sg_K-{\rm sm}}.
\end{align*}
\end{corollary}
\begin{proof} For the first quasi-isomorphism, we can represent $\R\Gamma_{\hk}(X_C)$ by a bounded complex of finite rank $(\phi,N,\sg_K)$-modules over ${C}^{\nr}$ with smooth semilinear \(\sg_K\)-action. Then it suffices to invoke Corollary \ref{derived-FFB}.
The second quasi-isomorphism follows from the first one and Proposition \ref{Guid2}. 
       \end{proof}
\section{Group actions and comparison theorems}

   In this  chapter we 
study the behaviour of de Rham and Hyodo-Kato cohomologies and comparison theorems  in the presence of locally profinite groups acting on analytic varieties. 
      \subsection{Smoothness of  group actions on De Rham and Hyodo-Kato cohomology}       We start by showing that the group action on compactly supported de Rham and Hyodo-Kato cohomologies tends to be smooth.

   \subsubsection{Smooth representations} Before we proceed, let us make  a small digression on smooth representations as defined in \cite[Sec. 5.1]{RR23}. 
    
  Let $G$ be a locally profinite group. Recall that a representation of $G $ on a solid $K$-vector
space $V$ is a map of condensed sets $G \times V \to V$ having the usual properties. If $G$ is compact, the Iwasawa 
algebra of $G$ with coefficients in $K$  is defined as
$$K_{\Box}[G]: =(\lim_{H\subset G}\so_K[G/H])[1/p],
$$
where $H$ runs over all open normal subgroups of $G$.  The category of $G$-representations on solid $K$-vector
spaces is equivalent to the category ${\rm Mod}^{\rm solid}_{K_{\Box}}[G] $ of solid $K_{\Box}[G]$-modules. 
We will write $\sd(K_{\Box}[G])$ for its $\infty$-derived category. Similarly, for a solid $K_{\Box}$-algebra $A$ we write $\sd(A_{\Box}[G])$ for the $\infty$-derived category of solid $A\otimes^{\Box}_KK_{\Box}[G]$-modules. 
  
  % Let $G$ be a locally profinite group. 
Recall that \cite[Prop. 2.2.5]{RR23}, if $G$ is compact, the algebra of smooth $K$-valued distributions of $G$ can be described as 
$$
\sd^{\rm sm}(G,K)\simeq \lim_{H\subset G}K_{\Box}[G/H].
$$
In general, for any compact open subgroup $H\subset G$, we have
$$
\sd^{\rm sm}(G,K)\simeq \sd^{\rm sm}(H,K)\otimes_{K_{\Box}[H]}K_{\Box}[G].
$$
 We quote the following: 
\begin{definition}{\rm (\cite[Def. 5.1.1]{RR23})}
\begin{enumerate}[leftmargin=*]
\item Let $V \in {\rm  Mod}^{\heartsuit}_{K_{\Box}}(\sd^{\rm sm}(G,K))$. The smooth vectors of $V$ are defined by
$$V^{\rm sm}: = \colim_{H\subset G}V^H = \colim_{H\subset G}\underline{\Hom}_{\sd^{\rm sm}(G,K)}(K_{\Box}[G/H], V ),
$$
where $H$ runs over all compact open  subgroups of $G$. We say that $V$ is a {\em smooth representation}
of $G$ if the natural map $ V^{\rm sm} \to V$  is an isomorphism.
\item  We let $(-)^{\R{\rm sm}} : {\rm Mod}_{K_{\Box}}(\sd^{\rm sm}(G,K)) \to {\rm Mod}_{K_{\Box}}(\sd^{\rm sm}(G,K))$ be the functor of derived smooth
vectors
$$V^{\rm \R{\rm sm}}: = \colim_{H\subset G}V^{\rm RH }=\colim_{H\subset G}\R\underline{\Hom}_{\sd^{\rm sm}(G,K)}(K_{\Box}[G/H], V ).
$$
We say that an object in ${\rm Mod}_{K_{\Box}}(\sd^{\rm sm}(G,K)) $ is {\em smooth}  if the natural arrow $V^{\R{\rm sm}}\to  V$ is an
equivalence. We let ${\rm Rep}^{\rm sm}_{K_{\Box}}(G) \subset {\rm Mod}_{K_{\Box}}(\sd^{\rm sm}(G,K))$  be the full subcategory consisting of smooth
objects.
\end{enumerate}
\end{definition}
We will often use the following result: 
\begin{proposition}{\rm(\cite[Prop. 5.1.10]{RR23})}\label{RR1}
An object $V \in{\rm  Mod}_{K_{\Box}}(\sd^{\rm sm}(G,K))$ is smooth if and only if  $H^i(V )$ is smooth,  for
all $i \in\Z$. Hence, the natural $t$-structure of ${\rm Mod}_{K_{\Box}}(\sd^{\rm sm}(G,K))$ induces a $t$-structure on ${\rm Rep}^{\rm sm}_{K_{\Box}}(G)$.
Moreover, ${\rm Rep}^{{\rm sm},\heartsuit}_{K_{\Box}}(G)$  is a Grothendieck abelian category and ${\rm Rep}^{\rm sm}_{K_{\Box}}(G)$ is the derived category of its heart.
\end{proposition}

   \subsubsection{Smoothness of  group actions on De Rham and Hyodo-Kato cohomology}
       We pass now to the group action on  de Rham and Hyodo-Kato cohomologies.  For Hyodo–Kato cohomology, by convention, we will use \(C^{\mathrm{nr}}\)-module objects in the corresponding categories of solid \(\mathbf Q_p\)-representations. In particular, \(\operatorname{Rep}^{\mathrm{sm}}_{C^{\mathrm{nr}}_\Box}(G)\) denotes the category of \(C^{\mathrm{nr}}\)-module objects in \(\operatorname{Rep}^{\mathrm{sm}}_{\mathbf Q_{p,\Box}}(G)\).
\begin{proposition}\label{derham-prop1}
For a locally profinite group $G$ acting (continuously) on a smooth  dagger analytic space $Y$ over $K$, we have 
\begin{enumerate}[leftmargin=*]
\item The $G$-action on $\R\Gamma_{\dr,c}(Y)$ extends to a canonical structure of solid $\sd^{\rm sm}(G,K)$-module, i.e. 
$$\R\Gamma_{\dr,c}(Y)\in {\rm Mod}_{K_{\Box}}( \sd^{\rm sm}(G,K)).$$
\item The $G$-representation $H^i_{\dr,c}(Y)$, $i\geq 0$,  is smooth. 
Hence $\R\Gamma_{\dr,c}(Y)\in{\rm Rep}^{\rm sm}_{K_{\Box}}(G)$.
\end{enumerate}The analogous claims hold for \(\R\Gamma_{\mathrm{HK},c}(Y_C)\), with the coefficient convention above\footnote{See \cite{AGN} for the definition of compactly supported Hyodo-Kato cohomology.}; in particular,
$$\R\Gamma_{\mathrm{HK},c}(Y_C)
\in \operatorname{Rep}^{\mathrm{sm}}_{C^{\mathrm{nr}}_\square}(G).$$
\end{proposition}
\begin{proof}   Let us start with  the first claim. We may assume that $G$ is profinite. Assume first that $Y$ is a dagger affinoid. Since $\R\Gamma_{\dr,c}(Y)$ is (functorially)  dual to $\R\Gamma_{\dr}(Y)$ (see \cite[Th. 5.29]{AGN}),  our claim follows from Lemma \ref{zurich1} below. In the general case, 
 by writing $Y=\cup_{i\in I} Y_i$ for an increasing sequence of $G$-stable  dagger quasi-compact opens $Y_i\subset Y$ we may assume that $Y$ is quasi-compact. Using (co-)Mayer-Vietoris for a (finite) dagger  affinoid covering of each $Y_i$ and passing to an open subgroup that stabilizes every affinoid in the covering, we reduce to the case of a dagger affinoid treated above. By allowing all coverings as above this can be done functorially yielding the first claim of the proposition.

      The first part of the second claim follows from the nonderived version of the argument for the  first claim (we could replace the reference to Lemma \ref{zurich1} by the reference to \cite[Prop. 1.9]{vdP}).  
      The second part of that claim follows now from   Proposition \ref{RR1}.
      
      The Hyodo-Kato claim is proved similarly using  \cite[Th. 5.34]{AGN}. 
\end{proof}
The following result in the case of dagger affinoids  is a $\infty$-categorical version of a result of  van der Put (see \cite[Prop. 1.9]{vdP}). 

\begin{lemma}\label{zurich1} Let $G$ be a locally profinite group. 
 Let $Y$ be a smooth rigid analytic or a dagger affinoid over $K$ with a continuous action of $G$. Then $\R\Gamma_{\dr}(Y)$ has a canonical (functorial in $Y$ and $G$) structure of solid $ \sd^{\rm sm}(G,K)$-module (i.e. object of 
 ${\rm Mod}_{K_{\Box}}( \sd^{\rm sm}(G,K))$) and similarly for $\R\Gamma_{\hk}(Y_C)$. Both complexes are smooth \(G\)-representations in their respective coefficient categories.
\end{lemma}
\begin{proof}  Let us look first at de Rham cohomology.  In the rigid analytic case, the claim follows from the fact that the de Rham cohomology "depends only on the reduction mod $p$ of $Y$". To make this  precise (see \cite[Sec. 4.2.2]{CN3} for the notation and the details), take the natural quasi-isomorphism in $\sd(K_{\Box}[G])$
$$
\R\Gamma_{\dr}(Y)\simeq \R\Gamma_{\rm conv}(Y)
$$
between de Rham cohomology and convergent cohomology. 
Here
$$
\R\Gamma_{\rm conv}(Y):=\colim_{\sy_{\jcdot}\to Y}\R\Gamma_{\rm conv}(\sy_{\jcdot,1})
$$
and the colimit  is over $\eta$-\'etale hypercoverings of $Y$ from $\sm^{\rm ss}_K$ (we may assume that in every degree we have a quasi-compact formal scheme). Hence the connected components of $\sy_n$ are  semistable formal schemes over $\so_L$ for a finite extension $L$ of $K$. As the notation suggests, the convergent cohomology depends only on the reduction mod $p$ of $\sy_n$. 
 It follows that there exists an open  subgroup $G_{n}\subset G$ that acts trivially on $\R\Gamma_{\rm conv}(\sy_{n,1})$. Since the de Rham cohomology of $Y$ is of bounded degree,  we may assume that our simplicial scheme $\sy_{\jcdot}$ is truncated and then there is  an open  subgroup $G_{\sy_{\jcdot}}\subset G$ that acts trivially on $\R\Gamma_{\rm conv}(\sy_{\jcdot,1})$. We used here that we can choose strictly (!) functorial complexes representing convergent cohomology.  It follows that $\R\Gamma_{\rm conv}(\sy_{\jcdot,1})\in {\rm Mod}_{K_{\Box}}(\sd^{\rm sm}(G_{\sy_{\jcdot}},K))$, functorially in the hypercoverings and the actions of the groups. This yields that $\R\Gamma_{\dr}(Y)\in  {\rm Mod}_{K_{\Box}}( \sd^{\rm sm}(G,K))$ and moreover that it is a smooth $G$-representation, as wanted. 

 The dagger case  follows from the rigid analytic case by taking a dagger presentation $\{Y_h\}$ of $Y$ and writing $\R\Gamma_{\dr}(Y)\simeq \colim_h\R\Gamma_{\dr}(Y_h)$. 
 
  For Hyodo-Kato cohomology we argue similarly using the fact that the Hyodo-Kato cohomology "depends only on the reduction mod $p$ of $Y_C$". Namely, by definition, we have 
   in $\sd({C}^{\nr}_{\Box}[G])$
$$
\R\Gamma_{\hk}(Y_C):=\colim_{\sy_{\jcdot}\to Y_C}\R\Gamma_{\hk}(\sy_{\jcdot,1}).
$$
See \cite[Sec. 4.3]{CN3} for details and notation. 
\end{proof}

\subsection{Equivariant properties of comparison theorems} We show here that comparison morphisms behave well with respect to group actions on the underlying variety. We already know that they are functorial hence equivariant for the action of the discrete groups; what we need to show  is that the comparison maps are actually in the relevant topological categories.
\subsubsection{Group action on cohomology of period sheaves}  Let $X$ be an analytic adic space over $C$. Assume that it is equipped with a continuous action of a profinite group $G$.
       We have the following useful corollary of Lemma \ref{HS}:
       \begin{lemma} \label{Gaction} \begin{enumerate}[leftmargin=*]
       \item Let  $\sff$ be as in Lemma \ref{HS}. Then $\R\Gamma_{\proeet}(X,\sff)$ has a natural structure of an object from $\sd(\Q_{p,\Box}[G])$. 
       \item Let $X\in {\rm Sm}_K$ and let $\sff\in\{F^j\so{\mathbb B}_{\dr}\otimes_{\so_X}\Omega^i_X, \so{\mathbb B}_{\dr}\otimes_{\so_X}\Omega_X^i\}$, $i\geq 0, j\in\Z$. Then $\R\Gamma_{\proeet}(X_C,\sff)$ has a natural structure of an object from $\sd(\Q_{p,\Box}[G])$. 
       \end{enumerate}
       \end{lemma}
       \begin{proof} For claim (1),  from Lemma \ref{HS}, we have the quasi-isomorphism:
       \begin{align}\label{dziesiec10}
         \R\Gamma_{\proeet}(X\times G, \sff) & \simeq \R\underline{{\rm Hom}}_{\Q_{p,\Box}}(\Q_{p,\Box}[{G}], \R\Gamma_{\proeet}(X, \sff)).
       \end{align}
       Hence the action map
       $a:      X\times G\to X   $
      together with \eqref{dziesiec10}  yields the action map 
       $$
       a^*: \quad  \Q_{p,\Box}[G]\otimes_{\Q_p}^{\LL_{\Box}}\R\Gamma_{\proeet}(X, \sff)      \to    \R\Gamma_{\proeet}(X, \sff),
       $$
       which satisfies all the higher coherences by the same argument. 
      
 For claim (2),  we can argue as in claim (1)  provided that we know that Lemma \ref{HS} holds for $\sff$. But this can be argued exactly as in the case of ${\mathbb B}^{+}_{\dr},{\mathbb B}_{\dr}$. 
       \end{proof}
\subsubsection{Equivariant properties of comparison theorems}

\begin{proposition} \label{action1} Let $X\in{\rm Sm}_K$ and let $G$ be a profinite group acting continuously on $X$. Then:
\begin{enumerate}[leftmargin=*]
\item The quasi-isomorphisms \eqref{ober1p}, \eqref{ober1pp} and \eqref{ober3}, \eqref{ober2} are $G$-equivariant, i.e., they are in 
$\sd(B_{\Box}[G])$, for  $B= \B^+_{\dr}, \B_{\dr}, \B^+_{\pdr},\B_{\pdr}$, respectively.
\item The quasi-isomorphism \eqref{dark3} is $G$-equivariant, i.e., it is in $\sd(\B_{\Box}[G])$. 
\end{enumerate}
\end{proposition}
\begin{proof}   For the first claim, it suffices to argue for the quasi-isomorphisms \eqref{ober1p}, \eqref{ober1pp}.  We treat the case of $F^r{\mathbb B}_{\dr}$; the other period sheaves could be treated similarly.  

   Recall that we have an exact sequence of sheaves of periods on $X_{C,\proeet}$:
\begin{equation}\label{brig1}
0\to F^r\sbb_{\dr}\to F^{r}\so\sbb_{\dr}\lomapr{\nabla}F^{r-1}\so\sbb_{\dr}\otimes_{\so_X}\Omega^1_X\lomapr{\nabla}\cdots
\end{equation}
The quasi-isomorphism \eqref{ober1p} for $F^r{\mathbb B}_{\dr}$  is   induced by the map
$$F^r\rg_{\dr}(X_C/\B_{\dr})\to \rg_{\proet}(X_C,F^r\sbb_{\dr}),
$$
which in turn,  is induced by the maps 
$$
\nu^*: \quad  \R\Gamma_{\eet}(X_C,F^{r-i}\B_{\dr}\otimes_K\Omega^i_X) \to \R\Gamma_{\proeet}(X_C,F^{r-i}\so\sbb_{\dr}\otimes_{\so_X}\Omega^i_X),
$$
where $\nu: X_{C,\proeet}\to X_{C,\eet}$ is the canonical projection. To check that the map $\nu^*$ is $\B^+_{\dr,\Box}[G]$-equivariant it is enough to check that it is $\Z_p[G]$-equivariant but this is clear.

   For the second claim, we need to work harder. Let $d=\dim X$ and let $r\geq 2d$. The proof of the comparison quasi-isomorphism has two steps (see the proof of \cite[Th. 7.3, Prop. 7.35]{CN4} for details).
   \vskip2mm
  {\bf Step 1} ({\em Pro-\'etale cohomology  of ${\mathbb B}^+$ vs Fontaine-Messing syntomic cohomology})  That passage is basically clear so we will just briefly sketch the argument (see the proof of \cite[Th. 7.3]{CN4} for details). Via a limit-argument we may replace ${\mathbb B}^+$ with ${\mathbb B}_I$, for a compact interval $I\subset (0,\infty)$ with rational end points (assuming functoriality with respect to $I$). We may even assume that $I\subset [1/(p-1),\infty)$ and use Frobenius twists to treat the remaining intervals. 
  We claim that we have the period quasi-isomorphism
  $$
 \alpha_{r,I}:\quad  \R\Gamma_{\proeet}(X_C,{\mathbb B}_I(r))\simeq  \R\Gamma^{\prime}_{\synt}(X_C,\Q_p(r)),
  $$
  where 
  $$
   \R\Gamma^{\prime}_{\synt}(X_C,\Q_p(r)):=[ [\R\Gamma_{\crr}(X_C)\{r\}\otimes^{\LL_{\Box}}_{\B^+_{\crr}}\B^+_{I,\log}]^{N=0}\to \bigoplus_{\Psi(I)}\R\Gamma_{\dr}(X_C/\B_{\dr}^+)/F^r] ,
  $$
with  the index set $\Psi(I)$ running  over zeros of $t$ in ${\rm Spa}(\B_I)$,  is the Frobenius untwisted syntomic cohomology.  This quasi-isomorphism is actually defined modulo  $p^n$ as  a functorial zigzag (see the definition   \cite[(7.34)]{CN4}), where all the maps are quasi-isomorphisms modulo a universal constant.

    We work modulo $p^n$ now. Both domain and target of the period map $\alpha_{r,I}$ are 
   functorial $\Z_{p,\Box}[G]$-modules: see Lemma \ref{Gaction} for the domain. For the target, it suffices to show this locally and functorially, for an \'etale $G$-equivariant map  $U\to X$ with  modulo  $p^n$ coordinate system   as in \cite{CN4} (also $G$-equivariant).  The action of $G$ on this reduced coordinate system  factors through $G/H_U$ for an open subgroup $H_U$.  We represent the map $\alpha_{r,I}$ modulo  $p^n$ via the zigzag
   $$
    \R\Gamma^{\prime}_{\synt}(U_C,\Z/p^n(r))\longrightarrow  M_n(U) \xleftarrow[\sim]{\rm PL} \R\Gamma_{\proeet}(U_C,{\mathbb A}_{I,n}(r)),
   $$
   where the middle term is a "product" of  the de Rham data coming from the coordinates  and the period sheaf data coming from the generic fiber. See the top of the diagrams \cite[(7.31)]{CN4}, \cite[Th. 4.16]{CN1} for details. We equip the middle term with $\Z_{p,\Box}[G]$-structure via the "product" $G$-action. Now map $\alpha_{r,I}$ is obtained by passing to limit over $n$ and inverting $p$. 
   \begin{remark}
   There are two subtle points in the above construction:
   \begin{enumerate}
   \item We use the modulo $p^n$ coordinate systems not the modulo $p^n$ reduction of  the coordinate systems. These systems are equivalent but not with the canonical $G$-actions: in the first case we get a smooth action locally on the syntomic side, in the second case the canonical action of  $G$ permutes distinct labels with identical reduction  so the action need not be continuous.
   \item We use the top of the diagrams defining the comparison instead of the bottom. This means using Fontaine-Messing zigzag instead of Colmez-Nizio{\l} zigzag. They are equivalent with abstract $G$-action but the second zigzag identifies formal Kummer towers with tilted Kummer towers complicating the study of $G$-action. The first one separates the formal coordinates and their Teichm\"uller lifts in the product construction and then identifies them via the Poincar\'e Lemma but only in the derived category. 
   \end{enumerate}
   \end{remark}

  \vskip2mm
  {\bf Step 2} ({\em Fontaine-Messing syntomic cohomology vs Bloch-Kato syntomic cohomology})  This step is  harder and, alas, requires us to dig deeper into the structure of the comparison morphism.   We need to show that the quasi-isomorphism (see \cite[(7.36)]{CN4})
  $$
  [\R\Gamma_{\crr}(X_C)\otimes^{\LL_{\Box}}_{\B^+_{\crr}}\B^+_{I,\log}]^{N=0}\simeq [\R\Gamma_{\hk}(X_C)\otimes^{\LL_{\Box}}_{{C}^{\nr}}\B^+_{I,\log}]^{N=0}
  $$
is $\Q_{p,\Box}[G]$-equivariant.  Again, assuming enough functoriality, one can argue  \'etale locally on $X$. That is, we need to show that the Hyodo-Kato section (see \cite[Sec. 2]{CN4} for the notation)
$$
\iota_l: \R\Gamma_{\hk}(\sx^0_1)_{\Q}\to \R\Gamma_{\crr}(\sx_1/R)_{l,\Q}
$$
of the canonical projection 
\begin{equation}\label{projfirst}
\R\Gamma_{\crr}(\sx_1/R)_{l,\Q}\to \R\Gamma_{\hk}(\sx^0_1)_{\Q}
\end{equation}
 is $\Q_{p,\Box}[G]$-equivariant.  Here $\sx_1$ is a semistable formal scheme over $\so_{K,1}$ equipped with a $G$-action.
 
 We note  that the action of $G$ on the domain and the target of the projection \eqref{projfirst} factors through $G/H_{\sx}$, where $H_{\sx}$ is an open subgroup of $G$ acting trivially on $\sx_1$. Hence it is enough to show that
  the map $\iota_l$ is  $\Z_{p,\Box}[G/H_{\sx}]$-equivariant in a functorial way. 
  
  But for this we can argue as in the proof of \cite[Prop. 2.14]{CN4}, where one had to show that the Hyodo-Kato section, a priori just Frobenius-equivariant by \cite[Th. 2.12]{CN4}, is actually also compatible with monodromy.  Then one needed to lift the monodromy action to an action of ${\mathbb G}_m^{\natural}$ on $\sx_1$ and run the argument for the Frobenius-equivariance from the proof of \cite[Th. 2.12]{CN4} again but for the simplicial variety representing $[\sx_1/{\mathbb G}_m^{\natural}]$.  Here the relevant simplicial variety  is $[\sx_1/(G/H_{\sx})]$, hence even simpler since $G/H_{\sx}$ is finite. So the same argument works but is now basically trivial.

\end{proof}
\subsubsection{Local analyticity} Let $X$ be a smooth rigid analytic variety over $K$. Let $I\subset (0,\infty)$ be a compact interval with rational endpoints and let $\Psi(I)$ be the set of zeros of  $t$  on ${\rm Spa}(\B_I)$. Then, by the proof of Proposition \ref{action1},  for $r\geq 2d$,  
we have  the following quasi-isomorphism 
\begin{equation}\label{CGNdrLum}
\rg_{\proet}(X_C,\sbb_I)  \simeq \big[[\rg_{\rm HK}(X_C)\{r\}\otimes^{\LL_{\Box}}_{{C}^{\nr}}\B_{I,\log}]^{N=0}\stackrel{\iota_{\hk}}{\longrightarrow} \bigoplus_{\Psi(I)}\R\Gamma_{\dr}(X_C,r)\big](-r).
\end{equation}
Moreover, if a  profinite group $G$ acts continuously on $X$  this quasi-isomorphism  is in $\sd(\Q_{p,\Box}[G])$. The following result  was proved in \cite[Cor.1.1]{RD}. Our argument using the quasi-isomorphism \eqref{CGNdrLum} yields a conceptual explanation of this phenomenon. 

\begin{corollary}  Assume that $G$ is a compact $p$-adic Lie group and $X\in {\rm Sm}_K$ is quasi-compact. Then the solid $G$-representation $\rg_{\proet}(X_C,\sbb_I) $ is $G$-locally analytic.  If $t$ is a unit in $\mathrm{Spa}(B_I)$ this representation is $G$-smooth. 
\end{corollary}
\begin{proof} We use the quasi-isomorphism \eqref{CGNdrLum}. But the action on the de Rham part  of the right-hand side of \eqref{CGNdrLum} is $G$-locally analytic and, by Lemma \ref{zurich1},  the action  on Hyodo-Kato part is actually $G$-smooth. 
\end{proof}
\section{Representation theoretic input}

   In this  short chapter we state and prove some basic representation-theoretic results which will be used later on in the proof of the main results.  
\subsection{Stable cohomology}

    One could obtain more refined versions of the following, but the one below suffices for our purposes.

    \begin{proposition}\label{stcoh}
    Let $G$ be a $p$-adic Lie group over $\mathbf Q_p$, with
Lie algebra $\mathfrak g$, and let
$\pi\in\operatorname{Rep}^{\mathrm{sm}}_{K_\Box}(G)$
   be a complex of smooth solid representations of $G$.
   There is a natural isomorphism (the colimit being taken over the compact open subgroups $H$ of $G$)
   $$\colim_{H} \R\Gamma(H, \pi)\simeq (\colim_{H} \R\Gamma(H, K))\otimes^{\LL_{\Box}}_{K} \pi.$$
This is further naturally isomorphic to $\R\Gamma(\mathfrak{g}, K)\otimes^{\LL_{\Box}}_{K} \pi$.    
    \end{proposition}
     \begin{proof} The second claim follows directly from the first one and classical results of Lazard.
For the first one, we may restrict the colimit to uniform pro-$p$ compact open
subgroups $H\subset G$.
By \cite[Prop.~2.2.1]{RR23}, after extension of scalars
to $K$, the trivial $K_\Box[H]$-module $K$ admits a bounded
resolution $P_H^\bullet$ by finite free $K_\Box[H]$-modules.
It follows that $\R\Gamma(H,-)$ commutes with filtered colimits and, for any $V\in \sd(K_\Box)$ with trivial $H$-action,
the canonical projection morphism is a quasi-isomorphism:
\begin{equation}\label{add1}
  \R\Gamma(H,K)\otimes_K^{\LL_\Box}V
  \xrightarrow{\sim}\R\Gamma(H,V).
\end{equation}
Indeed, this follows termwise by computing both sides using
$P_H^\bullet$.

  By d\'evissage we may assume that $\pi$ is in the heart of ${\rm Rep}^{\rm sm}_{K_{\Box}}(G)$.  By writing 
     $\pi=\colim_{H} \pi^H$ we are thus reduced to the case when some uniform pro-$p$ compact open subgroup $H$ of $G$ acts trivially on $\pi$. But then 
     $\R\Gamma(H,\pi)\simeq \R\Gamma(H, K) \otimes^{\LL_{\Box}}_{K} \pi$ by \eqref{add1} and we are done.
        \end{proof}
        \begin{remark}\label{res1}
        With the convention of Section 4.1.2, the same statement holds with \(C^{\mathrm{nr}}\)-coefficients, by applying the proposition over \(\mathbf Q_p\).
Explicitly, for
\(\pi\in\operatorname{Rep}^{\mathrm{sm}}_{C^{\mathrm{nr}}_\Box}(G)\), we have 
\[
\varinjlim_H \R\Gamma(H,\pi)
\simeq
\R\Gamma(\mathfrak g,\mathbf Q_p)
 \otimes_{\mathbf Q_p}^{\LL_{\Box}}\pi
\simeq
\R\Gamma(\mathfrak g,C^{\mathrm{nr}})
 \otimes_{C^{\mathrm{nr}}}^{\LL_{\Box}}\pi.
\]
        \end{remark}
   \subsection{Cancellation results}  
   
    Let $G$ be a locally profinite group. We say that $\pi\in {\rm Rep}^{\rm sm}_{K_{\Box}}(G)$ is \emph{admissible} if $\pi$ is concentrated in degree $0$ and if, 
   for all compact open subgroups $H$ of $G$, the solid $K$-vector space $\pi^H$ is of the form $\underline{V}$ for some finite dimensional $K$-vector space 
    $V$ (endowed with its natural topology). This is equivalent to saying that $\pi=\underline{\sigma}$ for an admissible smooth (in the usual sense) $K$-representation $\sigma$ of $G$ (endowed with its natural inductive limit topology). 
    
     We use the analogous notion for smooth representations over \(C^{\mathrm{nr}}_\Box\): admissibility means that the representation is concentrated in degree zero and that its invariants under every compact open subgroup are finite free \(C^{\mathrm{nr}}_\Box\)-modules.
    \begin{remark}
    The above definition is ad-hoc, but suitable for our applications. The "right" definition would be to drop the assumption that $\pi$ is concentrated in degree $0$, impose that $H^i(\pi)$ is admissible in the above sense for all $i$, and moreover that for each compact open subgroup $H$ of $G$ 
    $(H^i(\pi))^H$ vanishes for all but finitely many $i$. 
    \end{remark}

   \begin{proposition}\label{cancel1}
    Let $G$ be a profinite group and let 
    $\pi, \pi_1,\pi_2\in {\rm Rep}^{\rm sm}_{K_{\Box}}(G)$ be admissible representations of $G$ such that 
 $$
  \pi_1\oplus \pi \simeq \pi_2\oplus \pi
$$
  Then $\pi_1\simeq \pi_2$.
   \end{proposition}
   
   \begin{proof}
   We may assume that $\pi, \pi_1, \pi_2$ are admissible smooth non-condensed representations. Since $G$ is profinite, $\pi, \pi_1, \pi_2$ are semi-simple and each irreducible smooth representation $\sigma$ of $G$ is finite dimensional and has finite multiplicity in $\pi, \pi_1,\pi_2$. Since $\pi_1\oplus \pi \simeq \pi_2\oplus \pi
$, the multiplicities of $\sigma$ in $\pi_1$ and $\pi_2$ must be equal, so $\pi_1\simeq \pi_2$.
   \end{proof}
   
      While the above result is very elementary, the next one requires more serious input: 
   
    \begin{proposition}\label{rep1HK} Let $L$ be a finite extension of $\Q_p$. 
  Let $G$ be the group of $L$-rational points of a reductive group over $L$. 
  
  a) If 
  $\pi, \pi_1,\pi_2\in {\rm Rep}^{\rm sm}_{K_{\Box}}(G)$ are admissible representations of $G$
  such that $\pi\oplus \pi_1\simeq \pi\oplus \pi_2$, then 
  $\pi_1\simeq \pi_2$.

  b) Let \(\Gamma\) be a profinite group equipped with a continuous homomorphism
$\Gamma\longrightarrow\operatorname{Gal}(C^{\mathrm{nr}}/F),$
and let \(\pi,\pi_1,\pi_2\) be admissible smooth \(G\)-representations on solid \((\varphi,N,\Gamma)\)-modules\footnote{Here $\varphi$ is invertible and semilinear for the
absolute Frobenius of $C^{\nr}$, $N$ is $C^{\nr}$-linear
and satisfies $N\varphi=p\varphi N$, and $\Gamma$ acts
smoothly and $C^{\nr}$-semilinearly, commuting with
$\varphi$ and $N$. The $G$-actions commute with all these
structures.}
 over \(C^{\mathrm{nr}}\).
If $  \pi\oplus\pi_1\simeq\pi\oplus\pi_2 $
as $G$-representations on $(\varphi,N,\Gamma)$-modules,
then $\pi_1\simeq\pi_2$ in the same category.
  \end{proposition}
  
   \begin{proof}
  a) Using the Bernstein decomposition of the category of smooth representations of $G$, we may assume that 
  $\pi, \pi_1,\pi_2$ all belong to a block of this category. Bernstein also proved (see  \cite[Prop. VI.10.6]{Renard}) the existence of a compact open subgroup 
  ${H}$ of $G$ so that the functor $V\mapsto V^H$ defined on the block, with values in modules over the associated Hecke algebra 
  $H(G,H)$ is fully faithful. Thus it suffices to show that $\pi_1^H\simeq \pi_2^H$ as $H(G,H)$-modules. But 
  $\pi^H\oplus \pi_1^H\simeq \pi^H\oplus \pi_2^H$ as $H(G,H)$-modules. By assumption 
  $\pi^H, \pi_1^H, \pi_2^H$ are finite dimensional over the field of coefficients, thus artinian and noetherian
  $H(G,H)$-modules. The Krull-Schmidt theorem therefore applies and shows that 
  $\pi_1^H\simeq \pi_2^H$ as $H(G,H)$-modules, finishing the proof.
  
   b) Let $\mathcal A_\Gamma$ be the category of finite-rank
$(\varphi,N,\Gamma)$-modules over $C^{\nr}$.
This is an essentially small additive $\mathbf Q_p$-linear
category with finite-dimensional Hom spaces. 
Let $\mathcal O_G$ be the category whose objects are
$G/H$, for compact open subgroups $H\subset G$, and whose
morphisms are $G$-equivariant maps.
To each of the representations
$\rho=\pi,\pi_1,\pi_2$, we associate the diagram
\[
  \mathcal D_\rho:\mathcal O_G^{\mathrm{op}}
       \longrightarrow\mathcal A_\Gamma,
  \qquad
  G/H\longmapsto\rho^H.
\]
Its values have finite rank by admissibility and are
stable under $\varphi$, $N$, and $\Gamma$, since these
structures commute with the $G$-action.

The given isomorphism induces an isomorphism of diagrams
$ \mathcal D_\pi\oplus\mathcal D_{\pi_1}
  \simeq
  \mathcal D_\pi\oplus\mathcal D_{\pi_2}.$
By Lemma~\ref{berkeley1W}, we obtain $
  \mathcal D_{\pi_1}\simeq\mathcal D_{\pi_2}.$
In particular, there are isomorphisms $  \alpha_H:\pi_1^H\xrightarrow{\sim}\pi_2^H$
compatible with all orbit-category maps.
Since $  \pi_j
  \simeq\operatorname*{colim}_{H}\pi_j^H,  j=1,2,$
the maps $\alpha_H$  induce an isomorphism
$\alpha:\pi_1\xrightarrow{\sim}\pi_2$, commuting with
$\varphi$, $N$, and $\Gamma$.
Moreover, compatibility with the orbit-category maps gives
$  \alpha_{gHg^{-1}}(gv)=g\alpha_H(v)$ for $
  g\in G,\ v\in\pi_1^H,$
so $\alpha$ is $G$-equivariant.
      \end{proof}
\subsubsection{Cancellation of diagrams} The following lemma will allow us to lift the above cancellation results to more general settings. 
\begin{lemma}\label{berkeley1W}
Let $L$ be a field, let $\mathcal A$ be an essentially small
additive $L$-linear category, and assume that
$\operatorname{Hom}_{\mathcal A}(A,B)$ is finite-dimensional
over $L$ for all $A,B\in\mathcal A$.
For every small category $\mathcal I$, cancellation holds in
$\operatorname{Fun}(\mathcal I,\mathcal A)$: if
$V,W,Q:\mathcal I\to\mathcal A$ satisfy $V\oplus Q\simeq W\oplus Q, $
then $V\simeq W$ as diagrams.
\end{lemma}

\begin{proof}
First assume  that $\mathcal A=\operatorname{Vect}^{\mathrm{fd}}_L$.
By \cite[Th.~1.1]{BCB}, every pointwise finite-dimensional
diagram of $L$-vector spaces is a direct sum of indecomposable
diagrams with local endomorphism rings. Such a decomposition
is unique up to permutation and isomorphism of the summands
by the Krull--Remak--Schmidt--Azumaya theorem
(see also \cite[Section~2]{BCB}).
Each nonzero indecomposable has finite multiplicity in a
pointwise finite-dimensional diagram.
Hence, the multiplicities contributed by $Q$ can be
cancelled, giving $V\simeq W$.

For general $\mathcal A$, choose a small skeleton
$\mathcal A_0\subset\mathcal A$. To a diagram
$V:\mathcal I\to\mathcal A$ associate
\[
  \widetilde V:
  \mathcal I\times\mathcal A_0^{\mathrm{op}}
  \longrightarrow\operatorname{Vect}^{\mathrm{fd}}_L,
  \qquad
  (i,A)\longmapsto\operatorname{Hom}_{\mathcal A}(A,V(i)).
\]
This construction preserves finite direct sums. Thus the
first case gives $\widetilde V\simeq\widetilde W$.
By Yoneda\footnote{By Yoneda, $X\mapsto \Hom_{\mathcal A}(-, X)|_{A_0}$ is a fully faithful functor $A \to  \mathrm{Fun}(A_0^{\rm op}, {\rm Vect}_L)$, since every object of $A$ is isomorphic to one in $A_0$. Hence $V \mapsto  \wt{V}$, from $\mathrm{Fun}(I, A) $ to $\mathrm{Fun}(I\times A_0^{\rm op}, \mathrm{Vect}_L)$, is fully faithful, so it reflects isomorphisms.}, this isomorphism comes from an isomorphism
$V\simeq W$, and naturality in $i$ makes it an isomorphism
of diagrams. 
\end{proof}

\section{Flip-flopping} \label{five} We will now state and prove de Rham and Hyodo-Kato flip-flopping. For the set-up, 
consider a diagram
$$
\xymatrix{
& T\ar[dl]^-{{G}}_{\pi}\ar@(u,r)[]^{G\times\check{G}} \ar[dr]_{\check{G}}^{\check{\pi}}\\
X=[T/G]\ar@(l,u)[]^{\check{G}} & &   \check{X}=[T/\check{G}]\ar@(u,r)[]^{{{G}}} 
 }
$$
where
\begin{enumerate}[leftmargin=*]
\item[(a)] $X$ and $\check{X}$ are smooth rigid analytic spaces over $K$; $T$ is a diamond over $K$;
\item[(b)] $G$ and $\check{G}$ are profinite groups;
\item[(c)] $G\times\check{G}$ acts on  $T$ continuously and the map $\pi: T\to X$ (respectively $\check{\pi}: T\to \check{X}$) is a pro-\'etale $G$-torsor (respectively a pro-\'etale $\check{G}$-torsor). 
\end{enumerate}
  \subsection{De Rham flip-flopping} In this section we collect a number of de Rham flip-flopping results. 
  \subsubsection{Usual de Rham cohomology}
  The first main theorem of this paper  is the following: 
   \begin{theorem}{\rm ({\em De Rham flip-flopping})} \label{derham1} Assume that $X, \check{X}$ are partially proper. 
\begin{enumerate}[leftmargin=*]
\item There are  
natural quasi-isomorphisms in $\sd(K_{\Box})$ ($r\in \Z$):
\begin{equation}\label{berkeley2}
\rg(\check{G}, F^r\R\Gamma_{\dr}(X))  \simeq \rg({G}, F^r\R\Gamma_{\dr}(\check{X})).
\end{equation}

\item There are natural quasi-isomorphisms in $\sd({\B^+_{\dr,\Box}})$ and $\sd({\B_{\dr,\Box}})$, respectively:
\begin{equation}\label{berkeley1}
\rg(\check{G}, F^r\R\Gamma_{\dr}(X_C/\B_{\dr}))  \simeq \rg({G}, F^r\R\Gamma_{\dr}(\check{X}_C/\B_{\dr})), \quad r\in\Z,
\end{equation}
and
\begin{equation}
\label{berkeley11}\rg(\check{G},\R\Gamma_{\dr}(X_C/\B_{\dr}))  \simeq \rg({G},\R\Gamma_{\dr}(\check{X}_C/\B_{\dr})).
\end{equation}
\end{enumerate}
\end{theorem}
\begin{proof} Using the natural quasi-isomorphisms (see Lemma \ref{GBcoh}) 
\begin{align*}\R\Gamma(\check{G},  \R\Gamma_{\dr}(X_C/\B_{\dr})) & \simeq \R\Gamma(\check{G}, \R\Gamma_{\dr}(X))\otimes^{\LL_\Box}_K \B_{\dr},\\
 \R\Gamma(\check{G},  F^r\R\Gamma_{\dr}(X_C/\B_{\dr})) & \simeq \colim_{j+k\geq r} \R\Gamma(\check{G}, F^j \R\Gamma_{\dr}(X))\otimes^{\LL_\Box}_K F^k\B_{\dr},
 \end{align*}
 we see that the theorem follows once we construct the natural quasi-isomorphisms in \eqref{berkeley2}.

First, suppose that
$X$ and $\check{X}$ are quasi-compact. Then Corollary \ref{quotG} and Galois descent (see Lemma \ref{HS1}) yield natural quasi-isomorphisms in $\sd(K_{\Box})$
\begin{align*}\rg(\check{G}, F^r\R\Gamma_{\dr}(X))   & \simeq \rg(\check{G}\times \sg_K, \R\Gamma_{\proet}(X_C, F^r \sbb_{\pdr}))
  \simeq \rg(\check{G}\times \sg_K, \rg(G, \R\Gamma_{\proet}(T_C, F^r \sbb_{\pdr})))\\
   & \simeq \rg(G\times\check{G}\times \sg_K,  
 \R\Gamma_{\proet}(T_C, F^r \sbb_{\pdr}))
\end{align*}
and similarly with the roles of $G$ and $\check{G}$ (as well as $X$ and $\check{X}$) exchanged. Combining these quasi-isomorphisms we obtain the desired relation \eqref{berkeley2} when $X$ and $\check{X}$ are quasi-compact. 
 
 To treat the general case of \eqref{berkeley2}, let $\{X_n\}$, $n\in\N$, be an increasing covering of  $X$ by $\check{G}$-stable quasi-compact opens. Since $\check{G}$ is compact and $X$ is locally Stein (because $X$ is partially proper), such a covering always exists 
 by \cite[Lemma 2.2]{LT}. Moreover, if $X$ is Stein we can even take the $X_n$'s to be affinoid and the covering can be chosen to be strictly increasing. Let $\{\check{X}_n\}$,
     $n\in\N$, be the corresponding covering of $\check{X}$, i.e., $\check{X}_n:=\check{\pi}(\pi^{-1}(X_n))$. This is an increasing covering of $\check{X}$ by ${G}$-stable quasi-compact opens (again, strictly increasing if necessary). We compute  in $\sd(K_{\Box})$:  
  \begin{align*}
   \rg(\check{G},F^r\rg_{\dr}(X))&  \stackrel{\sim}{\to}\rg(\check{G},\R\lim_nF^r\rg_{\dr}(X_{n})) \stackrel{\sim}{\to}
   \R\lim_n\rg(\check{G},F^r\rg_{\dr}(X_{n}))\\
&  \simeq   \R\lim_n\rg({G},F^r\rg_{\dr}(\check{X}_{n}))\stackrel{\sim}{\leftarrow} \rg({G},\R\lim_nF^r\rg_{\dr}(\check{X}_{n})) \stackrel{\sim}{\leftarrow} \rg({G},F^r\rg_{\dr}(\check{X})).
  \end{align*}
This finishes the proof of \eqref{berkeley2} and our theorem.
\end{proof}
  \subsubsection{Compactly supported de Rham cohomology} Theorem \ref{derham1} has a version for de Rham cohomology with compact support. 
\begin{corollary}{\rm ({\em Compactly supported de Rham flip-flopping})}\label{bonn15} Assume that $X, \check{X}$ are partially proper. 
\begin{enumerate}[leftmargin=*]
\item  There are natural quasi-isomorphisms in $\sd({\B^+_{\dr,\Box}})$
\begin{equation}\label{berkeley1c}
\rg(\check{G},F^r\R\Gamma_{\dr,c}(X_C/\B_{\dr}))  \simeq \rg({G},F^r\R\Gamma_{\dr,c}(\check{X}_C/\B_{\dr})),\quad r\in\Z.
\end{equation}
\item There are
natural quasi-isomorphisms in $\sd(K_{\Box})$
\begin{equation}\label{berkeley2c}
\rg(\check{G},F^r\R\Gamma_{\dr,c}(X))  \simeq \rg({G},F^r\R\Gamma_{\dr,c}(\check{X})),\quad r\in\Z.
\end{equation}
\end{enumerate}
\end{corollary}
\begin{proof} 
Let $r\in\Z$.  We compute  in  $\sd({\B^+_{\dr,\Box}})$ using the quasi-isomorphism \eqref{berkeley1}:
 \begin{align*}
 \rg(\check{G},F^r\R\Gamma_{\dr,c}(X_C/\B_{\dr})) & \simeq \rg(\check{G},[F^r\R\Gamma_{\dr}(X_C/\B_{\dr})\to F^r\R\Gamma_{\dr}(\partial X_C/\B_{\dr})])\\
  & \simeq [\rg(\check{G},F^r\R\Gamma_{\dr}(X_C/\B_{\dr}))\to \rg(\check{G},F^r\R\Gamma_{\dr}(\partial X_C/\B_{\dr}))]\\
  & \simeq [\rg({G},F^r\R\Gamma_{\dr}(\check{X}_C/\B_{\dr}))\to \rg({G},F^r\R\Gamma_{\dr}(\partial \check{X}_C/\B_{\dr}))]\\
   &\simeq \rg({G},F^r\R\Gamma_{\dr,c}(\check{X}_C/\B_{\dr})).
 \end{align*}
Here we used the fact that  in  $\sd({\B^+_{\dr,\Box}})$
\begin{align*}
\rg(\check{G},F^r\R\Gamma_{\dr}(\partial X_C/\B_{\dr})) & \simeq  \rg(\check{G},\colim_n F^r\R\Gamma_{\dr}((X\setminus X_{n})_C/\B_{\dr}))\\
   & \simeq \colim_n\rg(\check{G},F^r\R\Gamma_{\dr}((X\setminus X_{n})_C/\B_{\dr})),
\end{align*}
 where the covering $\{X_n\}$ is chosen as in the proof of Theorem \ref{derham1}. The second quasi-isomorphism follows from the fact that $\Z_{\Box}[\check{G}]$, $\Z_{\Box}[{G}]$ are  internally compact objects in solid abelian groups. 
And similarly for the pair $G, \check{X}$ and the compatible covering $\{\check{X}_n\}$.  This proves claim (1) of our corollary. 
   
   The above argument can be easily adapted to yield quasi-isomorphism \eqref{berkeley2c} using quasi-isomorphism  \eqref{berkeley2}, hence proving claim  (2) of our corollary.  
  \end{proof}
 \subsubsection{Hodge and Lie algebra cohomologies}
    The following two results are simple corollaries of Theorem \ref{derham1} and Corollary \ref{bonn15}. 

\begin{corollary}{\rm ({\em Hodge  flip-flopping})}\label{IAS11} Assume that $X, \check{X}$ are Stein   of pure dimension $d$.  Let $j\in\N$.
  We have  natural quasi-isomorphisms in $\sd(K_{\Box})$
\begin{align*}\R\Gamma(\check{G},\Omega^j(X) )\simeq \R\Gamma(G, \Omega^j(\check{X})),\quad 
 \R\Gamma(\check{G}, H^d_c(X,\Omega^j) ) \simeq \R\Gamma(G, H^d_c(\check{X},\Omega^j)).
\end{align*}
\end{corollary}
\begin{proof} Since $\gr^j_F\R\Gamma_{\dr}(Y)\simeq \Omega^j(Y)[-j]$, for $Y=X,\check{X}$, the first quasi-isomorphism follows from the quasi-isomorphism \eqref{berkeley2}. Since $\gr^j_F\R\Gamma_{\dr,c}(Y)\simeq H^d_c(Y,\Omega^j)[-d-j]$ (see \eqref{drwinter}), the second 
quasi-isomorphism follows from the quasi-isomorphism \eqref{berkeley2c}. 
\end{proof}
\begin{remark}
 One might be tempted to prove the above corollary using the easier period sheaf $\hat{\mathcal{O}}$ and the isomorphism 
 $\R^i \nu_* \hat{\mathcal{O}}\simeq \Omega^i(-i)$, where $\nu: X_{C, \proet}\to X_{C,\eet}$ is the natural projection. The spectral sequences that arise when passing to cohomology are rather complicated, but as the referee observed there is a nice way to get around this difficulty. Taking limits as above, we may assume that $X$ is quasi-compact. The Postnikov layers of $\R \nu_* \hat{\mathcal{O}}$ are $\Omega^j_{X_C}(-j)[-j]$. Once we adjoin 
 $\log t$ and twist by $d$, Remark \ref{Galois-zurich} allows one to kill every layer except $j=d$ and yields (using also Proposition \ref{magic})
 \[
\R\Gamma\left(\sg_K,
  \R\Gamma_{\proet}
    (X_C,\widehat{\mathcal O}[\log t](d))\right)
\simeq \R\Gamma(X,\Omega_X^d)[-d]
\]
for all $d$, which allows one to conclude by the same argument as above.
\end{remark}

 We set in $\sd(K_{\Box})$
$$
\R\Gamma_{\dr,c}(X_{\infty})  :=\colim_{H}\R\Gamma_{\dr,c}(X_{H}),\quad 
\R\Gamma_{\dr,c}(\check{X}_{\infty}) :=\colim_{\check{H}}\R\Gamma_{\dr,c}(\check{X}_{\check{H}}),
$$
where $H, \check{H}$ are normal compact open subgroups of $G, \check{G}$, respectively, and $X_H, \check{X}_{\check{H}}$ are the respective quotients of $T$ (representable in rigid analytic spaces). 
 Let us assume that $X$ and $\check{X}$
    are partially proper. Then so are all $X_H, \check{X}_{\check{H}}$ (as they are finite \'etale over $X$ and $\check{X}$ respectively). By Proposition \ref{derham-prop1} we have $\R\Gamma_{\dr,c}(X_{H})\in {\rm Rep}^{\rm sm}_{K_{\Box}}(\check{G})$. The natural action of the finite group $G/H$ on $X_H$ over $X$ then shows that 
    $\R\Gamma_{\dr,c}(X_{H})\in {\rm Rep}^{\rm sm}_{K_{\Box}}(G\times\check{G})$ and so $\R\Gamma_{\dr,c}(X_{\infty}) \in {\rm Rep}^{\rm sm}_{K_{\Box}}(G\times\check{G})$. Similarly, $\R\Gamma_{\dr,c}(\check{X}_{\infty})\in   {\rm Rep}^{\rm sm}_{K_{\Box}}(G\times\check{G})$.
    
\begin{corollary}{\rm(Lie algebra  flip-flopping)} \label{dom1} Assume that $X, \check{X}$ are partially proper  and that  $G$, $\check{G}$ are compact $p$-adic Lie groups over $\Q_p$. Let $\mathfrak{g}, \check{\mathfrak{g}}$  be their respective Lie algebras.  
There are $G\times\check{G}$-equivariant natural quasi-isomorphisms  in $\sd(K_{\Box})$
\begin{align}\label{paris12}
\R\Gamma(\check{\mathfrak{g}},\R\Gamma_{\dr,c}(X_{\infty})) & \simeq \R\Gamma(\mathfrak{g},\R\Gamma_{\dr,c}(\check{X}_{\infty})),\\
\R\Gamma(\check{\mathfrak{g}}, K)\otimes_K^{\LL_{\Box}}\R\Gamma_{\dr,c}(X_{\infty}) & \simeq \R\Gamma(\mathfrak{g},K)\otimes^{\LL_{\Box}}_K\R\Gamma_{\dr,c}(\check{X}_{\infty}).\notag
\end{align}
\end{corollary}
\begin{remark}
The reader will notice that we do not claim that the quasi-isomorphism in the Corollary is filtered. Indeed, a key property used in its proof below is the smoothness of the total compactly supported de Rham complex. On the other hand, the filtered pieces need not be smooth.
\end{remark}
\begin{proof} The second quasi-isomorphism follows from the first one because the complexes $\R\Gamma_{\dr,c}(X_{\infty})$ and $\R\Gamma_{\dr,c}(\check{X}_{\infty})$ are in ${\rm Rep}^{\rm sm}_{K_{\Box}}(G\times\check{G})$, hence, 
 by Proposition \ref{RR1}, can be assumed to be represented by complexes of smooth representations (thus the respective actions of 
$\mathfrak{g}$ and $\check{\mathfrak{g}}$ are trivial).

  For the first quasi-isomorphism,
 by Proposition \ref{stcoh},  we have natural  (hence $G\times\check{G}$-equivariant) quasi-isomorphisms in $\sd(K_{\Box})$
$$\R\Gamma(\check{\mathfrak{g}},\R\Gamma_{\dr,c}(X_{\infty})) \simeq \colim_{\check{H}}R\Gamma(\check{H},\R\Gamma_{\dr,c}(X_{\infty}))\simeq \colim_{\check{H}}\colim_H R\Gamma(\check{H},\R\Gamma_{\dr,c}(X_{H}))$$
and similarly
$$\R\Gamma({\mathfrak{g}},\R\Gamma_{\dr,c}(\check{X}_{\infty}))  \simeq \colim_{{H}}\colim_{\check{H}}\R\Gamma({H},\R\Gamma_{\dr,c}(\check{X}_{\check{H}})).$$
Since, by \eqref{berkeley2c}, we have  natural quasi-isomorphisms $\R\Gamma(\check{H},\R\Gamma_{\dr,c}(X_{H}))\simeq \R\Gamma({H},\R\Gamma_{\dr,c}(\check{X}_{\check{H}}))$, the result follows. 
 \end{proof}
\subsection{Hyodo-Kato  flip-flopping} 
Our second main theorem is the following: 
\begin{theorem}{\rm ({\em Hyodo-Kato flip-flopping})}  \label{HKFF} Let $X$, $\check{X}$ be  partially proper. 
 We have   natural  $\sg_K$-equivariant quasi-isomorphisms in $\sd^{\prime}_{\phi,N}({{C}^{\nr}_{\Box}})$
\begin{align}\label{berkeley1HK}\rg(\check{G},\R\Gamma_{\hk}(X_C))  & \simeq \rg({G},\R\Gamma_{\hk}(\check{X}_C)),\\
\rg(\check{G},\R\Gamma_{\hk,c}(X_C))  & \simeq \rg({G},\R\Gamma_{\hk,c}(\check{X}_C)).\label{berkeley4HK}
\end{align}
\end{theorem}
\begin{proof}   We will first prove a flip-flopping quasi-isomorphism between  $\sbb^+$-cohomologies (see \eqref{kicia1HK} below).
Since $\rg_{\proet}(-,\sbb^+)$ satisfies Galois descent (see Lemma \ref{HS1}),  we have natural  quasi-isomorphisms in $\sd_{\phi}({\B^+_{\Box}})$ compatible with  
the  $\sg_K$-action: 
\begin{align*}
 \rg({G},\rg_{\proet}(T_C,\sbb^+))  &  \simeq \rg_{\proet}(X_C,\sbb^+),\\
 \rg(\check{G},\rg_{\proet}(T_C,\sbb^+))   & \simeq \rg_{\proet}(\check{X}_C,\sbb^+),\\
  \rg({G},\rg_{\proet}(\check{X}_C,\sbb^+)) & \simeq \rg(G\times\check{G},\rg_{\proet}(T_C,\sbb^+))   \simeq \rg(\check{G},\rg_{\proet}(X_C,\sbb^+)).
 \end{align*}
In particular, we have a natural  quasi-isomorphism in $\sd_{\phi}({\B^+_{\Box}})$
\begin{equation}
\label{kicia1HK}
\rg(\check{G},\rg_{\proet}(X_C,\sbb^+))\simeq \rg({G},\rg_{\proet}(\check{X}_C,\sbb^+)),
\end{equation}
which is  compatible with the  $\sg_K$-action.  It yields a natural  quasi-isomorphism in $\sd_{\phi}({\B_{\Box}})$
\begin{equation}
\label{kicia11HK}
\rg(\check{G},\rg_{\proet}(X_C,\sbb^+)\otimes^{\LL_{\Box}}_{\B^+}\B)\simeq \rg({G},\rg_{\proet}(\check{X}_C,\sbb^+)\otimes^{\LL_{\Box}}_{\B^+}\B),
\end{equation}
which is  compatible with the  $\sg_K$-action.

     Now we pass to the  Hyodo-Kato flip-flopping. Assume first that the de Rham cohomology of $X, \check{X}$ is of finite rank. Set 
 $Z=X,\check{X}$.  
Combining quasi-isomorphism from Proposition \ref{Guid2}  with the quasi-isomorphism \eqref{kicia11HK}, we get a natural  $\sg_K$-equivariant quasi-isomorphism in $\sd_{\phi}(\B_{\Box})$ 
$$
\rg(\check{G},[\rg_{\hk}(X_C)\otimes^{\LL_{\Box}}_{{C}^{\nr}}\B_{\rm FF}]^{N=0})\simeq \rg(G,[\rg_{\hk}(\check{X}_C)\otimes^{\LL_{\Box}}_{{C}^{\nr}}\B_{\rm FF}]^{N=0}).
$$
By taking $\phi$-invariants, this yields a natural  $\sg_K$-equivariant quasi-isomorphism in $\sd(\B_{e,\Box})$
$$
\rg(\check{G},[\rg_{\hk}(X_C)\otimes^{\LL_{\Box}}_{{C}^{\nr}}\B_{\rm FF}]^{N=0,\phi=1})\simeq \rg(G,[\rg_{\hk}(\check{X}_C)\otimes^{\LL_{\Box}}_{{C}^{\nr}}\B_{\rm FF}]^{N=0,\phi=1}).
$$
By tensoring it  with $\B_{\rm pFF}$ over $\B_e$, we get a natural  $\sg_K$-equivariant quasi-isomorphism in $\sd^{\prime}_{\phi,N}(\B_{{\rm pFF},\Box})$ 
\begin{equation}\label{rain21}
\rg(\check{G},[\rg_{\hk}(X_C)\otimes^{\LL_{\Box}}_{{C}^{\nr}}\B_{\rm FF}]^{N=0,\phi=1}\otimes^{\LL_{\Box}}_{\B_e}\B_{\rm pFF})\simeq \rg(G,[\rg_{\hk}(\check{X}_C)\otimes^{\LL_{\Box}}_{{C}^{\nr}}\B_{\rm FF}]^{N=0,\phi=1}\otimes^{\LL_{\Box}}_{\B_e}\B_{\rm pFF}),
\end{equation}
We used here the following computation: 
\begin{lemma} We have a canonical quasi-isomorphism in $\sd^{\prime}_{\phi,N}(\B_{\mathrm{pFF}, \Box})$ 
$$\rg(\check{G},[\rg_{\hk}(X_C)\otimes^{\LL_{\Box}}_{{C}^{\nr}}\B_{\rm FF}]^{N=0,\phi=1})\otimes^{\LL_{\Box}}_{\B_e}\B_{\rm pFF}\simeq \rg(\check{G},[\rg_{\hk}(X_C)\otimes^{\LL_{\Box}}_{{C}^{\nr}}\B_{\rm FF}]^{N=0,\phi=1}\otimes^{\LL_{\Box}}_{\B_e}\B_{\rm pFF}).$$
Similarly for ${G}$ and $\check{X}$. 
\end{lemma}
\begin{proof} Since we have assumed the Hyodo-Kato cohomology has finite rank over ${C}^{\nr}$, it suffices to show that,  for a finite rank  smooth $\check{G}$ representation $V$ over ${C}^{\nr}$, which is a $(\phi,N)$-module, we have in $\sd^{\prime}_{\phi,N}(\B_{\mathrm{pFF}, \Box})$ 
$$\rg(\check{G},(V\otimes^{{\Box}}_{{C}^{\nr}}\B_{\rm FF})^{N=0,\phi=1})\otimes^{\LL_{\Box}}_{\B_e}\B_{\rm pFF}\simeq \rg(\check{G},(V\otimes^{{\Box}}_{{C}^{\nr}}\B_{\rm FF})^{N=0,\phi=1}\otimes^{\LL_{\Box}}_{\B_e}\B_{\rm pFF}).$$
For that, take an  open normal subgroup \(H\subset\check G\) that fixes a basis and acts trivially on \(V\). Since $H$ acts trivially on $\wt{V}:=(V\otimes^{{\Box}}_{{C}^{\nr}}\B_{\rm FF})^{N=0,\phi=1}$, Proposition \ref{magic} yields 
\begin{align*}
\rg(H,\wt{V})\otimes^{\LL_{\Box}}_{\B_e}\B_{\rm pFF}& \simeq \rg(H,\Q_p)\otimes^{\LL_{\Box}}_{\Q_p}(\wt{V}\otimes^{\LL_{\Box}}_{\B_e}\B_{\rm pFF}),\\
\rg(H,(V\otimes^{{\Box}}_{{C}^{\nr}}\B_{\rm FF})^{N=0,\phi=1}\otimes^{\LL_{\Box}}_{\B_e}\B_{\rm pFF}) & \simeq \rg(H,\Q_p)\otimes^{\LL_\Box}_{\Q_p}(\wt{V}\otimes^{\LL_{\Box}}_{\B_e}\B_{\rm pFF})
\end{align*}
and, hence, our claim with $H$ in place of $\check{G}$. We get the claim for $\check{G}$ by taking invariants under $\check{G}/H$. 
\end{proof}

   The quasi-isomorphism \eqref{rain21} yields 
a natural  $\sg_K$-equivariant quasi-isomorphism in $\sd^{\prime}_{\phi,N}({C}^{\nr}_{\Box})$
$$
\rg(\check{G},([\rg_{\hk}(X_C)\otimes^{\LL_{\Box}}_{{C}^{\nr}}\B_{\rm FF}]^{N=0,\phi=1}\otimes^{\LL_{\Box}}_{\B_e}\B_{\rm pFF})^{\R\sg_K-{\rm sm}})
\simeq \rg(G,([\rg_{\hk}(\check{X}_C)\otimes^{\LL_{\Box}}_{{C}^{\nr}}\B_{\rm FF}]^{N=0,\phi=1}\otimes^{\LL_{\Box}}_{\B_e}\B_{\rm pFF})^{\R\sg_K-{\rm sm}}).
$$
Now, we can combine it with the $\sg_K$-equivariant quasi-isomorphism in $\sd^{\prime}_{\phi,N}({C}^{\nr}_{\Box})$ (see Corollary \ref{killing1HK})
$$
\rg_{\hk}(Z_C)\stackrel{\sim}{\to} ([\rg_{\hk}(Z_C)\otimes^{\LL_{\Box}}_{{C}^{\nr}}\B_{\rm FF}]^{N=0,\phi=1}\otimes^{\LL_{\Box}}_{\B_e}\B_{\rm pFF})^{\R\sg_K-{\rm sm}}, \quad Z=X, \check{X},
$$
to get  \eqref{berkeley1HK} in our theorem in the case of finite rank de Rham cohomology. 

         To treat the general case, let $\{X_n\}$, $n\in\N$, be a strictly increasing covering of  $X$ by $\check{G}$-stable quasi-compact opens.  We claim that we can find a covering $\{X^0_n\}$ of $X$ such that:  $X_{n-1}\subset X^0_n\subset X_n$, $X^0_n$ is partially proper,  $H^i_{\dr,c}(X^0_n)$ is of finite rank over $K$, and  the covering $\{X^0_n\}$ is  $\check{G}$-equivariant. Indeed, it suffices to find such $X^0_n$ for a pair  of quasi-compact opens $X_{n-1}\Subset X_n$. 
         Arguing as in the proof of \cite[Lemma~2.28]{AGN}, we may
choose a partially proper open $U$ and a quasi-compact
open $W\subset X_n$ such that
\[
  X_{n-1}\subset U\subset X_n,
  \qquad
  W=(X_n\setminus U)^\circ,
  \qquad
  U=X_n\setminus\overline{W}^{\,X_n}.
\]
Since $W$ is a quasi-compact open, it is stabilized by
an open subgroup $H\subset\check G$. This subgroup also
stabilizes $U$, since $\check G$ stabilizes $X_n$.
Set
\[
  X_n^0:=\bigcup_{g\in\check G/H}gU,
  \qquad
  W_n:=\bigcap_{g\in\check G/H}gW.
\]
These are finite unions and intersections. Thus $X_n^0$
is a $\check G$-stable partially proper open satisfying
$X_{n-1}\subset X_n^0\subset X_n$, and $W_n$ is a
$\check G$-stable quasi-compact open. Moreover, we have 
\begin{equation}\label{chatus11}
X_n^0=X_n\setminus\overline{W_n}^{\,X_n}.
\end{equation}

     Let $\{\check X_n^0\}_{n\in\mathbf N}$ be the corresponding
covering of $\check X$, defined by $
  \check X_n^0:=\check\pi\bigl(\pi^{-1}(X_n^0)\bigr).$
We claim that  these are $G$-stable partially proper opens.
 Indeed, the projections $\pi$ and $\check\pi$ are  quasi-compact and
universally open: this follows from the description of a
profinite torsor as an inverse limit of finite \'etale torsors
(see \cite[Lemma~10.13]{Sz-diamonds}).
They are also proper:  they
are separated and quasi-compact, and the valuative criterion
can be checked at each finite \'etale level; uniqueness of
the extensions makes them compatible with the transition maps
(see \cite[Prop.~18.3]{Sz-diamonds}).

 To prove partial properness, recall that an open immersion
of rigid analytic spaces is partially proper if and only if
its image is stable under specialization
(see \cite[Cor.~1.10.21]{Hub}).
The immersion $X_n^0\hookrightarrow X$ is partially proper,
since $X_n^0$ is partially proper over $K$ and $X$ is separated
over $K$. Thus $|X_n^0|$ is stable under specialization in $|X|$.

Let $\check x\in|\check X_n^0|$, and let
$\check y\in|\check X|$ be a specialization of $\check x$. Take the open subdiamond
$T_n^0:=\pi^{-1}(X_n^0)$ and 
choose $t\in|T_n^0|$ mapping to $\check x$.
As $|\check\pi|$ is a closed map, the closed subset $ |\check\pi|\bigl(\overline{\{t\}}\bigr)
  \subset|\check X| $
contains $\check x$, hence also $\check y$.
We may therefore choose a specialization $t'$ of $t$ with
$\check\pi(t')=\check y$.
Now $\pi(t')$ is a specialization of $\pi(t)\in|X_n^0|$,
so $\pi(t')\in|X_n^0|$. Hence $t'\in|T_n^0|$, and
$\check y\in|\check X_n^0|$.
This proves that $\check X_n^0\hookrightarrow\check X$ is
partially proper. Since $\check X$ is partially proper over
$K$, so is $\check X_n^0$.

   We now check  the finite-dimensionality of the cohomology
on both sides. Set 
$\check X_n:=\check\pi\bigl(\pi^{-1}(X_n)\bigr),
  \check W_n:=\check\pi\bigl(\pi^{-1}(W_n)\bigr). $
Since $X_n$ and $W_n$ are $\check G$-stable, we have  $
  \pi^{-1}(X_n)=\check\pi^{-1}(\check X_n),
  \pi^{-1}(W_n)=\check\pi^{-1}(\check W_n). $
The spaces $\check X_n$ and $\check W_n$ are $G$-stable
quasi-compact opens: quasi-compactness follows from that
of $\pi$, and openness from that of $\check\pi$.

Inverse images under open continuous maps commute with
taking closures. Applying this to the two torsor
projections on the common open
$\pi^{-1}(X_n)=\check\pi^{-1}(\check X_n)$ gives
\[
  \pi^{-1}\bigl(\overline{W_n}^{\,X_n}\bigr)
  =
  \overline{\pi^{-1}(W_n)}^{\,\pi^{-1}(X_n)}
  =
  \check\pi^{-1}
     \bigl(\overline{\check W_n}^{\,\check X_n}\bigr).
\]
Together with \eqref{chatus11} and
the equality
$\pi^{-1}(X_n^0)=\check\pi^{-1}(\check X_n^0)$, this yields
\begin{equation}\label{dual-exhaustion-complement}
  \check X_n^0
  =
  \check X_n\setminus
        \overline{\check W_n}^{\,\check X_n}.
\end{equation}
Apply now   \cite[Th.~3.6]{GK} to the
two pairs of dagger spaces
$(X_n^\dagger,W_n^\dagger),
  (\check X_n^\dagger,\check W_n^\dagger):$
The ambient dagger spaces are smooth and quasi-compact,
and the indicated admissible opens are quasi-compact.
Their admissible complements consequently have
finite-dimensional dagger de Rham cohomology.
Hence the cohomology groups $
 H^q_{\mathrm{dR}}(X_n^0)$ and 
$H^q_{\mathrm{dR}}(\check X_n^0)$ have finite rank. 
By Poincar\'e duality
\cite[Th.~5.23]{AGN},  the same holds for compactly
supported cohomology.

    By construction,  the coverings $\{X^0_n\}$ and 
     $\{\check{X}^0_n\}$ are both  $G$ and $\check{G}$ equivariant. 
\begin{remark}
Let $f:Y\to X$ be a finite \'etale map from the tower.
For each $n$, set
\[
  Y_n:=f^{-1}(X_n),\qquad
  W_{n,Y}:=f^{-1}(W_n),\qquad
  Y_n^0:=f^{-1}(X_n^0).
\]
The spaces $Y_n$ and $W_{n,Y}$ are quasi-compact opens,
and $Y_n^0$ is partially proper.
Since $f$ is open, we have 
$ Y_n^0
  =
  Y_n\setminus\overline{W_{n,Y}}^{\,Y_n}.$
 Applying the same finiteness argument as above to
$(Y_n^\dagger,W_{n,Y}^\dagger)$ shows that the cohomology groups 
$ H^q_{\mathrm{dR}}(Y_n^0),  H^q_{\mathrm{dR},c}(Y_n^0)$ have finite rank. 
The identical argument applies to finite \'etale maps
from the other tower. Thus the common exhaustion has
the required finite-dimensionality property at every
finite level.
\end{remark}
 
    We compute in $\sd'_{\phi,N,\sg_K}({C}^{\nr}_{\Box})$:  
  \begin{align*}
   \rg(\check{G},\rg_{\hk}(X_C))&  \stackrel{\sim}{\to}\rg(\check{G},\R\lim_n\rg_{\hk}(X^0_{n,C})) \stackrel{\sim}{\to}
   \R\lim_n\rg(\check{G},\rg_{\hk}(X^0_{n,C}))\\
&  \simeq   \R\lim_n\rg({G},\rg_{\hk}(\check{X}^0_{n,C}))\stackrel{\sim}{\leftarrow} \rg({G},\R\lim_n\rg_{\hk}(\check{X}^0_{n,C})) \stackrel{\sim}{\leftarrow} \rg({G},\rg_{\hk}(\check{X}_C)).
  \end{align*}
This finishes the proof of \eqref{berkeley1HK} in our theorem.

  For compactly supported cohomology we argue as in the proof of Corollary \ref{bonn15} using \eqref{berkeley1HK} that we have already shown.
   \end{proof}
 \subsubsection{Lie algebra flip-flopping} 
  We set
$$
\R\Gamma_{\hk,c}(X_{\infty,C})  \simeq\colim_{H}\R\Gamma_{\hk,c}(X_{H,C}),\quad 
\R\Gamma_{\hk,c}(\check{X}_{\infty,C})  \simeq \colim_{\check{H}}\R\Gamma_{\hk,c}(\check{X}_{\check{H},C}),
$$
where $H, \check{H}$ are compact open subgroups of $G, \check{G}$, respectively, and $X_H, \check{X}_{\check{H}}$ are the respective quotients of $T$.  
 By Proposition \ref{derham-prop1}, $\R\Gamma_{\hk,c}(X_{\infty,C})\in {\rm Rep}^{\rm sm}_{{C}^{\nr}_{\Box}}(G\times\check{G})$. Similarly for $\R\Gamma_{\hk,c}(\check{X}_{\infty,C})$. 

 \begin{corollary}{\rm(Lie algebra  flip-flopping)}  \label{paris30} Assume that $X$, $\check{X}$ are partially proper and that  $G$, $\check{G}$ are compact $p$-adic Lie groups over $\Q_p$, with respective Lie algebras 
 $\mathfrak{g}, \check{\mathfrak{g}}$.  
There are $G\times\check{G}$-equivariant natural quasi-isomorphisms in $\sd^{\prime}_{\phi,N,\sg_K}({C}^{\nr}_{\Box})$
\begin{align}\label{paris12HK}
\R\Gamma(\check{\mathfrak{g}},\R\Gamma_{\hk,c}(X_{\infty,C}))& \simeq \R\Gamma(\mathfrak{g},\R\Gamma_{\hk,c}(\check{X}_{\infty,C})),\\
\R\Gamma(\check{\mathfrak{g}}, {C}^{\nr})\otimes_{{C}^{\nr}}^{\LL_{\Box}}\R\Gamma_{\hk,c}(X_{\infty,C}) & \simeq 
\R\Gamma(\mathfrak{g},{C}^{\nr})\otimes^{\LL_{\Box}}_{{C}^{\nr}}\R\Gamma_{\hk,c}(\check{X}_{\infty,C}).\notag
\end{align}
\end{corollary}
\begin{proof}  Having the quasi-isomorphism \eqref{berkeley4HK},  we argue exactly as in the proof of Corollary \ref{dom1} using Remark \ref{res1}  for \(C^{\mathrm{nr}}\)-coefficients.
\end{proof}
\subsection{Flip-flopping at infinity}
  Below, a $(\phi,N,\sg_K)$-module over ${C}^{\nr}$ will always refer to a solid module, where $\phi$ is an isomorphism and the action of  $\sg_K$ is smooth. In the case of compactly supported cohomology we have the following underived version of flip-flopping if we are willing to pass to infinity. 
\begin{theorem}\label{dept1} Assume that $X, \check{X}$ are partially proper and 
 that $G, \check{G}$ are $p$-adic Lie groups with
$$
{\dim}_{K}H^i(\mathfrak{g},K)={\dim}_{K}H^i(\check{\mathfrak{g}},K),\quad i\geq 0.
$$
Then there exist compact open subgroups $H\subset G$, $\check{H}\subset \check{G}$ and a $H\times\check{H}$-equivariant isomorphism in $K_{\Box}$ and  $(\phi,N,\sg_K)$ modules over ${C}^{\nr}$, respectively:
\begin{align}\label{step1}
H^i_{\dr,c}(X_{\infty}) & \simeq H^i_{\dr,c}(\check{X}_{\infty}),\quad i\geq 0,\\
H^i_{\hk,c}(X_{\infty,C}) & \simeq H^i_{\hk,c}(\check{X}_{\infty,C}),\quad i\geq 0.\notag
\end{align}
\end{theorem}
\begin{remark}  The proof of the theorem will show that we can choose $H,\check{H}$ to be any subgroups acting trivially on all Lie algebra cohomology groups. 
\end{remark}
\begin{proof} We note that  the groups $H^i(\check{\mathfrak{g}},{K})$, $H^i({\mathfrak{g}},{K})$, for $i\geq 0$,   are of finite rank over ${K}$. We  choose $H\subset G, \check{H}\subset \check{G}$ to be any compact open subgroups acting trivially on all  Lie algebra cohomology groups. \\

 (1)  {\em De Rham cohomology.} Let us  start with  de Rham cohomology.  We will first assume that the $G\times\check{G}$-representations
$H^i_{\dr,c}(X_{\infty}),  H^i_{\dr,c}(\check{X}_{\infty})$ are admissible. 
We will argue  by induction on $i$. For $i=0$, the required isomorphism follows by taking
degree-zero cohomology in \eqref{paris12}, since the
complexes are connective and $
  H^0(\mathfrak g,K)=H^0(\check{\mathfrak g},K)=K. $ For the inductive step, we assume the isomorphism \eqref{step1}   for  all degrees $\leq  i$ and we want to prove it for $i+1$. 

  From \eqref{paris12} we get a $G\times\check{G}$-equivariant  isomorphism in ${K}_{\Box}$
 \begin{equation}\label{ias2s}
H^{i+1}(\R\Gamma(\check{\mathfrak{g}}, {K})\otimes_{{K}}^{\LL_{\Box}}\R\Gamma_{\dr,c}(X_{\infty}))  \simeq H^{i+1}(\R\Gamma(\mathfrak{g},{K})\otimes^{\LL_{\Box}}_{{K}}\R\Gamma_{\dr,c}(\check{X}_{\infty})).
 \end{equation}
By K\"unneth formula, it  yields a $G\times \check{G}$-equivariant
   isomorphism in ${K}_{\Box}$
   \begin{align}\label{kun10}
     & H^0(\check{\mathfrak{g}},{K})\otimes^{\Box}_{{K}}H^{i+1}_{\dr,c}(X_{\infty})  \oplus  H^1(\check{\mathfrak{g}},{K})\otimes^{\Box}_{{K}}H^{i}_{\dr,c}(X_{\infty})\oplus\cdots\oplus   H^{i+1}(\check{\mathfrak{g}},{K})\otimes^{\Box}_{{K}}H^{0}_{\dr,c}(X_{\infty})\\
      & \simeq  H^0(\mathfrak{g},{K})\otimes^{\Box}_{{K}}H^{i+1}_{\dr,c}(\check{X}_{\infty})\oplus  H^1(\mathfrak{g},{K})\otimes^{\Box}_{{K}}H^{i}_{\dr,c}(\check{X}_{\infty})\oplus\cdots\oplus   H^{i+1}(\mathfrak{g},{K})\otimes^{\Box}_{{K}}H^{0}_{\dr,c}(\check{X}_{\infty})\notag
   \end{align}
  Since  the groups $H^j(\check{\mathfrak{g}},{K})$, $H^j({\mathfrak{g}},{K})$, for $j\geq 0$,   are of finite rank over ${K}$ and   $H^0(\check{\mathfrak{g}},{K})=H^0(\mathfrak{g},{K})={K}$, if we want to prove the  claim of the theorem for $i+1$, by the inductive assumption, we find ourselves in the situation where  we have an isomorphism of
   admissible representations of $H\times\check{H}$ over $K$: 
$$
  \pi_1\oplus \pi \simeq \pi_2\oplus \pi
$$
  and we want to conclude that $\pi_1\simeq \pi_2$, as $H\times\check{H}$-representations.  This holds by Proposition \ref{cancel1}.

        To treat the general case, choose the coverings  $\{X^0_{n,\infty}, \check{X}^0_{n,\infty}\}$, $n\in\N$, as in the proof of Theorem \ref{HKFF}.   
Set  $ T_n:=\pi^{-1}(X_n^0)
      =\check\pi^{-1}(\check X_n^0)\subset T $ and let   $P=H\times\check H$.
Set
\[
  V_n^i:=H^i_{\mathrm{dR},c}(X_{n,\infty}^0),
  \qquad
  \check V_n^i:=H^i_{\mathrm{dR},c}
                         (\check X_{n,\infty}^0).
\]
Each $V_n^i$ and $\check V_n^i$ is a smooth admissible
$P$-representation. Indeed, for sufficiently small open
normal subgroups $U\subset H$ and $\check U\subset\check H$,
finite \'etale descent gives
\[
  (V_n^i)^{U\times\check U}
  \simeq H^i_{\mathrm{dR},c}(T_n/U)^{\check U},
  \qquad
  (\check V_n^i)^{U\times\check U}
  \simeq H^i_{\mathrm{dR},c}(T_n/\check U)^U.
\]
These spaces are finite-dimensional. 
 We have
\begin{equation}\label{sobota1}
  H^i_{\mathrm{dR},c}(X_\infty)
  \simeq\operatorname*{colim}_n V_n^i,
  \qquad
  H^i_{\mathrm{dR},c}(\check X_\infty)
  \simeq\operatorname*{colim}_n\check V_n^i.
\end{equation}

 Let $\mathcal O_P$ be the category whose objects are
$P/J$, for open normal subgroups $J\subset P$, and whose
morphisms are $P$-equivariant maps. Set $ \mathcal I:=\N\times\mathcal O_P^{\mathrm{op}}, $
where $\N$ is ordered by its usual order.
For every $i\geq0$, define finite-dimensional diagrams
\[
  \mathcal V^i(n,P/J):=(V_n^i)^J,
  \qquad
  \check{\mathcal V}^i(n,P/J):=(\check V_n^i)^J.
\]
The maps in $n$ come from extension by zero. For 
the orbit-category variable:  a $P$-equivariant map
$P/J'\to P/J$ sending $1J'$ to $gJ$ induces
$v\mapsto gv$ from $(V_n^i)^J$ to $(V_n^i)^{J'}$,
and similarly on the other side.

Set 
$b_a:=\dim_{\Q_p}H^a(\mathfrak g,\Q_p)
      =\dim_{\Q_p}
                   H^a(\check{\mathfrak g},\Q_p). $
We have $b_0=1$.
Fix bases over $\Q_p$ of these Lie algebra
cohomology groups.
Applying Corollary~\ref{dom1} to the dual towers defined
by $T_n$ and using the K\"unneth isomorphism yields $P$-equivariant isomorphisms
\[
  \bigoplus_{a=0}^i(V_n^{i-a})^{\oplus b_a}
  \simeq
  \bigoplus_{a=0}^i(\check V_n^{i-a})^{\oplus b_a}, 
\]
 which are compatible  with the transition maps in $n$ (by the naturality of Corollary~\ref{dom1} with respect
to the transition maps).  This compatibility is preserved in
passing to the level colimits.
Taking  invariants of these cohomology groups  gives (as diagrams on $\mathcal I$)
\begin{equation}\label{diagram-stabilized-dR}
  \bigoplus_{a=0}^i
           (\mathcal V^{i-a})^{\oplus b_a}
  \simeq
  \bigoplus_{a=0}^i
           (\check{\mathcal V}^{i-a})^{\oplus b_a}.
\end{equation}

We now argue by induction on $i$ in the diagram category.
For $i=0$, \eqref{diagram-stabilized-dR} gives
$\mathcal V^0\simeq\check{\mathcal V}^0$.
Suppose that diagram isomorphisms have been constructed
in all degrees less than $i$. Using them in
\eqref{diagram-stabilized-dR}, we obtain
\[
  \mathcal V^i\oplus\mathcal Q_i
  \simeq
  \check{\mathcal V}^i\oplus\mathcal Q_i,
  \qquad
  \mathcal Q_i:=
  \bigoplus_{a=1}^i(\mathcal V^{i-a})^{\oplus b_a}.
\]
Lemma~\ref{berkeley1W}, with
$\mathcal A=\operatorname{Vect}^{\mathrm{fd}}_K$,
yields $\mathcal V^i\simeq\check{\mathcal V}^i$.

By smoothness, we recover $V_n^i$ and $\check V_n^i$
by taking the colimit over the inclusions of invariant
spaces as $J$ shrinks. The resulting isomorphisms
$V_n^i\simeq\check V_n^i$ are $P$-equivariant  and they are compatible with the transition maps
in $n$. Taking the colimit in \eqref{sobota1} gives
\[
  H^i_{\mathrm{dR},c}(X_\infty)
  \simeq H^i_{\mathrm{dR},c}(\check X_\infty)
\]
as solid $P$-representations.

    (2)  {\em Hyodo-Kato cohomology.} We move now to Hyodo-Kato  cohomology.  We will first assume that the $G\times\check{G}$-representations
$H^i_{\hk,c}(X_{\infty,C}),  H^i_{\hk,c}(\check{X}_{\infty,C})$ are admissible. 
We will argue  by induction on $i$ as in the de Rham case. For $i=0$, the required isomorphism follows by taking
degree-zero cohomology in \eqref{paris12HK}, since the
complexes are connective and $ H^0(\mathfrak g,{C}^{\nr})=H^0(\check{\mathfrak g},{C}^{\nr})={C}^{\nr}.$  For the inductive step, we assume the isomorphism \eqref{step1}   for  all degrees $\leq  i$ and we want to prove it for $i+1$. 
 
    From Corollary \ref{paris30}  we get a quasi-isomorphism in $\sd_{\phi,N,\sg_K}({C}^{\nr}_{\Box})$
$$
H^{i+1}(\R\Gamma(\check{\mathfrak{g}}, {C}^{\nr})\otimes_{{C}^{\nr}}^{\LL_{\Box}}\R\Gamma_{\hk,c}(X_{\infty,C})) 
 \simeq H^{i+1}(\R\Gamma(\mathfrak{g},{C}^{\nr})\otimes^{\LL_{\Box}}_{{C}^{\nr}}\R\Gamma_{\hk,c}(\check{X}_{\infty,C})).
$$
This isomorphism is ${C}^{\nr}[G\times \check{G}]$-equivariant.
Using the K\"unneth formula, we get  a ${C}^{\nr}[G\times \check{G}]$-equivariant
   quasi-isomorphism in $\sd_{\phi,N,\sg_K}({C}^{\nr}_{\Box})$
   \begin{align}\label{kunnethHK1}
     & H^0(\check{\mathfrak{g}},C^{\nr})\otimes^{\Box}_{C^{\nr}}H^{i+1}_{\hk,c}(X_{\infty,C}) \oplus  H^1(\check{\mathfrak{g}},C^{\nr})\otimes^{\Box}_{C^{\nr}}H^{i}_{\hk,c}(X_{\infty,C})\oplus\cdots\oplus   H^{i+1}(\check{\mathfrak{g}},C^{\nr})\otimes^{\Box}_{C^{\nr}}H^{0}_{\hk,c}(X_{\infty,C})\\
      & \simeq  H^0(\mathfrak{g},C^{\nr})\otimes^{\Box}_{C^{\nr}}H^{i+1}_{\hk,c}(\check{X}_{\infty,C})\oplus  H^1(\mathfrak{g},C^{\nr})\otimes^{\Box}_{C^{\nr}}H^{i}_{\hk,c}(\check{X}_{\infty,C})\oplus\cdots\oplus   H^{i+1}(\mathfrak{g},C^{\nr})\otimes^{\Box}_{C^{\nr}}H^{0}_{\hk,c}(\check{X}_{\infty,C}).\notag
   \end{align}
   
 Since $H^a(\mathfrak g,C^{\mathrm{nr}})=H^a(\mathfrak g,\mathbf Q_p)\otimes_{\mathbf Q_p}C^{\mathrm{nr}}$
and similarly for \(\check{\mathfrak g}\), the rational bases fixed above identify these factors with finite direct sums of the unit \((\varphi,N,\sg_K)\)-module. Hence  if we want to prove the  claim of the theorem for $i+1$, by the inductive assumption, we find ourselves in the situation where  we have an isomorphism of  admissible representations of $H\times\check{H}$ on solid $(\phi,N,\sg_K)$-modules over $C^{\nr}$: 
  \begin{equation}\label{paris101}
  \pi_1\oplus \pi \simeq \pi_2\oplus \pi
  \end{equation}
  and we want to conclude that $\pi_1\simeq \pi_2$, as $H\times\check{H}$-representations on solid $(\phi,N,\sg_K)$-modules over ${C}^{\nr}$. For this, apply Lemma~\ref{berkeley1W} to the three diagrams
\[
  P/J\longmapsto\pi^J,\qquad
  P/J\longmapsto\pi_1^J,\qquad
  P/J\longmapsto\pi_2^J
\]
on $\mathcal O_P^{\mathrm{op}}$, with values in the
Hom-finite category $\mathcal A_{\mathrm{HK}}$
described below. Their values have finite rank by
admissibility. The resulting diagram isomorphism
$\pi_1^J\simeq\pi_2^J$ gives
$\pi_1\simeq\pi_2$ by smoothness.

    To treat the general case, we 
use the same exhaustion and the same category
$\mathcal I$.
Set
\begin{align*}
  &  W_n^i:=H^i_{\mathrm{HK},c}(X_{n,\infty,C}^0),
  \qquad
  \check W_n^i:=H^i_{\mathrm{HK},c}
                         (\check X_{n,\infty,C}^0),\\
 &   \mathcal W^i(n,P/J):=(W_n^i)^J,
  \qquad
  \check{\mathcal W}^i(n,P/J):=(\check W_n^i)^J.
\end{align*}
The latter two  spaces have finite
rank over $C^{\nr}$.
They retain their Frobenius, monodromy, and smooth
$\sg_K$-action.

Let $\mathcal A_{\mathrm{HK}}$ be the category of finite-rank
$(\varphi,N,\sg_K)$-modules over $C^{\nr}$, with invertible
semilinear Frobenius and smooth Galois action.
This is an essentially small additive $\Q_p$-linear
category with finite-dimensional Hom spaces: 
we have 
\[
  \operatorname{Hom}_{\mathcal A_{\mathrm{HK}}}(D_1,D_2)
  \subset
  \operatorname{Hom}_{\varphi}(D_1,D_2),
  \qquad
  \dim_{\Q_p}\operatorname{Hom}_{\varphi}(D_1,D_2)
  \leq
  (\dim_{C^{\nr}}D_1)(\dim_{C^{\nr}}D_2).
\]
Thus Lemma~\ref{berkeley1W} applies to
$\mathcal A_{\mathrm{HK}}$-valued diagrams.

Corollary~\ref{paris30}, its compatibility with extension
by zero, and the K\"unneth isomorphism give
\begin{equation}\label{diagram-stabilized-HK}
  \bigoplus_{a=0}^i
           (\mathcal W^{i-a})^{\oplus b_a}
  \simeq
  \bigoplus_{a=0}^i
           (\check{\mathcal W}^{i-a})^{\oplus b_a}
\end{equation}
in $\operatorname{Fun}(\mathcal I,\mathcal A_{\mathrm{HK}})$,
where we used the bases over $\Q_p$ fixed above.
The same induction as in the de Rham case, now using
Lemma~\ref{berkeley1W} for $\mathcal A_{\mathrm{HK}}$,
produces diagram isomorphisms $
  \mathcal W^i\simeq\check{\mathcal W}^i $, 
for all $i\geq0$.
Taking the colimits over the inclusions of invariants
and then over $n$ yields $P$-equivariant isomorphisms
\[
  H^i_{\mathrm{HK},c}(X_{\infty,C})
  \simeq
  H^i_{\mathrm{HK},c}(\check X_{\infty,C})
\]
compatible with $\varphi$, $N$, and $\sg_K$.
     This finishes the proof of the theorem.
 \end{proof}

\section{Flip-flopping for dual  towers of local Shimura varieties} \label{shimura1} We will now show that the compactly supported de Rham and Hyodo-Kato cohomologies of basic dual towers of local Shimura varieties flip-flop together with the actions of the associated $p$-adic Lie groups. 

\subsection{Flip-flopping for Drinfeld and Lubin-Tate towers}  We start with the key motivating example of the Drinfeld and Lubin-Tate towers.

   Let $K$ be a finite extension of $\Q_p$, let  $\varpi$ be a uniformizer of $K$, and let $d\geq 1$. Let $G=\mathbb{GL}_{d+1}(K)$ and let $\check{G}=D^{\times}$  be the invertible elements in the central
division algebra $D$ of dimension $(d+1)^2$ and  invariant $1/(d+1)$ over $K$. Let $\O_D$ be the maximal order in  $D$ and  let  $\varpi_D$ be a uniformizer of  $\O_D$.

\subsubsection{Drinfeld and Lubin-Tate towers}\label{recall1} We start with a quick review of the Drinfeld and  Lubin-Tate towers.  

 ($\bullet$) {\em Drinfeld tower.} Recall that 
the Drinfeld  space 
${\mathbb P}^d_K\setminus \cup_{H\in\sh}H$, where $\sh$ is the set of $K$-rational hyperplanes,  admits a natural rigid analytic structure 
 ${\mathbb H}^d_K$ over  $K$ and an action of $G$ by  homographies,
which respects that structure.
Drinfeld defined a tower of coverings 
 $\breve{\cal M}_n$, for $n\in\N$, 
of ${\mathbb H}^d_{\breve{K}}$ with the following properties:
\begin{enumerate}[leftmargin=*]
\item  $\breve{\cal M}_n$ is defined over  $\breve{K}$
and is equipped with an action of the Weil group ${\rm W}_K$ compatible with its natural action on $\breve{K}$.

\item  $\breve{\cal M}_n$ is equipped with commuting  actions of  $G$ and  $\check G$
which also commute with the action of  ${\rm W}_K$. The transition maps 
 $\breve{\cal M}_{n+1}\to \breve{\cal M}_n\to
{\mathbb H}^d_{\breve{K}}$ are  ${\rm W}_K$, $\check G$ and  $G$-equivariant (the action of 
$\check G$ on  ${\mathbb H}^d_{\breve{K}}$ being trivial).

\item  $\breve{\cal M}_0=\Z\times{\mathbb H}^d_{\breve{K}}$ and,
for  $n\geq 1$,  $\breve{\cal M}_n$ is a Galois covering 
of  $\breve{\cal M}_0$, with  Galois group $\O_D^\times/(1+\varpi_D^n\O_D)$.
\end{enumerate}

   We set ${\cal M}_n:= \breve{\cal M}_n\times_{\breve{K}}C.$
Let  ${\cal M}_\infty$ be the  (non completed) projective system 
of  ${\cal M}_n$. If  $H=G,\check G,{\rm W}_K$, we have a natural morphism
$\nu_H:H{\hskip.5mm\to\hskip.5mm}K^\times,$
where  $\nu_G=\det$, $\nu_{\check G}$ is the reduced norm,
and $\nu_{{\rm W}_K}$ is the composition of  ${\rm W}_K\to {\rm W}_K^{\rm ab}$ and 
the isomorphism  ${\rm W}_K^{\rm ab}\simeq  K^\times$ of local class field theory.
The set  $\pi_0({\cal M}_\infty)$ of the connected components of 
${\cal M}_\infty$ is a principal homogeneous space under the action of 
 $K^\times$ and  $H=G,\check G,{\rm W}_K$ acts on  $\pi_0({\cal M}_\infty)$
via  $\nu_H:H\to K^\times$.  In particular, $\check G$ acts on 
$\pi_0({\cal M}_\infty)$  via the reduced norm.

  Let  $\widehat{\sm}_\infty$ be the completion of $\sm_{\infty}$. We denote by $\widehat{\sm}_{\infty,\breve{K}}, \sm_{\infty,\breve{K}}$ the analogous  spaces over $\breve{K}$.

  \vskip2mm ($\bullet$) {\em Lubin-Tate tower.} Lubin-Tate theory yields a  formal scheme
    ${\rm Spf}(A_0)$
    over $\mathcal{O}_{\breve{K}}$, classifying deformations by quasi-isogenies of the unique 
  $\mathcal{O}_K$-formal module of dimension $1$  and height  $d+1$ (relative to  $K$) over ~$\overline{\mathbf{F}}_p$. 
  Adding level structures to these deformations, Drinfeld  constructed a tower of formal schemes
        $({\rm Spf}(A_n))_{n\geq 1}$ over 
         $\mathcal{O}_{\breve{K}}$. The rigid analytic  generic fiber $\breve{{\rm LT}}_n$ of ~${\rm Spf}(A_n)$ is a finite \'etale Galois covering of $\breve{{\rm LT}}_0$ with Galois group
       $\mathbb{GL}_{d+1}(\mathcal{O}_K/\varpi^n)$. $\breve{{\rm LT}}_0$ is 
 a countable disjoint union of open polydiscs of dimension $d$ over $\breve{K}$.

     One can define an action of   $G\times \check{G}$ on the tower $({\rm LT}_n)_{n\geq 0}$ (unlike in the case of  Drinfeld tower, we do not have an action of $G\times \check{G}$ on every level of the tower). Set  ${\rm LT}_n:=\breve{{\rm LT}}_n\times_{\breve{K}} C$.
Let  $\widehat{\rm LT}_\infty$ be the completion of the limit ${\rm LT}_\infty$ of the tower of 
$({\rm LT}_n)_{n\geq 0}$. We denote by $\widehat{\rm LT}_{\infty,\breve{K}}, {\rm LT}_{\infty,\breve{K}}$ the analogous  spaces over $\breve{K}$.
 \subsubsection{Duality of the towers} 
The completions  $\widehat{\rm LT}_\infty$ and  $\widehat{\cal M}_{\infty}$ are 
equipped with an action of  $G\times \check{G}$
and isomorphic as  $(G\times\check G)$-spaces~\cite[Th. 7.2.3]{SW}. Same holds for their $\breve{K}$-analogous  $\widehat{\rm LT}_{\infty,\breve{K}}$ and 
 $\widehat{\cal M}_{\infty,\breve{K}}$. Moreover, the completions  $\widehat{\rm LT}_\infty$ and  $\widehat{\cal M}_{\infty}$ are 
 perfectoid spaces~\cite[Th. 6.5.4]{SW}. 
  Write
$\widehat{\rm LT}_{\infty}^\varpi$ and  $\widehat{\cal M}_{\infty}^\varpi$
(resp.~$\wh{\rm LT}_{\infty,\breve{K}}^\varpi$ and  $\wh{\cal M}_{\infty,\breve{K}}^\varpi$, etc.) for the  quotients of 
$\widehat{\rm LT}_{\infty}$ and $\widehat{\cal M}_{\infty}$
(resp.~$\wh{\rm LT}_{\infty,\breve{K}}$ and $\wh{\cal M}_{\infty,\breve{K}}$, etc.) 
     by the action of  $\varpi$ seen as an  element of the center of  $G$ (or, what amounts to the same,  of  $\check G$). 
     The Weil descent data on  $\wh{\rm LT}_{\infty,\breve{K}}^\varpi$ and $\wh{\cal M}_{\infty,\breve{K}}^\varpi$ are effective. We denote the resulting towers over \(K\) by \((M^\varpi_{n,K})_n\) and \((\mathrm{LT}^\varpi_{n,K})_n\).  We have duality isomorphisms 
     $$\widehat{\rm LT}_{\infty,{K}}^\varpi\simeq \widehat{\cal M}_{\infty,{K}}^\varpi,\quad \widehat{\rm LT}_{\infty}^\varpi\simeq \widehat{\cal M}_{\infty}^\varpi,
     $$
the latter being $\sg_K$-equivariant. 

 Set 
$$G_n:=\begin{cases}\mathbb{GL}_{d+1}(\O_K)&{\text{if $n=0$,}}\\
1+\varpi^n{\mathbb M}_{d+1}(\O_K)&{\text{if $n\geq 1$,}}\end{cases}
\quad
\check G_n:=\begin{cases}\O_D^\times&{\text{if $n=0$,}}\\
1+\varpi_D^n\O_D&{\text{if $n\geq 1$.}}\end{cases}$$ 
Let $n,m\geq 0$. We have the following diagram of dual towers
$$
\xymatrix{
& T\ar[dl]^-{{\check{G}_m}}_{\pi}\ar@(u,r)[]^{G_n\times \check{G}_m} \ar[dr]_{G_n}^{\check{\pi}}\\
{\cal M}_{m,{K}}^\varpi\ar@(l,u)[]^{{G_n}} & &  {\rm LT}_{n,{K}}^\varpi\ar@(u,r)[]^{{{\check{G}_m}}}, 
 }
$$
where we wrote  $T=\widehat{\rm LT}_{\infty,{K}}^\varpi\simeq \widehat{\cal M}_{\infty,{K}}^\varpi$.

\subsubsection{Admissibility}  We pass now to the  flip-flopping.  We start with some preparation. We have the following finiteness property of compactly supported de Rham and Hyodo-Kato cohomologies. 
\begin{theorem}\label{adm1}Let $i\geq 0$. 
The $G$-representations $H^i_{\dr,c}({\cal M}_{\infty}^\varpi)$ and $H^i_{\hk,c}({\cal M}_{\infty}^\varpi)$ over $C$ and $C^{\nr}$, respectively,  are admissible.
\end{theorem}
\begin{proof}
We write $\mathfrak{g}_{K}$, $\check{\mathfrak{g}}_{K}$ for the $K$-Lie algebras of $G$ and $\check{G}$, respectively;  $\mathfrak{g}$, $\check{\mathfrak{g}}$ for their $\Q_p$-versions. 
We claim that, for $i\geq 0$, we have
\begin{equation}\label{niedziela1}
\dim_{C^{\nr}} H^i(\mathfrak{g},C^{\nr})=\dim_{C^{\nr}} H^i(\check{\mathfrak{g}},C^{\nr})
\end{equation}
and, moreover, that the actions of $G, \check{G}$ on the respective Lie algebra cohomologies in \eqref{niedziela1} are trivial. The first claim follows from the fact that we have an isomorphism of Lie algebras
$$
\mathfrak{g}_K\otimes_KL\stackrel{\sim}{\to}\mathfrak{g}_L\simeq \check{\mathfrak{g}}_L\stackrel{\sim}{\leftarrow} \check{\mathfrak{g}}_K\otimes_KL,
$$
for a finite field extension $L/K$, and it suffices to show \eqref{niedziela1} after replacing $C^{\nr}$ by $C$. For the second claim we have the following result:
\begin{lemma}\label{luminy1}
Let $\mathbb{G}$ be a connected algebraic group over $\Q_p$ and let  $\mathfrak{g}={\rm Lie}(\mathbb{G}(\Q_p))$. Then $\mathbb{G}(\Q_p)$ acts trivially on $H^*(\mathfrak{g}, \Q_p)$. 
\end{lemma}
\begin{proof}Note that  $H^*(\mathfrak{g}, \Q_p)$ is finite dimensional and smooth, so fixed by an open subgroup $H$ of $\mathbb{G}(\Q_p)$. Now, $H$ is Zariski dense in $\mathbb{G}$, since it is open in 
$\mathbb{G}(\Q_p)$ and $\mathbb{G}(\Q_p)$ is Zariski dense in $\mathbb{G}$ (see \cite[Cor. 18.3]{Bor}). But the action of $\mathbb{G}$ on 
$H^*(\mathfrak{g}, \Q_p)$ is clearly algebraic (use the explicit Chevalley-Eilenberg complex and the fact that $\mathbb{G}$ acts algebraically on its Lie algebra), so $\mathbb{G}(\Q_p)$ acts trivially on $H^*(\mathfrak{g}, \Q_p)$, as wanted.
\end{proof}

   Let $X_{\infty}:={\cal M}_{\infty,\breve{K}}^\varpi, \check{X}_{\infty}:={\rm LT}_{\infty, \breve{K}}^\varpi$.  The above implies  that we can use  Theorem \ref{dept1} (with the roles of $G$ and $\check{G}$ interchanged) with $H=G_0, \check{H}=\check{G}_0$. We get  a $G_0\times\check{G}_0$-equivariant isomorphism  in $\breve{K}_{\Box}$
$$
H^i_{\dr,c}(X_{\infty})\simeq H^i_{\dr,c}(\check{X}_{\infty}).
$$
Hence isomorphisms in $\breve{K}_{\Box}$
$$
H^i_{\dr,c}(X_{\infty})^{G_n}\simeq H^i_{\dr,c}(\check{X}_{\infty})^{G_n}\simeq H^i_{\dr,c}({\rm LT}_{n,\breve{K}}^\varpi).
$$
Our theorem now follows from the following lemma:
\begin{lemma} The cohomology group $H^i_{\dr,c}({\rm LT}_{n,\breve{K}}^\varpi)$ is of finite rank over $\breve{K}$. 
\end{lemma}
\begin{proof} It suffices to treat one fixed-height piece, since
$LT_{n,\breve K}^{\varpi}$ is a finite disjoint union
of isomorphic such pieces. Denote by  $\check{X}$ such a piece. Since $\check{X}$ is Stein, by Poincar\'e duality (see \cite{Chi90})  it suffices to show that
$H^i_{\dr}(\check{X})$ has finite rank. We recall that, by \cite[III.4.1]{HT}, there exists a Shimura variety\footnote{The relevant Shimura variety is described in \cite{HT}; see also \cite[2.4]{Yosh}.} over $\breve{K}$ with a proper flat integral model $\sx$ over $\so_{\breve{K}}$ such that the tube  $U$ in the rigid analytic variety $\widehat{\sx}_{\breve{K}}$ of any supersingular closed point $Z$ is isomorphic to $\check{X}$. 
We note that tubes of locally closed subsets of $\sx_0$ are admissible open of $\widehat{\sx}_{\breve{K}}$. (For more details concerning tubes see \cite[Sec. 5.1.]{Ay}, \cite[Sec. 2.1.2]{LSt}.)
We want to show that the de Rham cohomology of the tube $U$ is finite dimensional.  But $U=\widehat{\sx}_{\breve{K}}\setminus V$, where $V$ is the tube of $\sx_0\setminus Z$. 
Since $\sx_0\setminus Z$ is quasi-compact open, the tube $V$ is quasi-compact.  Hence we can use  \cite[Th. 3.6]{GK} to finish the argument. 
\end{proof}
\begin{remark} The above lemma is well-known though we did not find a reference in the literature. The presented argument is a version of an argument communicated to us by Laurent Fargues.
\end{remark}

   Concerning Hyodo-Kato cohomology, we have a Hyodo-Kato isomorphism in $C^{\nr}_{\Box}$
$$
\iota_{\hk}:\quad H^i_{\hk,c}(\sm^{\varpi}_{\infty})\otimes^{\LL_{\Box}}_{C^{\nr}}C\stackrel{\sim}{\to}H^i_{\dr,c}(\sm^{\varpi}_{\infty}).
$$
It is $G_0$-equivariant yielding, for each $n$, an isomorphism in $C^{\nr}_{\Box}$
$$
\iota_{\hk,n}:\quad (H^i_{\hk,c}(\sm^{\varpi}_{\infty})^{G_n})\otimes^{\LL_{\Box}}_{C^{\nr}}C\stackrel{\sim}{\to}(H^i_{\dr,c}(\sm^{\varpi}_{\infty}))^{G_n}.
$$
Hence,  the wanted admissibility of $H^i_{\hk,c}(\sm^{\varpi}_{\infty})$ follows  from the one of $H^i_{\dr,c}(\sm^{\varpi}_{\infty})$ proved above. 
\end{proof}

\subsubsection{Flip-flopping} We are now ready to prove the main flip-flopping result of this section. 
\begin{theorem}\label{adm11}Let $i\geq 0$. 
 There are 
 $G\times \check{G}$-equivariant  isomorphisms in $\breve{K}_{\Box}$ and $(\phi,N,\sg_{{K}})$-modules over $C^{\nr}$, respectively:
 \begin{equation}\label{niedziela2}
 H^i_{\dr,c}({\cal M}_{\infty,\breve{K}}^\varpi)\simeq H^i_{\dr,c}({\rm LT}_{\infty,\breve{K}}^\varpi),\quad  H^i_{\hk,c}({\cal M}_{\infty}^\varpi)\simeq H^i_{\hk,c}({\rm LT}_{\infty}^\varpi).
 \end{equation}
\end{theorem}
\begin{proof}  Let $X_{\infty}:={\rm LT}_{\infty,\breve{K}}^\varpi$ and $\check{X}_{\infty}:= {\cal M}_{\infty,\breve{K}}^\varpi$. 

(1)  {\em De Rham cohomology.} We will first deal with de Rham cohomology by passing to cohomology groups in \eqref{paris12}. For \(i=0\), the claim follows by taking degree-zero cohomology in \eqref{paris12}.
 For the inductive step, we assume  the isomorphism 
 \eqref{niedziela2}   for $\leq  i$ and we want to prove it for $i+1$.  We write $\mathfrak{g}$, $\check{\mathfrak{g}}$ for the $\Q_p$-Lie algebras of $G$ and $\check{G}$, respectively. From \eqref{paris12} we get an isomorphism in $\sd(\breve{K}_{\Box})$
 \begin{equation}\label{ias22}
H^{i+1}(\R\Gamma(\check{\mathfrak{g}}, \breve{K})\otimes_{\breve{K}}^{\LL_{\Box}}\R\Gamma_{\dr,c}(X_{\infty}))  \simeq H^{i+1}(\R\Gamma(\mathfrak{g},\breve{K})\otimes^{\LL_{\Box}}_{\breve{K}}\R\Gamma_{\dr,c}(\check{X}_{\infty})).
 \end{equation}
A priori, this isomorphism is just $\breve{K}[G_0\times \check{G}_0]$-equivariant but we have: 
\begin{lemma}
The isomorphism \eqref{ias22} is  $\breve{K}[G\times \check{G}]$-equivariant.
\end{lemma}
\begin{proof}  To see this recall that the isomorphism \eqref{ias22} is obtained 
from  a $G_0\times\check{G}_0$-equivariant quasi-isomorphism in $\sd(\breve{K}_{\Box})$
\begin{align}\label{evening1}
 \colim_{\check{H}}\colim_HR\Gamma(\check{H},\R\Gamma_{\dr,c}(X_{H}))\simeq 
 \colim_{{H}}\colim_{\check{H}}\R\Gamma({H},\R\Gamma_{\dr,c}(\check{X}_{\check{H}})),
\end{align}
where $H, \check{H}$ run through the normal compact open subgroups of $G_0$ and $\check{G}_0$, respectively. We note that we may equivalently take these colimits over all compact open subgroups of \(G_0\) and \(\check G_0\), since the normal ones form cofinal systems. This quasi-isomorphism   is induced by the natural 
 quasi-isomorphisms $\R\Gamma(\check{H},\R\Gamma_{\dr,c}(X_{H}))\simeq \R\Gamma({H},\R\Gamma_{\dr,c}(\check{X}_{\check{H}}))$ from \eqref{berkeley2c}.

  Let $g\in G$ and choose  ${H}$ small enough so that $g^{-1}Hg\subset G_0$. We claim that, for every $\check{H}\subset \check{G}_0$ as in \eqref{evening1}, we have a commutative diagram
\begin{equation}\label{evening2}
\xymatrix{
 \R\Gamma(g^{-1}{H}g,\R\Gamma_{\dr,c}(\check{X}_{\check{H}}))\ar[d]^{\wr}_{\alpha_1}\ar[d]^{\wr}\ar[r]^-{g^*} &\R\Gamma({H},\R\Gamma_{\dr,c}(\check{X}_{\check{H}}))\ar[d]^{\wr}_{\alpha_2}\\
\R\Gamma(\check{H},\R\Gamma_{\dr,c}(X_{g^{-1}{H}g}))\ar[r]^-{g^*} &\R\Gamma(\check{H},\R\Gamma_{\dr,c}({X}_{{H}})),
}
\end{equation}
where the quasi-isomorphisms $\alpha_1, \alpha_2$ come from Corollary \ref{bonn15}. 
But this follows from the naturality of the maps $\alpha_1, \alpha_2$  with respect to the map $g$ acting on the two torsors data: 
$$
\xymatrix{
& T\ar[dl]^-{{H}}_{\pi}\ar@(u,r)[]^{H\times\check{H}} \ar[dr]_{\check{H}}^{\check{\pi}}\ar[rrrr]^{g} & &  &  & T\ar[dl]^-{{g^{-1}Hg}}_{\pi}\ar@(u,r)[]^{H\times\check{H}} \ar[dr]_{\check{H}}^{\check{\pi}}\\
X_H\ar@(l,u)[]^{\check{H}}  \ar@/_10pt/[rrrr]_{g} & &   \check{X}_{\check{H}}\ar@(u,r)[]^{{{H}}}   \ar@/_10pt/[rrrr]_{g} & & X_{g^{-1}Hg}\ar@(l,u)[]^{\check{H}} & &   \check{X}_{\check{H}}\ar@(u,r)[]^{{{H}}} 
 }
$$

  Now, taking colimit of the diagram \eqref{evening2} over $H$ and $\check{H}$ we get an action of $g$ on the isomorphism \eqref{ias22}, as wanted. A similar argument gives us an action of $\check{G}$. 
\end{proof}

  By K\"unneth formula, the isomorphism \eqref{ias22}  yields a $\breve{K}[G\times \check{G}]$-equivariant
   isomorphism in $\sd(\breve{K}_{\Box})$
   \begin{align}\label{kunneth11}
     & H^0(\check{\mathfrak{g}},\breve{K})\otimes^{\Box}_{\breve{K}}H^{i+1}_{\dr,c}(X_{\infty})  \oplus  H^1(\check{\mathfrak{g}},\breve{K})\otimes^{\Box}_{\breve{K}}H^{i}_{\dr,c}(X_{\infty})\oplus\cdots\oplus   H^{i+1}(\check{\mathfrak{g}},\breve{K})\otimes^{\Box}_{\breve{K}}H^{0}_{\dr,c}(X_{\infty})\\
      & \simeq  H^0(\mathfrak{g},\breve{K})\otimes^{\Box}_{\breve{K}}H^{i+1}_{\dr,c}(\check{X}_{\infty})\oplus  H^1(\mathfrak{g},\breve{K})\otimes^{\Box}_{\breve{K}}H^{i}_{\dr,c}(\check{X}_{\infty})\oplus\cdots\oplus   H^{i+1}(\mathfrak{g},\breve{K})\otimes^{\Box}_{\breve{K}}H^{0}_{\dr,c}(\check{X}_{\infty})\notag
   \end{align}
  Since   the groups $H^j(\check{\mathfrak{g}},\breve{K})$, $H^j({\mathfrak{g}},\breve{K})$, for $j\geq 0$,   are of finite rank over $\breve{K}$,  $H^0(\check{\mathfrak{g}},\breve{K})=H^0(\mathfrak{g},\breve{K})=\breve{K}$, and 
we have the two properties of Lie algebra cohomologies shown in the proof of Theorem  \ref{adm1},  if we want to prove the  claim of the theorem for $i+1$, by the inductive assumption, we find ourselves in the situation, where  we have an isomorphism of admissible representations of $G\times\check{G}$ over $\breve{K}$ (by Theorem \ref{adm1}): 
  $$
  \pi_1\oplus \pi \simeq \pi_2\oplus \pi
  $$
  and we want to conclude that $\pi_1\simeq \pi_2$, as $G\times\check{G}$-representations.  But this follows from Proposition \ref{rep1HK}.

   (2)  {\em Hyodo-Kato cohomology.} We will now pass to Hyodo-Kato cohomology.  Here we use the descended towers \(X_\infty=\mathrm{LT}^\varpi_{\infty,K}\) and \(\check X_\infty=M^\varpi_{\infty,K}\), and apply \eqref{paris12HK} over \(K\). We will argue by induction on $i$ as in the de Rham case starting with $i=0$, where the claim of the theorem is clear by passing to cohomology in \eqref{paris12HK}. For the inductive step, we assume  the isomorphism 
 \eqref{niedziela2}   for $\leq  i$ and we want to prove it for $i+1$.
 
  From \eqref{paris12HK} we get an isomorphism in $\sd_{\phi,N,\sg_{K}}(C^{\nr}_{\Box})$
$$
H^{i+1}(\R\Gamma(\check{\mathfrak{g}}, C^{\nr})\otimes_{C^{\nr}}^{\LL_{\Box}}\R\Gamma_{\hk,c}(X_{\infty,C})) 
 \simeq H^{i+1}(\R\Gamma(\mathfrak{g},C^{\nr})\otimes^{\LL_{\Box}}_{C^{\nr}}\R\Gamma_{\hk,c}(\check{X}_{\infty,C})).
$$
A priori, this isomorphism is just $C^{\nr}[G_0\times \check{G}_0]$-equivariant but arguing as in the de Rham case we see that  it is in fact $C^{\nr}[G\times \check{G}]$-equivariant.
Using the K\"unneth formula, we get  a $G\times \check{G}$-equivariant
   isomorphism in $\sd_{\phi,N,\sg_{K}}(C^{\nr}_{\Box})$
   \begin{align*}
     & H^0(\check{\mathfrak{g}},C^{\nr})\otimes^{\Box}_{C^{\nr}}H^{i+1}_{\hk,c}(X_{\infty,C})  \oplus  H^1(\check{\mathfrak{g}},C^{\nr})\otimes^{\Box}_{C^{\nr}}H^{i}_{\hk,c}(X_{\infty,C})\oplus\cdots\oplus   H^{i+1}(\check{\mathfrak{g}},C^{\nr})\otimes^{\Box}_{C^{\nr}}H^{0}_{\hk,c}(X_{\infty,C})\\
      & \simeq  H^0(\mathfrak{g},C^{\nr})\otimes^{\Box}_{C^{\nr}}H^{i+1}_{\hk,c}(\check{X}_{\infty,C})\oplus  H^1(\mathfrak{g},C^{\nr})\otimes^{\Box}_{C^{\nr}}H^{i}_{\hk,c}(\check{X}_{\infty,C})\oplus\cdots\oplus   H^{i+1}(\mathfrak{g},C^{\nr})\otimes^{\Box}_{C^{\nr}}H^{0}_{\hk,c}(\check{X}_{\infty,C})\notag
   \end{align*}

  Since the groups $H^j(\check{\mathfrak{g}},C^{\nr})$, $H^j({\mathfrak{g}},C^{\nr})$, for $j\geq 0$ are of finite rank over $C^{\nr}$,  $H^0(\check{\mathfrak{g}},C^{\nr})=H^0(\mathfrak{g},C^{\nr})=C^{\nr}$, and 
we have the two properties of Lie algebra cohomologies shown in the proof of Theorem \ref{adm1},  if we want to prove the  claim of the theorem for $i+1$, by the inductive assumption, we find ourselves in the situation, where  we have an isomorphism of admissible representations of $G\times\check{G}$ on $(\phi,N,\sg_{K})$-modules over 
$C^{\nr}$ (by Theorem \ref{adm1}): 
  $$
  \pi_1\oplus \pi \simeq \pi_2\oplus \pi
  $$
  and we want to conclude that $\pi_1\simeq \pi_2$, as $G\times\check{G}$-representations on $(\phi,N,\sg_{{K}})$-modules over 
$C^{\nr}$.  But this follows from the representation-theoretic Proposition  \ref{rep1HK}.
\end{proof}

     The following fact is an immediate corollary of Theorem \ref{adm11} (obtained by taking $G_n\times\check{G}_m$-invariants of the isomorphisms \eqref{niedziela2}):
 \begin{corollary}
 For any $m,n\geq 0$, we have  isomorphisms in $\breve{K}_{\Box}$ and  solid $(\phi,N,\sg_{\breve{K}})$-modules over $C^{\nr}$
 $$
 H^i_{\dr,c}({\cal M}_{m,\breve{K}}^\varpi)^{G_n}\simeq H^i_{\dr,c}({\rm LT}_{n,\breve{K}}^\varpi)^{\check{G}_m},\quad  
 H^i_{\hk,c}({\cal M}_{m}^\varpi)^{G_n}\simeq H^i_{\hk,c}({\rm LT}_{n}^\varpi)^{\check{G}_m}.
 $$
 \end{corollary}    
 \subsection{Flip-flopping for dual towers of local Shimura varieties} The above computations for the Drinfeld and Lubin-Tate towers generalize easily to all basic local Shimura varieties. 
 
  \subsubsection{Local Shimura varieties}  We briefly review here a few facts about dual towers of (basic) local Shimura varieties. See \cite{Fa}, \cite{RV}, \cite{SW}, \cite{FS} for details.  
  
  Let $(G,[b],\{\mu\})$ be a local Shimura datum over $\Q_p$ with a basic $b$: $G$ is a connected reductive group over $\Q_p$, $\mu:\mathbb{G}_{m, \overline{\Q}_p}\to G_{\overline{\Q}_p}$ is a (conjugacy class of a) minuscule cocharacter, and $b\in B(G,\mu)$ is the unique basic neutral acceptable element. Let
$E = E(G,\{\mu\})$ be the reflex field.

  Since $b$ is basic, its $\sigma$-centralizer
$\check G:=J_b$ is an inner form of $G$.
Choose a representative $b\in G(L)$ of $[b]$, where
$L=\breve{\mathbf Q}_p$. Under the standard identification
$\check G_L\simeq G_L$, set $  \check b:=b^{-1}\in\check G(L). $
Let $\{\check\mu\}$ be the conjugacy class of cocharacters of
$\check G_{\overline{\mathbf Q}_p}$ corresponding to
$\{\mu^{-1}\}$ under the inner identification
$\check G_{\overline{\mathbf Q}_p}\simeq
 G_{\overline{\mathbf Q}_p}$.
The triple $
  (\check G,[\check b],\{\check\mu\}) $
is the dual basic local Shimura datum. It has the same reflex
field $E$, and $J_{\check b}$ is naturally isomorphic to $G$.

   Let $K_0 \subset G(\Q_p)$ be a maximal compact subgroup and let $K \triangleleft K_0$ be an open normal subgroup. Then the finite-level local Shimura variety $\mathrm{Sh}_K(G,b,\mu)$ is a smooth rigid-analytic space (partially proper) over $\breve E$, equipped with {natural actions} of
\begin{enumerate}
    \item $K_0 / K$ via Hecke correspondences,
    \item the group $\check{G}(\Q_p)$ via quasi-isogenies, and
    \item $\sg_{\breve{E}}:=\mathrm{Gal}(C/\breve E)$ after base change to $C$. 
\end{enumerate}
Symmetrically, the  finite-level space $\mathrm{Sh}_{\check{K}}(\check G,\check b,\check\mu)$, for an open  normal  subgroup $\check{K}$ of $\check{K}_0$ -- a maximal compact subgroup of $\check{G}$ -- carries the {action} of $ G(\Q_p)$ and of the quotient $\check{K}_0/\check{K}$.

    The collection of all levels forms the tower
\[
\mathrm{Sh}_\bullet(G,b,\mu) := \{\mathrm{Sh}_K(G,b,\mu)\}_K,
\]
with commuting actions of $G(\Q_p)$, $\check{G}(\Q_p)$, and $\sg_{\breve{E}}$. Its completed limit
\[
\widehat{\mathrm{Sh}}(G,b,\mu) := \varprojlim_K \mathrm{Sh}_K(G,b,\mu)
\]
makes sense as a diamond over ${\rm Spd}(\breve{E})$, endowed with 
 continuous actions of the groups $G(\Q_p)$, $\check{G}(\Q_p)$ and $\sg_{\breve{E}}$ (the Galois group acting after base-change to $C$). The dual completed space $\widehat{\mathrm{Sh}}(\check G,\check b,\check\mu)$ satisfies the same properties. 

By the duality theorem  (see \cite[Cor. 23.3.2]{SW1}), there is a canonical isomorphism of diamonds 
$$
\widehat{\mathrm{Sh}}(G,b,\mu) \;\xrightarrow{\sim}\; \widehat{\mathrm{Sh}}(\check G,\check b,\check\mu)
$$
which is equivariant with respect to the actions of  $G(\Q_p), \check{G}(\Q_p)$,   and 
$\sg_{\breve{E}}$ (the Galois group acting after base-change to $C$).

 \subsubsection{Flip-flopping}  The main theorem of this section is the following flip-flopping result: 
\begin{theorem}[Local Shimura varieties flip-flopping]\label{mainSh}
For every $i\geq0$, there are
$G(\mathbf Q_p)\times\check G(\mathbf Q_p)$-equivariant
isomorphisms
\[
 H^i_{\mathrm{dR},c}
   (\operatorname{Sh}(G,b,\mu)_\infty)
 \simeq
 H^i_{\mathrm{dR},c}
   (\operatorname{Sh}(\check G,\check b,\check\mu)_\infty)
\]
in solid $\breve E$-modules, and
\[
 H^i_{\mathrm{HK},c}
   (\operatorname{Sh}(G,b,\mu)_{\infty,C})
 \simeq
 H^i_{\mathrm{HK},c}
   (\operatorname{Sh}(\check G,\check b,\check\mu)_{\infty,C})
\]
in solid $(\varphi,N,\sg_{\breve E})$-modules over $C^{\nr}$.
\end{theorem}

\begin{proof}
Put $F=\breve E$, $\Gamma=\sg_F$, and
$\mathcal G=G(\mathbf Q_p)\times\check G(\mathbf Q_p)$.
By the duality theorem recalled above, the two towers have
an identified infinite-level diamond $T$ over $\operatorname{Spd}(F)$.
For sufficiently small compact open subgroups
$K\subset G(\mathbf Q_p)$ and
$\check K\subset\check G(\mathbf Q_p)$, write
\[
 X_K=T/K=\operatorname{Sh}_K(G,b,\mu),\qquad
 \check X_{\check K}=T/\check K
   =\operatorname{Sh}_{\check K}
       (\check G,\check b,\check\mu).
\]

\smallskip
 ($\bullet$) {\em The indexing category.}
Consider a triple $
 c=(S,K,\check K), $
where $S\subset T$ is a $K\times\check K$-stable open
contained in a quasi-compact open of $T$, and, for every
open subgroup $K'\subset K$ and $\check K'\subset\check K$,
the rigid spaces $S/K'$ and $S/\check K'$ are partially
proper and have finite-dimensional compactly supported
de Rham cohomology over $F$.
For each fixed pair $(K,\check K)$, the common exhaustion
provides an increasing covering of $T$ by such data.
In particular, every quasi-compact subset of $T$ is
contained in one of them.

Let $\mathcal I_0$ be the poset of such data, with
\[
 (S,K,\check K)\leq(S',K',\check K')
 \quad\Longleftrightarrow\quad
 S\subset S',\quad K'\subset K,\quad
 \check K'\subset\check K.
\]
This poset is filtered. 
For $h=(g,\check g)\in\mathcal G$, set $
 hc=(hS,gKg^{-1},\check g\check K\check g^{-1}).$
Translation  acts on
$\mathcal I_0$. Let $\mathcal I$ be the category with the
same objects as $\mathcal I_0$ and with
\[
 \operatorname{Hom}_{\mathcal I}(c,c')
   =\{h\in\mathcal G\mid hc\leq c'\}.
\]
Composition is multiplication in $\mathcal G$.

\smallskip
 ($\bullet$) {\em The  de Rham diagrams.}
For $r\geq0$, define
\[
 \mathcal V^r(c)=H^r_{\mathrm{dR},c}(S/K)^{\check K},
 \qquad
 \check{\mathcal V}^r(c)
   =H^r_{\mathrm{dR},c}(S/\check K)^K.
\]
These are finite-dimensional $F$-vector spaces.
For a refinement $c\leq c'$, the map on $\mathcal V^r$
is the composite of finite \'etale pullback,
extension by zero, and inclusion of invariants:
\[
 H^r_{\mathrm{dR},c}(S/K)^{\check K}
 \longrightarrow
 H^r_{\mathrm{dR},c}(S/K')^{\check K'}
 \longrightarrow
 H^r_{\mathrm{dR},c}(S'/K')^{\check K'}.
\]
Define the map on $\check{\mathcal V}^r$ similarly.
These
maps respect composition.
Translations act by transport along the corresponding
isomorphisms of quotient spaces, or equivalently by
pullback along their inverses. We therefore obtain
functors
\[
 \mathcal V^r,\check{\mathcal V}^r:
 \mathcal I\longrightarrow\operatorname{Vect}^{\mathrm{fd}}_F.
\]

Let $\mathfrak g$ and $\check{\mathfrak g}$ be the
$\mathbf Q_p$-Lie algebras of the two groups, and put
\[
 b_a=\dim_{\mathbf Q_p}H^a(\mathfrak g,\mathbf Q_p)
     =\dim_{\mathbf Q_p}
         H^a(\check{\mathfrak g},\mathbf Q_p).
\]
The equality follows because $G$ and $\check G$ are
inner forms; in particular, $b_0=1$.
By Lemma~\ref{luminy1}, the two groups act trivially on
their respective Lie algebra cohomology groups.
Fix bases of these groups over $\mathbf Q_p$.
Apply Corollary~\ref{dom1} to the dual towers obtained
from $S$, take degree-$i$ cohomology, and then take
ordinary $K\times\check K$-invariants.
Finite \'etale descent identifies
$$
 \left(\operatorname*{colim}_{K'\subset K}
      H^i_{\mathrm{dR},c}(S/K')\right)^{K\times\check K}
   \simeq H^i_{\mathrm{dR},c}(S/K)^{\check K},
$$
and similarly on the dual side. Hence, 
the K\"unneth isomorphism yields
\begin{equation}\label{LS-stabilized-diagrams}
 \bigoplus_{a=0}^i(\mathcal V^{i-a})^{\oplus b_a}
 \simeq
 \bigoplus_{a=0}^i(\check{\mathcal V}^{i-a})^{\oplus b_a}
\end{equation}
as diagrams on $\mathcal I$.

 It is important to spell out the naturality needed for this claim.
For refinements, it follows from compatibility of the
compactly supported flip-flop with extension by zero
and finite \'etale pullback, together with restriction
in group cohomology. For translations, use the
isomorphism of torsor data induced by $h$, exactly as
in the conjugation argument in the proof of Lemma~7.9,
now with $S$ and $hS$ in place of $T$.

   Now, we construct isomorphisms
$\mathcal V^i\simeq\check{\mathcal V}^i$ by induction on $i$.
For $i=0$, this is \eqref{LS-stabilized-diagrams}.
Once such isomorphisms have been chosen in degrees
less than $i$, that equation becomes
\[
 \mathcal V^i\oplus\mathcal Q_i
   \simeq\check{\mathcal V}^i\oplus\mathcal Q_i,
 \qquad
 \mathcal Q_i=
   \bigoplus_{a=1}^i(\mathcal V^{i-a})^{\oplus b_a}.
\]
Lemma~\ref{berkeley1W}, applied to
$\operatorname{Vect}^{\mathrm{fd}}_F$-valued diagrams on
$\mathcal I$, gives the required diagram isomorphism.

 Finally,  fix compact open subgroups $K_0$ and $\check K_0$, and
choose a common increasing exhaustion $(S_n)$ for this
pair. The triples $(S_n,K,\check K)$, with $K\subset K_0$
and $\check K\subset\check K_0$, form a cofinal subsystem
of $\mathcal I_0$. By smoothness and the compact-support
colimit formula, we obtain
   \[
 \begin{split}
 \operatorname*{colim}_{\mathcal I_0}\mathcal V^i
 &\simeq
   \operatorname*{colim}_{n,K,\check K}
        H^i_{\mathrm{dR},c}(S_n/K)^{\check K}\simeq
   \operatorname*{colim}_{n,K}
        H^i_{\mathrm{dR},c}(S_n/K)\\
         & \simeq
   H^i_{\mathrm{dR},c}
      (\operatorname{Sh}(G,b,\mu)_\infty),
 \end{split}
\]
and similarly  for $\check{\mathcal V}^i$. Since we have  $ \mathcal V^i\simeq \check{\mathcal V}^i$, this yields an isomorphism
   $$
     H^i_{\mathrm{dR},c}
      (\operatorname{Sh}(G,b,\mu)_\infty)  \simeq H^i_{\mathrm{dR},c}(\operatorname{Sh}(\check G,\check b,\check \mu)_\infty),
   $$
   as wanted. 
This isomorphism is $\mathcal G$-equivariant, since 
 the diagram isomorphisms also commute with the
translation arrows in $\mathcal I$.

\smallskip 
($\bullet$) {\em Hyodo--Kato cohomology.}
On the same indexing category define
\[
 \mathcal W^i(c)
   =H^i_{\mathrm{HK},c}((S/K)_C)^{\check K},
 \qquad
 \check{\mathcal W}^i(c)
   =H^i_{\mathrm{HK},c}((S/\check K)_C)^K.
\]
These take values in the category $\mathcal A_{\mathrm{HK}}$
of finite-rank $(\varphi,N,\Gamma)$-modules over $C^{\nr}$,
with invertible semilinear Frobenius and smooth
$\Gamma$-action. As explained in the general Hyodo--Kato
case of Theorem~\ref{dept1}, this is an essentially small
additive $\Q_p$-linear category with
finite-dimensional Hom spaces.
All support, level, and translation maps preserve
these structures.

Corollary~\ref{paris30}, with the same naturality
argument as above, yields
\[
 \bigoplus_{a=0}^i(\mathcal W^{i-a})^{\oplus b_a}
 \simeq
 \bigoplus_{a=0}^i(\check{\mathcal W}^{i-a})^{\oplus b_a}
\]
in $\operatorname{Fun}(\mathcal I,\mathcal A_{\mathrm{HK}})$.
The rational bases above are compatible with Frobenius
and Galois: the Lie algebra cohomology factors are
copies of the unit isocrystal, with zero monodromy.
Apply Lemma~\ref{berkeley1W} in this coefficient category
and repeat the induction to obtain
$\mathcal W^i\simeq\check{\mathcal W}^i$.
Taking the same support-and-level colimits as above yields the
required $\mathcal G$-equivariant isomorphism, preserving
$\varphi$, $N$, and $\Gamma=\sg_{\breve E}$.
\end{proof}

\end{document}